\documentclass[11pt]{article}
\usepackage{amsfonts,amsmath,amssymb,mathrsfs,bm,amsthm}
\usepackage[dvipdfmx]{graphicx}%
\usepackage{caption,epic,eepic}
\usepackage{color}
\usepackage{multicol,wrapfig}
\DeclareFontFamily{U}{mathc}{}
\DeclareFontShape{U}{mathc}{m}{it}{<->s*[1.03] mathc10}{}
\DeclareMathAlphabet{\mathcal}{U}{mathc}{m}{it}
\usepackage[scr=rsfso,calscaled=1.2]{mathalfa}
\newcommand{\noi}{\noindent}
\newcommand{\halmos}{\rule{1ex}{1.4ex}}
\def \qed {\nopagebreak{\hspace*{\fill}$\halmos$\medskip}}

\newtheorem{theorem}{Theorem}[section]
\newtheorem{proposition}[theorem]{Proposition}
\newtheorem{corollary}[theorem]{Corollary}
\newtheorem{conjecture}[theorem]{Conjecture}
\newtheorem{lemma}[theorem]{Lemma}

\newtheorem{thmx}{Theorem}

\theoremstyle{definition}
\newtheorem{remark}[theorem]{Remark}
\newtheorem{definition}[theorem]{Definition}

\newcommand{\bd}{\begin{definition}}
\newcommand{\ed}{\end{definition}}
\newcommand{\bt}{\begin{theorem}}
\newcommand{\et}{\end{theorem}}
\newcommand{\bl}{\begin{lemma}}
\newcommand{\el}{\end{lemma}}
\newcommand{\bp}{\begin{proposition}}
\newcommand{\ep}{\end{proposition}}
\newcommand{\bcor}{\begin{corollary}}
\newcommand{\ecor}{\end{corollary}}
\newcommand{\br}{\begin{remark}\rm}
\newcommand{\er}{\end{remark}}
\newcommand{\bcon}{\begin{conjecture}}
\newcommand{\econ}{\end{conjecture}}

\renewcommand{\theequation}{\thesection .\arabic{equation}}

\newcommand{\be}{\begin{equation}}
\newcommand{\ee}{\end{equation}}

\newcommand{\Bi}{{\cal B}}

\newcommand{\Li}{{\cal L}}

\newcommand{\Oi}{{\cal O}}

\newcommand{\Ti}{{\cal T}}

\newcommand{\R}{{\mathbb R}}
\newcommand{\N}{{\mathbb N}}
\newcommand{\Z}{{\mathbb Z}}

\renewcommand{\P}{{\mathbb P}}
\newcommand{\E}{{\mathbb E}}

\newcommand\bP{\ensuremath{\boldsymbol{\mathrm{P}}}}
\newcommand\bE{\ensuremath{\boldsymbol{\mathrm{E}}}}

\newcommand\eps{\epsilon}

\renewcommand{\tilde}{\widetilde}          % wider `tilde'
\newcommand{\relmiddle}[1]{\mathrel{}\middle#1\mathrel{}}

\newcommand{\ball}[2]{B(#1;#2)}

\newcommand{\vol}[1]{{\rm vol}(#1)}

\begin{document}

%numbering formulas within sections
\makeatletter\@addtoreset{equation}{section}
\makeatother\def\theequation{\thesection.\arabic{equation}}

\title{Geometry of the random walk range conditioned on survival among Bernoulli obstacles}
\author{Jian Ding$^{\,1}$ \and Ryoki Fukushima$^{\,2}$ \and Rongfeng Sun$^{\,3}$ \and Changji Xu$^{\,4}$}

\date{\today}

\maketitle

\begin{center}
\textit{Dedicated to the memory of Kazumasa Kuwada} 
\end{center}

\footnotetext[1]{Statistics Department, University of Pennsylvania. Email: dingjian@wharton.upenn.edu}

\footnotetext[2]{Research Institute for Mathematical Sciences, Kyoto University. Email: ryoki@kurims.kyoto-u.ac.jp}

\footnotetext[3]{Department of Mathematics, National University of Singapore. Email: matsr@nus.edu.sg}

\footnotetext[4]{Department of Statistics, University of Chicago. Email: changjixu@galton.uchicago.edu}

\begin{abstract}
We consider a discrete time simple symmetric random walk among Bernoulli obstacles on $\Z^d$, $d\geq 2$, where the walk is killed when it hits an obstacle. It is known that conditioned on survival up to time $N$, the random walk range is asymptotically contained in a ball of radius $\varrho_N=C N^{1/(d+2)}$ for any $d\geq 2$. For $d=2$, it is also known that the range asymptotically contains a ball of radius $(1-\epsilon)\varrho_N$ for any $\eps>0$, while the case $d\geq 3$ remains open. We complete the picture by showing that for any $d\geq 2$, the random walk range asymptotically contains a ball of radius $\varrho_N-\varrho_N^\epsilon$ for some $\epsilon \in (0,1)$. Furthermore, we show that its boundary is of size at most $\varrho_N^{d-1}(\log \varrho_N)^a$ for some $a>0$.
\end{abstract}

\vspace{.3cm}

\noi
{\it MSC 2000.} Primary: 60K37; Secondary: 60K35.

\noi
{\it Keywords.} Bernoulli obstacles, random walk range, Faber--Krahn inequality, annealed law.
\vspace{12pt}

\tableofcontents

\medskip

\noindent
%\blue{(The section titles are made free from numbered theorems. I think this makes it easier to get information quickly. I don't resist to use numbers in the subsection titles. By the way, shall we keep this table of contents? The initial purpose was to show \emph{you} how the paper is organized.)}

\section{Introduction}
\label{sec:intro}
Let $S:=(S_n)_{n\geq 0}$ be a discrete time simple symmetric random walk on $\Z^d$. We will use $\bP_x$ and $\bE_x$ to denote probability and expectation for $S$ with $S_0=x\in \Z^d$, and we will omit the subscript $x$ when $x=0$. Independently for each $x\in \Z^d$, an obstacle is placed at $x$ with probability $1-p$ for some fixed $p\in (0,1)$, which generates the so-called {\em Bernoulli obstacle} configuration and plays the role of a random environment.  Probability and expectation for the obstacles will be denoted by $\P$ and $\E$, respectively. Let $\Oi$ denote the set of sites occupied by obstacles. When there is no obstacle at the site $x\in \Z^d$, we will say $x$ is open. The random walk is killed at the moment it hits an obstacle (called hard obstacles), namely, at the stopping time
\be\label{tauO}
\tau_\Oi:=\min\{n\geq 0: S_n\in \Oi\}.
\ee
More generally, we will use $\tau_A$ to denote the first hitting time of a set $A\subset \Z^d$. {We will write $\bE[f(S)\colon A]=\bE[f(S)1_A]$ and $\E[g(\Oi)\colon B]=\E[g(\Oi)1_B]$.}

We are interested in $\P\otimes \bP((S, \Oi) \in \cdot \mid \tau_\Oi >N)$, the so-called {\em annealed law} of $(S, \Oi)$ conditioned on the random walk's survival up to time $N$. For simplicity, we will denote
\be
\mu_N( (S, \Oi)\in \cdot ) := \P\otimes \bP((S, \Oi) \in \cdot \mid \tau_\Oi>N).
\label{eq:def_mu}
\ee
In particular, we are interested in the law of the random walk range
\be
S_{[0,N]} := \{S_i: 0\leq i\leq N\}
\ee
under the conditioned measure $\mu_N$. It is worth noting that the marginal law of $\mu_N$ for the random walk admits a representation in terms of the range of the random walk:
\begin{equation}
 \mu_N( S\in \cdot )
=\frac{\bE\left[p^{|S_{[0,N]}|}\colon S \in \cdot \right]}{\bE\left[p^{|S_{[0,N]}|}\right]},\label{eq:mu-Gibbs}
\end{equation}
where $|S_{[0,N]}|$ denotes the cardinality of the set $S_{[0,N]}$. 

Let us review known results on this model. The first result dates back to Donsker--Varadhan's work~\cite{DV79} which determined the leading order asymptotics of the denominator in~\eqref{eq:mu-Gibbs}, which can be regarded as the ``partition function'' of a self-attracting polymer model. The main result of~\cite{DV79} reads as
\begin{equation}
\begin{split}
 \P\otimes \bP(\tau_\Oi>N)&=\bE\left[p^{|S_{[0,N]}|}\right]\\
 &=\exp\left\{-c(d,p)N^{\frac{d}{d+2}}(1+o(1))\right\},\\
\textrm{with }c(d,p):&=\frac{d+2}{2}\left(\log(1/p)\right)^{\frac{2}{d+2}}
 \left(
\frac{2\lambda_1}{d}
%{2\lambda_d}{d^2}
\right)^{\frac{d}{d+2}},
\end{split}
\label{eq:surv}
\end{equation}
where $\lambda_1$ is the principal Dirichlet eigenvalue of $-\frac{1}{2d}\Delta$ in the ball of unit volume in $\R^d$ centered at the origin. In fact, Donsker--Varadhan studied this problem in the continuum setting first in~\cite{DV75} as the asymptotics of the moment generating function of the Wiener sausage. This corresponds to a space-time continuum analogue of a random walk among Bernoulli obstacles, known as a {\em Brownian motion among Poissonian obstacles}, where each obstacle takes a fixed shape, say a ball, and the centers of the obstacles follow a homogeneous Poisson point proces on $\R^d$. This model has been studied extensively and most of the results can be found in the celebrated monograph by Sznitman~\cite{S98}. The core of the method of Sznitman, called \emph{the method of enlargement of obstacles}, is translated to the discrete setting in~\cite{A95,BAR00} and therefore most of the results in continuum setting can be converted to the discrete setting. For this reason, in this section we will not explicate in which setting a result has been proved.

The argument of Donsker--Varadhan indicates that the dominant contribution to the partition function comes from the strategy of finding a ball of optimal radius
\begin{equation}
 \label{eq:Rn}
\varrho_N:= \Big(\frac{2{\lambda_1}}{d\log(1/p)}\Big)^{\frac{1}{d+2}} N^{\frac{1}{d+2}},
\end{equation}
which is free of obstacles and the random walk is confined in that ball up to time $N$. 
%\blue{In fact, it is natural to expect that for any bounded set $U$ containing the origin, we have
%\begin{equation}
%\begin{split}
% \P\otimes \bP(\tau_\Oi>N)
%& \ge \P(\Oi\cap U=\emptyset) \bP(\tau_U>N)\\
%& \gtrsim \exp\left\{|U|\log(1/p)-N\lambda_U\right\},
%%{2\lambda_d}{d^2}
%\end{split} 
%\end{equation}
%where $\lambda_U$ stands for the principal Dirichlet eigenvalue of $-\frac{1}{2d}\Delta$ in $U$.}
It has been proved later that this is what happens under the annealed measure in~\cite{S91} and~\cite{B94} for $d=2$ and~\cite{P99} for $d\ge 3$:
\begin{thmx}[Confinement]\label{th:confine}
For any $d\geq 2$, there exists $\eps_1\in(0,1)$ and $\mathcal{x}_N\in \Z^d$ depending only on the obstacle configuration $\Oi$, such that $\mathcal{x}_N\in \ball{0}{\varrho_N}$, the ball of radius $\varrho_N$ centered at $0$, and
\be\label{confine1}
\lim_{N\to\infty} \mu_N \big( S_{[0,N]} \subset \ball{\mathcal{x}_N}{\varrho_N+\varrho_N^{\eps_1}} \big) = 1.
\ee
The law of $\varrho_N^{-1}\mathcal{x}_N$ converges to $\phi_{\ball{0}{1}}dx$ as $N\to\infty$, where $\phi_{\ball{0}{1}}$ is the $L^1$-normalized principal Dirichlet eigenfunction of $-\frac1{2d}\Delta$ in $\ball{0}{1}$.
Furthermore, for $d=2$ and for any $\epsilon\in(0,1)$,
\be\label{confine2}
\lim_{N\to\infty} \mu_N \big( \ball{\mathcal{x}_N}{{(1-\epsilon)}\varrho_N} \subset S_{[0,N]}\big) =1. 
\ee
\end{thmx}
\begin{remark}
The above formulation of confinement is in fact far more precise than what Donsker--Varadhan's argument suggests. Their argument is based on the large deviation principle for the empirical measure and thus it only indicates that the random walk spends most of the time in the ball. The interested reader can find a detailed explanation in~\cite[Section~2.5]{B02}.
\end{remark}
It remains open to show that \eqref{confine2} also holds for dimensions $d\geq 3$, that the random walk range covers a full ball with radius almost $\varrho_N$ (see~\cite[Conjecture 1.3]{B94}). Our first main result resolves this question.

\bt[Ball covering]\label{th:cover}
Let $d\geq 2$, and let $\varrho_N$ and $\mathcal{x}_N$ be as in~\eqref{eq:Rn} and Theorem~\ref{th:confine}, respectively. Then there exists $\eps_2\in (0,1)$, such that
\be\label{cover}
\lim_{N\to\infty} \mu_N 
\left( %\bigcup_{ z\in \ball{0}{\varrho_N}} 
\ball{\mathcal{x}_N}{\varrho_N - \varrho_N^{\eps_2}} \subset S_{[0,N]}  
\right)
%\big( \exists\, z\in \ball{0}{\varrho_N} \textrm{\upshape{ s.t. }} \ball{z}{\varrho_N - \varrho_N^{\eps_2}} \subset S_{[0,N]} \big) 
= 1.
\ee
\et
\begin{remark}
This theorem extends and refines~\eqref{confine2} for general $d\ge 2$. In fact, we will first prove the extension of~\eqref{confine2} to $d\ge 3$ as an intermediate step to the above refined result. The interested reader may jump to Section~\ref{S:Ticoverw} after reading Subsections~\ref{SS:switch} and~\ref{SS:coverw-outline}. 
\end{remark}
We proceed to the second main result of this paper, which is about the boundary of the range of the random walk under the annealed law. For any set $A\subset \Z^d$, we define its external %and internal 
boundary by
\begin{equation}
\partial A  := \{y\in \Z^d\setminus A\colon \Vert y-x\Vert=1 \mbox{ for some }x\in A\},
\label{eq:ex_boundary}
\end{equation}
where $\Vert \cdot\Vert$ denotes the Euclidean norm. 
Theorem~\ref{th:confine} and Theorem \ref{th:cover} together imply that, conditioned on survival up to time $N$, the rescaled boundary of the random walk range, ${\varrho_N^{-1}}\partial S_{[0,N]}$, converges in probability to a unit sphere as $N$ tends to infinity, and $\partial S_{[0,N]}$ fluctuates on a scale of at most $\varrho_N^\epsilon$ with $\eps=\max\{\epsilon_1,\epsilon_2\}\in(0,1)$. Identifying the precise scale of fluctuation is an extremely interesting, but also challenging question. The following theorem is a step in this direction, which bounds the size of $\partial S_{[0,N]}$. 

\bt[Boundary size]\label{th:boundary}
Let $d\geq 2$, and let $\varrho_N$ be defined as in \eqref{eq:Rn}. Then there exists $a>0$, such that
\be\label{boundary}
\lim_{N\to\infty} \mu_N \big(|\partial S_{[0,N]}| \leq \varrho_N^{d-1} (\log \varrho_N)^a\big) = 1.
\ee
\et

\begin{remark}
Our proofs of Theorems~\ref{th:cover} and~\ref{th:boundary} assume Theorem~\ref{th:confine} as an input. Strictly speaking, Theorem~\ref{th:confine} has only been proved in the continuum setting for $d\ge 3$ in~\cite{P99} using the method of enlargement of obstacles. We briefly explain how the argument can be adapted to the discrete setting in Appendix~\ref{app:Povel}. In fact, it is possible to prove Theorems~\ref{th:cover} and Theorem~\ref{th:confine} together. In the follow-up paper~\cite{DFSX19}, we present such an argument and further derive an extension of Theorem~\ref{th:confine} for a random walk with small bias conditioned to avoid Bernoulli obstacles. 
\end{remark}
\begin{remark}
After this paper was submitted for publication, Berestycki and Cerf announced an independent work~\cite{BC18}, where Theorem~\ref{th:cover} is proved by a different method. In addition,~\cite[Theorem~1.5]{BC18} proves a quantitative control on the random walk local time, which together with Theorem~\ref{th:cover} makes it possible to prove Theorem~\ref{th:confine} following the strategy of~\cite{B94}. For more detail, we refer the reader to the introduction of~\cite{BC18}. 
\end{remark}

Let us mention a few more related works. There is a general framework containing our setting called the parabolic Anderson model, where the obstacles are replaced by independent and identically distributed random potential $\{\omega(x)\}_{x\in\Z^d}$. One is interested in what happens under the measures
\begin{equation}
\frac{\E\otimes\bE\left[\exp\left\{\sum_{k=1}^N\omega(S_k)\right\}\colon (S,\omega)\in\cdot\right]}
{\E\otimes\bE\left[\exp\left\{\sum_{k=1}^N\omega(S_k)\right\}\right] } \textrm{ or }
\frac{\bE\left[\exp\left\{\sum_{k=1}^N\omega(S_k)\right\}\colon S\in\cdot\,\right]}
{\bE\left[\exp\left\{\sum_{k=1}^N\omega(S_k)\right\}\right]}.
\label{PAM-Gibbs}
\end{equation}
Formally our model corresponds to the case where $\omega$ takes value $0$ or $-\infty$. More generally if $\omega$ is non-positive, the above weighted measures can be interpreted as the law of random walk killed with probability $1-e^{\omega(x)}$ when it visits $x$, conditioned to survive until $N$. Thus $\omega$ plays the role of \emph{repulsive} impurities. On the other hand, positive $\omega$ corresponds to \emph{attractive} impurities. There are various localization results depending on the distribution of $\omega$. See, for example, the recent monograph by K\"onig~\cite{K16} for an up-to-date review. The first measure is known as the \emph{annealed law}, which is what we study in this paper, while the second measure in~\eqref{PAM-Gibbs} conditions on the random potential and is called the \emph{quenched law}.  

In the case of Bernoulli obstacles, it has been proved recently in~\cite{DX17,DX18} that under the quenched law, with high probability, within $o(N)$ steps, the random walk reaches a ball of volume $\Theta(\log N)$ which is almost free of obstacles, and then stays close to that ball till time $N$. However, because the size of the ball is much smaller than in the annealed setting, it is unlikely that the random walk will be confined to the ball of localization once it is reached, and it will be an interesting problem to characterize the random walk range in the quenched setting. We will address this question elsewhere. 

In what follows, we will use $c, c', C, C'$ to denote generic constants depending only on $d$ and $p$, whose values may change from line to line. For $G\subset\R^d$, we write $|G|$ for the number of points in $G\cap\Z^d$ and $\vol{G}$ for the Euclidean volume. A list of frequently used notation is compiled in Appendix~\ref{app:notation}. 

\section{Proof Outline}\label{S:prelim}
In this section, we list the main ingredients needed and outline the proof structure. An overview of how the rest of the paper is organized will be given at the end of the section. In what follows, many statements are supposed to hold with $\mu_N$-probability tending to one as $N$ tends to infinity, but we often make it implicit for brevity.
 
\subsection{Path and Environment Switching}
\label{SS:switch}
An argument that will be used repeatedly in our proof is path and environment switching. More precisely, if $A_1, A_2$ are two sets of random walk path configurations, and $E_1, E_2$ are two sets of obstacle configurations, then we can switch from $(A_1, E_1)$ to $(A_2, E_2)$ and bound
\be\label{eq:switch}
\begin{aligned}
&\mu_N((S, \Oi)\in (A_1, E_1)) \\
&\quad \leq  \frac{\P\otimes \bP(S\in A_1, \Oi\in E_1, \tau_\Oi>N)}{\P\otimes \bP(S \in A_2, \Oi \in E_2, \tau_\Oi>N)} \\
&\quad = \frac{\P(\Oi\in E_1)}{\P(\Oi\in E_2)} \cdot \frac{\E[\bP(S\in A_1,\tau_{\Oi}>N) \mid \Oi\in E_1]}{\E[\bP(S\in A_2,\tau_{\Oi}>N) \mid \Oi\in E_2]}.
\end{aligned}
\ee
The first factor determines the probability gain or cost in the environment when we switch from obstacle configurations in $E_1$ to $E_2$, while the second factor determines the gain or cost in the random walk when we switch from paths in $A_1$ to $A_2$. We will find suitable choices of $A_2$ and $E_2$ so that the gain in one factor will beat the cost in the other.

One way to bound the second factor in \eqref{eq:switch} is to find a coupling between two obstacle configurations $(\Oi_1, \Oi_2)$ with marginal distributions $\P(\cdot \mid\Oi\in E_1)$ and $\P(\cdot\mid \Oi\in E_2)$, and then bound $\bP(S\in A_1,\tau_{\Oi_1}>N)/\bP(S\in A_2,\tau_{\Oi_2}>N)$ uniformly with respect to $(\Oi_1, \Oi_2)$. This is possible because typically, $A_2$ and $E_2$ will be constructed by local modifications of paths in $A_1$ and obstacle configurations in $E_1$, respectively. 

This type of comparison argument is much more useful in the study of the conditional measure $\mu_N$ than a direct analysis, since we only have the crude leading order asymptotics on the partition function in~\eqref{eq:surv}. %\blue{We will use the environment and path switching exactly as in~\eqref{eq:switch} in Section~\ref{S:Ticoverw} where we prove~\eqref{confine2} for general $d\ge 2$ and some related results. In Section~\ref{S:Topen}, when we prove a control on the boundary size, we will apply the switching argument to the path and environment separately. However this is more apparent than substantial, and our guiding principle is always to switch both the paths and the environments.}

\subsection{Proof Outline for the Weaker Version of Ball Covering Theorem}
\label{SS:coverw-outline}
We first prove~\eqref{confine2} for general $d\ge 2$, which will play an important role in the proof of Theorems~\ref{th:cover} and~\ref{th:boundary}. The key step is to show that if $x\in\Oi$, then there is a positive fraction of closed sites in its neighborhood. 
\begin{lemma}
[Density of obstacles]\label{lem:odensity} For each $x\in \Z^d$, $l>0$, and $\delta>0$, let 
 \begin{equation}
 E_l^\delta(x):=\left\{x\in \Oi \textrm{ and }{\frac{|\Oi \cap \ball{x}{l}|}{|\ball{x}{l}|}} < \delta \right\}. 
\label{eq:E_l^delta}
 \end{equation}
 Then there exists $\delta>0$, such that
\begin{equation}
 \label{odensity}
 \mu_N\left(\bigcup_{x\in \ball{0}{2\varrho_N}} \, \bigcup_{{(\log N)^3}\leq l\leq {\varrho_N}} E_l^\delta(x)\right) \to 0 \textrm{ as } N\to\infty
\end{equation}
faster than any negative power of $N$.
\end{lemma}
The proof of Lemma 2.1 will be based on path and environment switching arguments. Roughly speaking, if for some $x\in \Oi$, $\ball{x}{l}$ contains few obstacles, then: either the walk visits $\ball{x}{l}$ many times, in which case we remove all the obstacles in $\ball{x}{l}$ and we will show that the gain in the random walk survival probability beats the loss from environment switching; or the walk visits $\ball{x}{l}$ rarely, in which case we switch to typical obstacle configurations in $\ball{x}{l}$ and force the walk to avoid $\ball{x}{l}$, and we will show that the gain in environment switching beats the loss in path switching. A more precise outline and the proof will be given in Section~\ref{S:Ticoverw}. 

Lemma~\ref{lem:odensity} implies that if there is an obstacle inside the ball $\ball{\mathcal{x}_N}{(1-\epsilon)\varrho_N}$, then the confinement ball $\ball{\mathcal{x}_N}{\varrho_N}$ contains order $\varrho_N^d$ obstacles. This makes it too difficult for the random walk to survive and we can then deduce that the ball $\ball{\mathcal{x}_N}{(1-\epsilon)\varrho_N}$ is free of obstacles. More precisely, we have the following result.
\begin{proposition}
\label{prop:Ticoverw} Let $d\geq 2$. Then for any $\eps>0$, we have
\begin{equation}
\label{eq:Ticoverw}
\lim_{N\to\infty} \mu_N
\left(\ball{\mathcal{x}_N}{(1-\eps)\varrho_N} \cap \Oi=\emptyset\right) = 1.
\end{equation}
\end{proposition}
Once we have this proposition, the covering property~\eqref{confine2} readily follows. Indeed, if the random walk avoids a site $x\in\ball{\mathcal{x}_N}{(1-\epsilon)\varrho_N}$ with positive probability uniformly in $N$, then we can close that site at little cost, which contradicts Proposition~\ref{prop:Ticoverw}. 

%In order to prove Lemma~\ref{lem:odensity}, we have to show that there is no closed site whose neighborhood contains few other obstacles. Suppose, to the contrary, that there exists $x\in\Oi$ that has few obstacles nearby. Roughly speaking, we apply  the path and environment switching~\eqref{eq:switch} to the following two events for the random walk separately: 
%\begin{enumerate}
%\item First, consider the set of random walk paths that come close to $x$ many times. In this case, we can gain a lot in the random walk probability by removing the obstacles near $x$ (including itself), while the cost for this environment switching turns out to be small since only few obstacles are removed. 
%\item Second, consider the set of paths that come close to $x$ only few times. In this case, we switch the random walk paths inside the neighborhood of $x$ to avoid a slightly smaller neighborhood. This causes a cost in the random walk probability but not much since only a small part of the paths are modified. On the other hand, having the random walk paths modified as above, we can change the obstacles configuration in the above smaller neighborhood of $x$ to typical ones, and this gives us much gain in the environment probability.  
%\end{enumerate}
%These switching operations and~\eqref{eq:switch} show that the probability of the event that an $x\in\Oi$ has few obstacles nearby is much smaller than its complement. 

\subsection{Reduction to the Cluster of ``Truly''-Open Sites}
The key idea in our proof of Theorems~\ref{th:cover} and~\ref{th:boundary} is to approximate the range of the random walk, $S_{[0,N]}$, by a set of {\em ``truly''-open sites} $\Ti$ that depends only on the obstacle configuration $\Oi$. Unlike sites in $S_{[0,N]}$, we can easily control the environment cost of creating a ``truly''-open site, which facilitates the application of the switching argument in \eqref{eq:switch}.
\begin{definition}[``Truly''-open sites]\label{D:trulyopen}
Given an obstacle configuration $\Oi$ and $N\in\N$, a site $x\in \Z^d$ is called ``truly''-open if
\begin{equation}
 \label{eq:topen}
\bP_x\left(\tau_\Oi > (\log N)^5\right) \geq \exp\left\{-(\log N)^2\right\}.
\end{equation}
If the origin is ``truly''-open, then we let $\Ti$ denote the connected component of ``truly''-open sites inside $\ball{\mathcal{x}_N}{\varrho_N+\varrho_N^{\epsilon_1}}$ containing the origin, where $\epsilon_1$ is the constant appearing in Theorem~\ref{th:confine}. Otherwise let $\Ti=\emptyset$.
\end{definition}
\begin{remark}
\label{rem:topen}
A ``truly''-open site is a site whose surrounding environment is atypically favorable for the random walk survival. If the environment is typical, then the probability in~\eqref{eq:topen} would decay like $\exp\{-c(\log N)^{5+o(1)}\}$ (cf.~\cite[Theorem 5.1 on p.~196]{S98}). 
Note that whether $x\in\Z^d$ is ``truly''-open or not depends only on the obstacle configuration in the $l^1$-ball of radius $(\log N)^5$ centered at $x$. 
\end{remark}
The following two lemmas justifies the approximation of $\partial S_{[0,N]}$ by the boundary of ``truly''-open sites $\partial\Ti$.
\bl\label{lem:TinS}
Let $d\geq 2$. Then
%\el
%However, this result does not rule out the existence of islands of size less than $(\log \varrho_N)^2$ deep in the interior of $\Ti$ that are not visited by the walk. The next result rules this out.
%\bl\label{lem:STiboundaries}
%For $\beta$ sufficiently large, there exists $c=c(\beta)>0$ such that
\begin{equation}
 \lim_{N\to\infty} \mu_N \left(S_{[0,N]} \supset \left\{x\in\Ti\colon {\rm dist}(x,\partial\Ti)\ge{(\log N)^3}\right\}\right) = 1
\label{eq:TinS}
\end{equation}
and
\begin{equation}
\lim_{N\to\infty} \mu_N \left(\Ti \subset \left\{x\in\Z^d\colon {\rm dist}(x,S_{[0,N]})\le(\log N)^5\right\}\right) = 1.
\label{eq:Ti2S}
\end{equation}
\el
\bl\label{lem:SinTi}
%For $\beta$ sufficiently large, we have
Let $d\geq 2$. Then
\be\label{eq:SinTi}
\lim_{N\to\infty} \mu_N \big(S_{[0,N]} \subset \Ti) = 1.
\ee
\el
Indeed,~\eqref{eq:TinS} and~\eqref{eq:SinTi} imply that
\begin{equation}
 \mu_N\left(\partial S_{[0,N]}\subset \bigcup_{x\in\partial\Ti}\ball{x}{(\log N)^5}\right)\to 1
\end{equation}
and therefore, Theorems~\ref{th:cover} and~\ref{th:boundary} follow immediately from their analogues for $\Ti$.
\bt\label{th:Ticover}
Let $d\geq 2$. Then there exists $\eps_2\in (0,1)$ such that
\be\label{eq:Ticover}
\lim_{N\to\infty} \mu_N 
\left({\ball{\mathcal{x}_N}{\varrho_N - \varrho_N^{\eps_2}} \subset \Ti}
%\bigcup_{ z\in \ball{0}{\varrho_N}}\{\ball{z}{\varrho_N - \varrho_N^{\eps_2}} \subset \Ti\} 
\right)
%\big( \exists\, z\in \ball{0}{\varrho_N}, \ s.t.\ \ball{z}{\varrho_N - \varrho_N^{\eps_2}} \subset \Ti \big) 
= 1.
\ee
\et
\bt\label{th:Tiboundary}
Let $d\geq 2$. Then there exists $a>0$ such that
\be\label{eq:Tiboundary}
\lim_{N\to\infty} \mu_N \big(|\partial \Ti| \leq \varrho_N^{d-1} (\log \varrho_N)^a\big) = 1.
\ee
\et

\subsection{Proof Outline for the Cluster of ``Truly''-Open Sites} \label{SS:outline}
In this subsection, we provide an outline for the proof of Theorems~\ref{th:Ticover} and~\ref{th:Tiboundary}, assuming Lemmas~\ref{lem:TinS} and~\ref{lem:SinTi}.

Note that the random walk is confined in $\Ti$ by Lemma~\ref{lem:SinTi}. 
The geometry of $\Ti$ under $\mu_N$ is then determined by an entropy-energy balance, namely, the number of possible configurations for $\partial \Ti$, vs the probability that the random walk stays confined in $\Ti$ up to time $N$ (equivalently, the principal Dirichlet eigenvalue for the discrete Laplacian on $\Ti$).
%First, we control the entropy for $\partial \Ti$ by proving a weaker version of Theorem \ref{th:Ticover}.
%
%\begin{proposition}
%\label{prop:Ticoverw} Let $d\geq 2$.% and let $\beta$ be sufficiently large. 
%Then for any $\eps>0$, we have
%\begin{equation}
%\label{eq:Ticoverw}
% \lim_{N\to\infty} \mu_N
% \left(\ball{\blue{\mathcal{x}_N}}{(1-\eps)\varrho_N} \cap \Oi=\emptyset\right) 
%%\big(\exists\, z\in \ball{0}{\varrho_N} \ s.t.\  \ball{z}{(1-\eps)\varrho_N} \cap \Oi=\emptyset \big) 
% = 1.
%\end{equation}
%\end{proposition}
%
By definition, $\Ti$ is contained in the confinement ball $\ball{\mathcal{x}_N}{\varrho_N+\varrho_N^{\epsilon_1}}$ in Theorem~\ref{th:confine}. On the other hand, Proposition \ref{prop:Ticoverw} implies that for any $\epsilon>0$, $\ball{\mathcal{x}_N}{(1-\eps)\varrho_N}$ is a ball of ``truly''-open sites. %Together with Theorem \ref{th:confine}, this implies that for any $\eps>0$, with $\mu_N$ probability tending to one as $N$ tends to infinity, 
Therefore, it follows that
\begin{equation}
\partial\Ti \subset A(\mathcal{x}_N; (1-\eps) \varrho_N, \varrho_N+\varrho_N^{\epsilon_1}) := \ball{\mathcal{x}_N}{\varrho_N+\varrho_N^{\epsilon_1}} \setminus \overline{\ball{\mathcal{x}_N}{(1-\eps)\varrho_N}}.
\label{eq:6TinA}
\end{equation}
%
%Next, 
We bound the entropy for $\partial \Ti$ by proving the following weaker version of Theorem~\ref{th:Tiboundary}:
\begin{proposition}
\label{prop:Tiboundaryw}
Let $d\geq 2$. Then for any $b>0$,
\begin{equation}
 \label{eq:Tiboundaryw}
 \lim_{N\to\infty} \mu_N \big(|\partial \Ti| \leq \varrho_N^{d-1+b}\big) = 1.
\end{equation}
\end{proposition}
%The \blue{proof of Proposition~\ref{prop:Tiboundaryw} relies} on path and environment switching. 
%Roughly speaking, to prove Proposition~\ref{prop:Ticoverw}, the key idea is that if there is an obstacle at $x\in\Z^d$ deep within the ball of confinement, then we can do better by modifying the obstacle configuration in a neighborhood of $x$ and modifying the random walk path when it visits the same neighborhood. See Section~\ref{S:Ticoverw} for more detailed outline.\footnote{This is tentative. I don't like the outline split into two sections. This is another reason why I want to say more about \eqref{confine2} earlier in this section.}
%To bound $|\partial \Ti|$ and prove 
We prove Proposition~\ref{prop:Tiboundaryw} by considering the expected number of visits to\\ $\bigcup_{x\in \partial\Ti}\ball{x}{(\log N)^6}$ by a random walk killed upon hitting $\Oi$.
%it suffices to give a uniform lower bound on the harmonic measure (random walk exit distribution) on $\partial \Ti$. 
It suffices to prove that
\begin{enumerate}
\item the expectation of the total number of visits to $\bigcup_{x\in \partial\Ti}\ball{x}{(\log N)^6}$ is bounded from above by $(\log N)^c$ for some $c>0$;
\item uniformly in $x\in\partial \Ti$, the expected number of visits to $\ball{x}{(\log N)^6}$ is bounded from below by $N^{1-d+b}$ for any $b>0$.
\end{enumerate}
Here we consider visits to $(\log N)^6$ neighborhood of $x\in \partial \Ti$ because if the walk does not visit $\ball{x}{(\log N)^6}$, then we can switch a ``truly''-open site next to $x$ to be not ``truly''-open by only modifying the obstacle configuration inside $\ball{x}{(\log N)^6}$. The first item above follows from the fact that the random walk will typically be killed soon after visiting $x\in\partial\Ti$. The second item is proved by the path and environment switching arguments. Roughly speaking, a site $x\in\partial\Ti$ with atypically low expected number of visits is \emph{costly} for the random walk to visit. Thanks to the confinement of $\partial \Ti$ to the annulus in~\eqref{eq:6TinA}, $\ball{x}{(\log N)^6}$ can be visited only by an excursion away from $\ball{\mathcal{x}_N}{(1-\epsilon)\varrho_N}$ and we can gain a lot in the random walk probability by switching such excursions to those that stay inside $\ball{\mathcal{x}_N}{(1-\epsilon)\varrho_N}$. This implies that the random walk does not visit the $(\log N)^6$ neighborhood of $x$. We can then gain further in the environment probability by switching a ``truly''-open site next to $x$ to be not ``truly''-open, which shows that such $x\in \partial \Ti$ does not exist. 

%More precisely, if the expected number of visits to $\ball{x}{(\log \varrho_N)^6}$ for some $x\in \partial \Ti$ is exceptionally small, then the walk is unlikely to visit $\ball{x}{(\log \varrho_N)^6}$, and hence we can do better by changing the obstacle configuration in $\ball{x}{(\log \varrho_N)^6}$ to that of a typical configuration while forcing the random walk to avoid $\ball{x}{(\log \varrho_N)^6}$. Proposition \ref{prop:Ticoverw} is needed in the proof to carry out path switching, since it implies that $\partial \Ti$ is contained in the annulus $A(z; (1-\eps) \varrho_N, (1+\eps) \varrho_N)$. 

Proposition%s \ref{prop:Ticoverw} and
~\ref{prop:Tiboundaryw} provides a good enough bound on the entropy for $\partial\Ti$ to allow us to strengthen the bound on the fluctuation of $\partial \Ti$ in~\eqref{eq:6TinA} to  Theorem~\ref{th:Ticover}. More precisely, %we will show that 
if $\Ti^c$ contains a point in $\ball{\mathcal{x}_N}{\varrho_N-\varrho_N^{\eps_2}}$, then Lemma~\ref{lem:odensity} implies that $\Ti$ differs significantly in volume from the confinement ball $\ball{\mathcal{x}_N}{\varrho_N+\varrho_N^{\epsilon_1}}$. 
%We can then apply a quantitative version of the Faber-Krahn inequality by Brasco et al.~\cite{BdPV15} to bound 
Recalling that $\Ti\subset\ball{\mathcal{x}_N}{\varrho_N+\varrho_N^{\epsilon_1}}$ by definition, we can then use the Faber--Krahn inequality to show that the principal Dirichlet eigenvalue on $\Ti$ %(and the random walk survival probability) 
deviates so much from that of $\ball{\mathcal{x}_N}{\varrho_N+\varrho_N^{\epsilon_1}}$ that the loss in survival probability dominates the entropy for $\partial \Ti$. %and it would have been better to replace $\Ti$ with a ball of the same volume. 

Using Theorem~\ref{th:Ticover} on the fluctuation of $\partial \Ti$ as an input in place of the weaker Proposition~\ref{prop:Ticoverw}, we then repeat the proof of Proposition~\ref{prop:Tiboundaryw}. Now that the excursions of random walk visiting $\partial\Ti$ are smaller, the switching argument becomes more efficient and we obtain
\begin{enumerate}
 \item[$2'$.] uniformly in $x\in\partial \Ti$, the expected number of visits to $\ball{x}{(\log N)^6}$ is bounded from below by $N^{1-d}(\log N)^{-c'}$ for some $c'>0$. 
\end{enumerate} 
Combining this with the first item above, we obtain Theorem~\ref{th:Tiboundary}.\\
\smallskip

\noindent
{\bf Organization of the paper.} 
The rest of the paper is organized as follows. Section~\ref{S:Ticoverw} is devoted to the proofs of Proposition~\ref{prop:Ticoverw} and~\eqref{confine2} for general $d\ge 2$. %which will be used in Section~\ref{S:Topen} to prove Lemma~\ref{lem:Ti2S}. 
In Section~\ref{S:Topen}, we first prove Lemma~\ref{lem:TinS} with an additional property for ``truly''-open sites, and then prove Lemma~\ref{lem:SinTi} and derive Proposition~\ref{prop:Tiboundaryw} from Proposition~\ref{prop:Ticoverw}. We will in fact formulate a lemma (Lemma~\ref{lem:Glbd}) which unifies the derivation of Proposition~\ref{prop:Tiboundaryw} from Proposition~\ref{prop:Ticoverw} and the derivation of Theorem~\ref{th:Tiboundary} from Theorem~\ref{th:Ticover}. Lastly, in Section~\ref{S:Thms}, we conclude with the proof of Theorem~\ref{th:Ticover}. In Appendix~\ref{app:eigen}, we prove some technical estimates on the Dirichlet eigenvalues and eigenfunctions for the generator of the random walk, as well as a lower bound on the survival probability slightly better than in~\eqref{eq:surv}. In Appendix~\ref{app:Povel}, we briefly explain how to prove Theorem~\ref{th:confine} by adapting the argument in~\cite{P99}. Appendix~\ref{app:notation} provides an index of notation.

\section{Proof of the Weaker Version of Ball Covering Theorem
%Proposition \ref{prop:Ticoverw}
}\label{S:Ticoverw}
\subsection{Proof of Proposition~\ref{prop:Ticoverw} and the Extension of~\eqref{confine2}}
In this subsection, we prove Proposition~\ref{prop:Ticoverw} and then~\eqref{confine2} for general $d\ge 2$, assuming Lemma~\ref{lem:odensity}, which says that under the conditioned law $\mu_N(\cdot)$, obstacles cannot be too isolated.
%\bl[Density of obstacles]\label{lem:odensity} For each $x\in \Z^d$, $l>0$, and $\delta>0$, let 
%\begin{equation}
%E_l^\delta(x):=\left\{x\in \Oi \textrm{ and }{\frac{|\Oi \cap \ball{x}{l}|}{|\ball{x}{l}|}} < \delta \right\}. 
%\end{equation}
%Then there exist $\delta>0$, such that
%\be\label{odensity}
%{\mu_N\left(\bigcup_{x\in \ball{0}{2\varrho_N}} \, \bigcup_{{(\log N)^3}\leq l\leq {\varrho_N}} E_l^\delta(x)\right) \to 0 \textrm{ as } N\to\infty}
%\ee
%faster than any negative power of $N$. 
%\el
We need another lemma which states that the size of the random walk range $S_{[0,N]}$ satisfies a weak law of large numbers under $\mu_N$:
\bl[Size of random walk range]\label{lem:vol} For all $\eps>0$, we have
\be
{\mu_N \left(\left|\frac{|S_{[0,N]}|}{|\ball{0}{\varrho_N}|}-1\right|>\eps\right) \to 0 \textrm{ as } N\to\infty}
\label{eq:range-vol}
\ee
{faster than any negative power of $N$.}
\el
\noindent{\bf Proof of Lemma \ref{lem:vol}.}
For a Brownian motion among Poissonian obstacles, the corresponding result is proved in~\cite{F08}. It is straightforward to adapt the argument there to the current discrete setting. Indeed, the key point of the argument therein was the following formula for the generating function 
\begin{equation}
\int \exp\left\{\xi|S_{[0,N]}|\right\}d\mu_N 
=\frac{\bE\left[\exp\{(\xi-\log(1/p)) |S_{[0,N]}|\} \right]}{\bE\left[\exp\{-\log(1/p) |S_{[0,N]}|\} \right]}.
\end{equation}
One can use~\eqref{eq:surv} to derive the asymptotics of this for $|\xi|<\log (1/p)$ and then~\eqref{eq:range-vol} follows by standard exponential Chebyshev bounds. \qed
\medskip

%\noindent
%{\bf Proof of Proposition \ref{prop:Ticoverw}.} Suppose that Proposition \ref{prop:Ticoverw} fails so that for some $\eps>0$,
%\begin{equation}
% \limsup_{N\to\infty} \mu_N \big( \ball{\mathcal{x}_N}{(1-\eps)\varrho_N} \cap \Oi\neq \emptyset \big) >0,
%\end{equation}
%where $\mathcal{x}_N$ is the center of confinement given in Theorem \ref{th:confine}. Then by Lemma \ref{lem:odensity}, on the event that $\ball{\mathcal{x}_N}{(1-\eps)\varrho_N}$ contains an obstacle at some $x\in\Z^d$, with high probability under $\mu_N$, $\ball{x}{\eps \varrho_N}$ contains at least $\delta$ fraction of obstacles. Therefore the volume fraction of obstacles in $\ball{\mathcal{x}_N}{\varrho_N}$ does not tend to $0$ in probability. This leads to a contradiction, because Theorem \ref{th:confine} and Lemma \ref{lem:vol} imply that the volume fraction of obstacles in $\ball{\mathcal{x}_N}{\varrho_N}$ tends to $0$ in probability under $\mu_N$.
%\qed
%Let us first prove Proposition~\ref{prop:Ticoverw} and~\eqref{confine2} assuming Lemma~\ref{lem:odensity}. 
%\medskip

\noindent
{\bf Proof of Proposition~\ref{prop:Ticoverw}.} 
%We prove that $\lim_{N\to\infty}\mu_N(\ball{\mathcal{x}_N}{(1-\epsilon)\varrho_N}\cap \Oi=\emptyset)=1$. 
Thanks to Theorem~\ref{th:confine} and Lemma~\ref{lem:odensity}, we may assume that $S_{[0,N]} \subset \ball{\mathcal{x}_N}{\varrho_N+\varrho_N^{\eps_1}}$, and for any $x\in\Oi\cap\ball{0}{2\varrho_N}$, 
\begin{equation}
 \frac{|\Oi\cap\ball{x}{\epsilon\varrho_N}|}{|\ball{x}{\epsilon\varrho_N}|}\ge \delta.\label{eq:odensity2}
\end{equation}
Suppose that there is a point $x\in\ball{\mathcal{x}_N}{(1-\epsilon)\varrho_N}\cap \Oi$. Then by~\eqref{eq:odensity2}, at least $\delta$ fraction of sites in $\ball{x}{\epsilon\varrho_N}$ are closed and hence are outside $S_{[0,N]}$. Combined with $S_{[0,N]} \subset \ball{\mathcal{x}_N}{\varrho_N+\varrho_N^{\eps_1}}$, this implies that the ratio $|S_{[0,\varrho_N]}|/|\ball{0}{\varrho_N}|$ stays strictly less than one, which has $\mu_N$-probability tending to zero as $N\to\infty$ by Lemma~\ref{lem:vol}. \qed 

\medskip

\noindent
{\bf Proof of \eqref{confine2} for general $d\ge 2$.} 
We derive~\eqref{confine2} as a consequence of the following lemma, which asserts that the random walk visits all $x$ in the confinement ball such that $\ball{x}{(\log N)^3}$ is free of obstacles. 
\begin{lemma}
\label{lem:hit-centers}
\begin{equation}
 \lim_{N\to\infty}\mu_N\left(%\bigcup_{y\in\ball{0}{\varrho_N}}
\bigcup_{x\in\ball{0}{2\varrho_N}}
\left\{\tau_x%\wedge\tau_{\ball{y}{\varrho_N+\varrho_N^{\epsilon_1}}^c}
>N, \Oi\cap\ball{x}{{(\log N)^3}}=\emptyset\right\}\right)=0.
\end{equation}
\end{lemma}
Since we know from Proposition~\ref{prop:Ticoverw} that for any $\epsilon>0$, the ball $\ball{\mathcal{x}_N}{(1-\epsilon/2)\varrho_N}$ is free of obstacles with $\mu_N$-probability tending to one, Lemma~\ref{lem:hit-centers} implies that
\begin{equation}
\lim\mu_N\left(\ball{\mathcal{x}_N}{(1-\epsilon)\varrho_N}\subset S_{[0,N]}\right)=1.
%\subset\ball{\mathcal{x}_N}{\varrho_N+\varrho_N^{\epsilon_1}}, 
\label{in-and-out}
\end{equation}
%where $\mathcal{x}_N$ is the center of confinement ball given in Theorem~\ref{th:confine}. This inclusion relation implies that $\mathcal{x}_N$ is at most $2\epsilon\varrho_N$ away from the $x$ in~\eqref{in-and-out}. 
%Since $\epsilon>0$ is arbitrary, we complete the proof of~\eqref{confine2}. 
\qed

\medskip
\noindent\textbf{Proof of Lemma~\ref{lem:hit-centers}.}
By the union bound,
\begin{equation}
\begin{split}
&\mu_N \left(%\bigcup_{y\in\ball{0}{\varrho_N}}
{\bigcup_{x\in\ball{0}{2\varrho_N}}}\{\tau_x%\wedge \tau_{\ball{y}{\varrho_N+\varrho_N^{\epsilon_1}}^c} 
>N,\Oi\cap\ball{x}{{(\log N)^3}}=\emptyset\} \right)\\
&\quad\le %\sum_{y\in\ball{0}{\varrho_N}}
{\sum_{x\in\ball{0}{2\varrho_N}}}\mu_N\left(\tau_x%\wedge\tau_{\ball{y}{(1+\epsilon)\varrho_N}^c}
>N, \Oi\cap\ball{x}{{(\log N)^3}}=\emptyset\right).
\end{split}
\end{equation}
Thus it suffices show that for any $x\in{\ball{0}{2\varrho_N}}$,
\begin{equation}
 \mu_N\left(\tau_x>N,{\Oi\cap\ball{x}{{(\log N)^3}}=\emptyset}\right)
\to 0\textrm{ as }N\to \infty
\label{avoid-x}
\end{equation}
faster than any negative power of $N$. 

To this end, observe that the probability
$\bP(\tau_x\wedge\tau_\Oi>N)$ is independent of whether $x\in\Oi$ or not. Therefore, we have
\begin{equation}
\begin{split}
 &\mu_N\left(\tau_x>N, \Oi\cap\ball{x}{{(\log N)^3}}=\emptyset\right)\\
 &\quad=\frac{p}{1-p}\mu_N\left(\Oi\cap\ball{x}{{(\log N)^3}}=\{x\}\right),
\end{split}
\end{equation}
where we have switched the environment at $x$ conditionally on the random walk event $\{\tau_x>N\}$. Lemma~\ref{lem:odensity} then shows that the right-hand side decays faster than any negative power of $N$.
\qed

\medskip

The rest of this section is devoted to the proof of Lemma \ref{lem:odensity}.

\subsection{Proof Outline for Lemma \ref{lem:odensity}}\label{SS:odenoutline}

The proof of Lemma \ref{lem:odensity} is based on the environment and path switching argument in \eqref{eq:switch}. The rough (and not fully accurate) heuristic is as follows. Suppose that the event $E^\delta_l(x)$ occurs, i.e., $x\in \Oi$ and $|\Oi\cap \ball{x}{l}|/|\ball{x}{l} <\delta$ for some $\delta$ much smaller than the typical obstacle density $1-p$. Either the random walk spends a lot of time in $\ball{x}{l}$, in which case we remove all obstacles in $\ball{x}{l}$, and we expect the gain in the random walk survival probability to beat the environment cost of removing the obstacles; or the random walk spends little time in $\ball{x}{l}$, in which case we force the random walk not to visit $\ball{x}{l/4}$, change to typical obstacle configurations in $\ball{x}{l/4}$ and remove all obstacles in
$\ball{x}{l}\setminus \ball{x}{l/4}$, and we expect the probability gain in the environment to beat the cost in changing the random walk.
To make this heuristic rigorous, we need the following ingredients.
\smallskip

\noindent
{\bf Path decomposition.} First, we decompose the random walk path $(S_n)_{0\leq n\leq N}$ into successive crossings between the inner and outer shells of the annulus
\be
A(x; l/2, l) := \ball{x}{l} \setminus \overline{\ball{x}{l/2}}.
\label{eq:PathDecompB}
\ee
More precisely, we define {$\overline{\ball{x}{l/2}}:=\ball{x}{l/2}\cup\partial\ball{x}{l/2}$ and} stopping times
\begin{align}
\sigma_1 & := \min\{n\geq 0\colon S_n \in {\overline{\ball{x}{l/2}}}\}{\wedge N},
\end{align}
and for $k\in\N$, 
\begin{align}
\tau_k & := \min\{n> \sigma_k \colon S_n \in B^c(x,l)\}{\wedge N}; \\
\sigma_{k+1} & := \min\{n> \tau_k \colon S_n \in {\overline{\ball{x}{l/2}}}\}{\wedge N}.\label{eq:PathDecompE}
\end{align}
We will perform path switching on the crossings from ${\overline{\ball{x}{l/2}}}$ to $\ball{x}{l}^c$, and perform environment switching in the ball $\ball{x}{l}$.
\medskip

\noindent{\bf Skeletal approximation of $\Oi\cap \ball{x}{l}$.}
When the random walk spends a lot of time in $\ball{x}{l}$, we will remove all the obstacles in $\ball{x}{l}$. We need to estimate the gain in the random walk survival probability as a function of $\Oi\cap \ball{x}{l}$. This is not a very simple task and one of the difficulties is that the contribution from each obstacle is highly non-uniform, depending on others. Indeed, if an obstacle is well surrounded by others, then we gain little in the survival probability by removing it. In order to make the gain from an obstacle independent from others, we will only count the gain from a skeletal set ${\cal X}_l(x)\subset \Oi$ with properties 
\begin{align}
&\textrm{$x\in {\cal X}_l(x)$ and all sites in ${\cal X}_l(x)$ are at least distance $l^{1/2d}$ away}\notag \\
&\textrm{ from each other;}\label{X-separation}\\
&\textrm{each $y\in \Oi\cap \ball{x}{l}$ is within distance $l^{1/2d}$ of some site in ${\cal X}_l(x)$.}\label{X-covering}
\end{align}
Such a set can be constructed iteratively. First include $x\in {\cal X}_l(x)$ and remove all obstacles in $\ball{x}{l^{1/2d}}$.
Next pick any one of the remaining obstacles at $y\in \ball{x}{l}$ that is closest to $x$ and add it to ${\cal X}_l(x)$, and remove all obstacles in $\ball{y}{l^{1/2d}}$. Repeat this procedure until no obstacles remain in $\ball{x}{l}$.
Another difficulty is that {we gain little in the random walk survival probability by removing obstacles in $\ball{x}{l}$ near $\partial\ball{x}{l}$ since we will only count the gain from the crossings $\{S_{[\sigma_k,\tau_k]}\}_{k\in\N}$, that typically spend little time near $\partial\ball{x}{l}$.} Therefore we focus on the obstacles deeply inside $\ball{x}{l}$ by setting
\be
{\cal X}_l^\circ:={\cal X}_l^\circ(x):= {\cal X}_l(x) \cap \ball{x}{l/2}.
\label{def:calXo}
\ee

\smallskip
\noindent{\bf Random walk estimates.} 
For $D\subset \Z^d$, $u\in D$ and $v\in D\cup\partial D$, we denote 
\be
p^D_n(u,v) := \bP_u(S_n=v, S_{[1, n-1]}\subset D).
\label{eq:trans_prob}
\ee
This is nothing but the discrete {space-time} Dirichlet heat kernel on $D$ if $v\in D$ and, while it is not for $v\in\partial D$, it always satisfies the discrete heat equation in $(n,u)$. 

{\begin{remark}
\label{rem:parity}
Since the symmetric simple random walk has period 2, we have $p^D_n(u,v)=0$ when $n+|u-v|_1$ is an odd number. In what follows, we adopt a convention that $p^D_n(u, v)$ is understood to be $p^D_{n+1}(u,v)$ if $n+|u-v|_1$ is odd. 
%Since this situation will not happen in the arguments below, we assume $p^D_n(u,v)>0$ whenever we make a statement for generic $(n,u,v)$. 
\end{remark}
}
Now we are ready to state the gain in the random walk survival probability when we remove obstacles: for $c_0, c_1>0$ to be determined later in Lemma~\ref{lem:rwest}, uniformly in $m \ge (c_0l)^2$, {$u \in \ball{x}{l/2}$ and $v \in \overline{\ball{x}{l}}$}, 
\begin{align}
p^{\ball{x}{l} \setminus {\Oi}}_m(u, v)  
\leq e^{-c_1 (\lfloor m /(c_0l)^2\rfloor{-1}) \Gamma(|{\cal X}_l^\circ|)}p_m^{\ball{x}{l}}(u,v), 
\tag{RW1}
\label{remove-O}
\end{align}
where for $k\in\N$,
\be
\begin{split}
\Gamma(k)
 :=
 \begin{cases}
  \left(\log \frac{(c_0l)^{3/2}}{2k}\right)^{-1}\vee 0, & d = 2,\\[10pt]
 \frac{l^{2-d}k}{1+ l^{{(2-d)}/{2d}} k^{2/d}}, &d\ge 3.
 \end{cases}
\end{split}
\label{eq:Gamma}\ee
Roughly speaking, this estimate means that if the random walk stays in $\ball{x}{l}$, then in every $(c_0l)^2$ steps, it has more than $c_1\Gamma(|{\cal X}_l^\circ|)$ probability of hitting an obstacle (see Lemma~\ref{lem:calXo}). 
As mentioned at the beginning of this subsection, we force the crossing to avoid $\ball{x}{l/4}$ when the random walk spends little time in $\ball{x}{l}$. We need another random walk estimate to quantify the effect of this switching and also some others to deal with complementary cases. Those random walk estimates will be tagged as (RW1)--(RW5) and {restated and proved in Lemma~\ref{lem:rwest}} in Subsection~\ref{SS:tech-proof}.
\medskip

\noindent{\bf Obstacles deep in the interior of $\ball{x}{l}$.} Note that in~\eqref{remove-O}, the bound is in terms of $|{\cal X}_l^\circ(x)|$, the number of skeletal points of $\Oi\cap \ball{x}{l}$ in $\ball{x}{l/2}$. 
%In fact, if the obstacles in $\ball{x}{l}$ all cluster around the boundary, then we will gain very little in the random walk survival probability by removing all obstacles in $\ball{x}{l}$. 
To ensure that the gain in~\eqref{remove-O} dominates the cost of removing all obstacles in $\ball{x}{l}$, we need that $\ball{x}{l}$ has sufficiently many obstacles deep in its interior, namely, $|{\cal X}_l^\circ(x)| \geq \rho |{\cal X}_l(x)|$ for some $\rho>0$. The next lemma guarantees that we can achieve this by slightly changing the radius.
\bl\label{lem:obstaclebal}
Suppose that $x\in \Oi$ and $|\Oi\cap \ball{x}{L}|/|\ball{x}{L}|<\delta$ for some $L\in\N$. Then we can find $\rho>0$ independent of $L$ and $\delta$, such that there exists $L^{5/6}\leq l\leq L$ with $|\Oi\cap \ball{x}{l}|/|\ball{x}{l}|<\delta$ and $|{\cal X}^\circ_l(x)| \geq \rho \min\{|{\cal X}_l(x), \delta l^{d-1/2}\}$.
\el
Therefore it suffices to prove Lemma \ref{lem:odensity} where we replace the event $E^\delta_l(x)$ by the event
\be\label{eq:Edeltarho}
E^{\delta, \rho}_l(x) := E^\delta_l(x) \cap \big\{\, |{\cal X}^\circ_l(x)| \geq \rho \min\{|{\cal X}_l(x)|, \delta l^{d-1/2}\} \,\big\}
\ee
with $l\in[(\log N)^{5/2},\varrho_N]$, which will be carried out in the next subsection. 

\medskip
%%%%%%%%%%%%%%%%%%%%%%%%%%%%%%%%%%%%%%%%%%%%%%%%%%%%%%%%%%%%%

\medskip
\noindent
{\bf Proof of Lemma \ref{lem:obstaclebal}.}
Let
\be
j^* := \min\big\{j\geq 0\colon l=L/2^j \mbox{ satisfies } |{\cal X}^\circ_l(x)| \geq \rho \min\{|{\cal X}_l(x)|, \delta l^{d-1/2}\} \big\}. 
\ee
We will show that $l^*:=L/2^{j^*}$ satisfies the desired properties if $\rho>0$ is small enough. 

If $j^*=0$, then $l^*=L$ works. Otherwise, for all $l=L/2^j$ with $0\leq j\leq j^*-1$, we have
\be\label{eq:calXbd}
|{\cal X}^\circ_l(x)| < \rho \min\{|{\cal X}_l(x), \delta l^{d-1/2}\} \leq \rho |{\cal X}_l(x)|. 
\ee
We claim that for all $l\geq 1$, 
\be\label{eq:calXcomp}
|{\cal X}_l(x)| \leq C_d |{\cal X}^\circ_{2l}(x)| \quad \mbox{for some $C_d$ depending only on $d$}. 
\ee
Together with \eqref{eq:calXbd}, this implies that 
\be
|{\cal X}^\circ_{2l^*}(x)| < \rho C_d |{\cal X}^\circ_{4 l^*}(x)| <\cdots < (\rho C_d)^{j^*-1} |{\cal X}^\circ_L(x)| \leq (\rho C_d)^{j^*-1} C L^d.
\ee
Since $|{\cal X}^\circ_{L}(x)|, \dotsc, |{\cal X}^\circ_{2l^*}(x)|\geq 1$ because $x\in \Oi$, we must have $(\rho C_d)^{-j^*+1} \leq CL^d$. 
We can then choose $\rho>0$ small such that $2^{j^*}\leq L^{5/6}$, and hence $l^*=L/2^{j^*} \in [L^{5/6}, L]$. 

To bound the volume fraction of obstacles in $\ball{x}{l^*}$ when $j^*\geq 1$, we can apply \eqref{eq:calXbd} at $l= L/2^{j^*-1} =2l^*$ to obtain
\begin{equation}
\begin{split}
  |\Oi \cap \ball{x}{l^*}| 
&\leq |{\cal X}^\circ_{2l^*}(x)| \cdot |\ball{0}{(2l^*)^{1/2d}}|\\ 
&< \rho \delta (2l^*)^{d-1/2} |\ball{0}{(2l^*)^{1/2d}}|\\
& < \delta |\ball{x}{l^*}|,
\end{split}
\end{equation}
where the last inequality holds if $\rho>0$ is chosen small.  

Lastly, we prove the claim \eqref{eq:calXcomp}. Note that by our construction of ${\cal X}_l(x)$ and ${\cal X}_{2l}^\circ(x)$, 
\begin{equation}
 \begin{split}
 &\bigcup_{y\in {\cal X}_l(x)} \ball{y}{l^{1/2d}} \subset \bigcup_{y\in \Oi\cap \ball{x}{l}} \ball{y}{l^{1/2d}}, \\
 &\quad \mbox{where} \quad  \Oi\cap \ball{x}{l} \subset \bigcup_{z\in {\cal X}_{2l}^\circ(x)} \ball{z}{(2l)^{1/2d}}. 
\end{split}
\end{equation}
Therefore by doubling the radii of the balls $\{\ball{z}{(2l)^{1/2d}}\}_{z\in {\cal X}_{2l}^\circ(x)} $, we obtain
$$
\bigcup_{y\in {\cal X}_l(x)} \ball{y}{l^{1/2d}} \subset \bigcup_{z\in {\cal X}_{2l}^\circ(x)} \ball{z}{2(2l)^{1/2d}}. 
$$
Since the balls $\{\ball{y}{l^{1/2d}}\}_{y\in {\cal X}_l(x)}$ are disjoint, a volume calculation then yields \eqref{eq:calXcomp}.\qed

\subsection{Proof of Lemma \ref{lem:odensity}}\label{SS:odenproof}
As remarked after Lemma \ref{lem:obstaclebal}, if suffices to prove Lemma \ref{lem:odensity} with the event $E^\delta_l(x)$ replaced by $E^{\delta, \rho}_l(x)$, with $l\in[(\log N)^{5/2},\varrho_N]$. We will prove the following bound on $\mu_N(E^{\delta, \rho}_l(x))$, which immediately implies Lemma \ref{lem:odensity} by a union bound over all $x\in \ball{0}{2\varrho_N}$ and $(\log N)^{{5/2}} \le l \le \varrho_N$.

\bl\label{lem:Edeltalbd}
Let $E^{\delta, \rho}_l(x)$ be defined as in \eqref{eq:Edeltarho}.
%, with $(\log N)^\alpha \le l \le \varrho_N$ for some $\alpha>0$.
There exist $c_3>0$ depending on $d$, $p$ and $\delta$ such that for all $l\in[\log N, \varrho_N]$, we have 
\be
\mu_N(E^{\delta, \rho}_l(x)) \leq \exp\{-c_3l^{1/2}\}.
\ee
\el
\noindent{\bf Proof of Lemma~\ref{lem:Edeltalbd}.}
Recall the path decomposition introduced in Section \ref{SS:odenoutline}, where we identified the successive crossings from ${\overline{\ball{x}{l/2}}=\ball{x}{l/2}\cup\partial\ball{x}{l/2}}$ to $\ball{x}{l}^c$ during the time intervals $[\sigma_k, \tau_k]$, $k\in\N$. Since these stopping times are truncated by $N$, the duration $\tau_k-\sigma_k$ can be zero. Henceforth, the word \emph{crossing} refers to $S_{[\sigma_k,\tau_k]}$ with $\tau_k-\sigma_k>0$. In particular, the last crossing may be incomplete. To carry out the path and environment switching, we distinguish between three cases and in order to describe them, we need some more notation. We denote a sequence of numbers or vectors in bold face as $\bm{a}=(a_k)_{k\ge 1}$ and introduce the set of interlacing sequences 
\begin{equation}
\bm{I}_N :=\{(\bm{s},\bm{t})\colon 0\le s_k\le t_k\le s_{k+1}\le N \textrm{ for all }k\in\N\}.
\end{equation}
For $(\bm{s},\bm{t})\in\bigcup_{N\ge 1}\bm{I}_N$, we write 
\begin{equation}
K(\bm{s},\bm{t}):=\sup\{k\ge 1\colon t_k-s_k>0\}
\end{equation}
which represents the number of crossings when $(\bm{s},\bm{t})=(\bm{\sigma},\bm{\tau})$. Now we are ready to describe the three cases. Recall that $c_0>0$ has already been chosen to satisfy Lemma~\ref{lem:rwest}. The constant $\delta>0$ is to be determined later, depending only on the dimension $d$ and the open probability $p$.
\begin{itemize}
\item[(1)]
There are many crossings and more than half of them are short ($\le(c_0l)^2$), that is, $(\bm{\sigma},\bm{\tau})$ belongs to 
\begin{equation}
 \bm{F}_1:=\left\{(\bm{s},\bm{t})\in\bm{I}_N\colon 
  \left|\{k\ge 1\colon 0<t_k - s_k\le (c_0l)^2\}\right|>\frac{1}{2}K(\bm{s},\bm{t})\vee \delta^{1/d} l^d \right\}.
\label{ManyShort}
\end{equation}
\item[(2)]
The total time duration of the long crossings ($>(c_0l)^2$) is long, that is, $(\bm{\sigma},\bm{\tau})$ belongs to 
\begin{equation}
\bm{F}_2:=\left\{(\bm{s},\bm{t})\in\bm{I}_N\colon\sum_{k\ge 1} (t_k - s_k)1_{\{t_k-s_k>(c_0l)^2\}}> \delta^{1/d}c_0^2 l^{d+2}\right\}.
\label{TotalLong} 
\end{equation}
\item[(3)]
The number of crossings as well as {their total duration} are small, that is, $(\bm{\sigma},\bm{\tau})$ belongs to 
\begin{equation}
 \bm{F}_3:=\left\{(\bm{s},\bm{t})\in\bm{I}_N\colon K(\bm{s},\bm{t}) \le 2\delta^{1/d} l^d\textrm{ and } 
\sum_{k\ge 1} (t_k - s_k) \le 2\delta^{1/d}c_0^2 l^{d+2}\right\}.
\label{FewAndShort}
\end{equation}
\end{itemize}
These three cases exhaust all possibilities. Indeed, if $(\bm{s},\bm{t})\not\in \bm{F}_2$, then the number of long crossings is at most $\delta^{1/d}l^d$, and their total duration is at most $\delta^{1/d}c_0^2 l^{d+2}$. If in addition, $(\bm{s},\bm{t})\not\in \bm{F}_1$, then the number of short crossings is either less than the number of long crossings, or less than $\delta^{1/d}l^d$; either way, it is bounded by $\delta^{1/d}l^d$, and their total duration is at most $\delta^{1/d}c_0^2 l^{d+2}$. Combining the short and long crossings, one finds that $(\bm{s},\bm{t})\in \bm{F}_3$.

%Note that $\bm{F}_2$ includes the case where there are many crossings and more than half of them are long, and thereby $ \bm{F}_1\cup \bm{F}_2\supset \{(\bm{s},\bm{t})\in\bm{I}_N\colon K(\bm{s},\bm{t})\ge 2\delta^{1/d} l^d\}$. Note also that if $K(\bm{s},\bm{t})< 2\delta^{1/d} l^d$ and $\sum_{k\ge 1} (t_k - s_k)> 3\delta^{1/d}c_0^2 l^{d+2}$, then since the short crossings can contribute at most $2\delta^{1/d} c_0^2l^{d+2}$ to the total length, the condition in~\eqref{TotalLong} holds. Therefore it follows that the above (1)--(3) cover all the possibilities. 

For each of the three cases above, by summing over all possible values of $(\sigma_k, \tau_k) \in \bm{F}_i$ ($i\in\{1,2,3\}$) and the position of the walk at these times, we obtain
\begin{equation}
\begin{split}
 &\bP\left(\tau_{\Oi}>N\textrm{ and }(\bm{\sigma},\bm{\tau})\in \bm{F}_i\right)\\
 &\quad 
 =\sum_{(\bm{s},\bm{t})\in \bm{F}_i}
 \sum_{\bm{u},\bm{v}} p^{{{\Z^d}\setminus (\Oi\cup\ball{x}{l/2})}}_{s_1}(0,u_1)\\
 &\qquad\times\prod_{k\ge 1} p^{\ball{x}{l}\setminus \Oi}_{t_k-s_k}(u_k,v_k)p^{\Z^d\setminus(\Oi\cup\ball{x}{l/2})}_{s_{k+1}-t_k}(v_k,u_{k+1}),
\label{RW-decomp}
\end{split}
\end{equation}
where $\bm{u}$ and $\bm{v}$ range over all the possible starting and ending points of crossings with $(\bm{\sigma},\bm{\tau})=(\bm{s},\bm{t})$. In particular, $u_k\in \partial\ball{x}{l/2}$ and $v_k\in \partial\ball{x}{l}$ as long as $s_k<N$ and $t_k<N$ respectively, {except possibly $u_1=0$ when $0\in{\ball{x}{l/2}}$.} For simplicity, we assume $\delta^{1/d}l^d\in \N$ in this proof. 

\medskip
\noindent\underline{Case (1)}: In this case, we remove all the obstacles inside $\ball{x}{l}$ and lengthen all the short crossings by $l^2$. We formalize this as the environment and path switching~\eqref{eq:switch} by setting
\begin{align}
&(A_1,A_2):=(\{(\bm{\sigma},\bm{\tau})\in \bm{F}_1\}, \{\tau_{\Oi}>N\});\\
&(E_1,E_2):=({E^{\delta,\rho}_l(x)},\{\Oi\cap\ball{x}{l}=\emptyset\}).
%&(\Oi_1,\Oi_2):=(\Oi,\Oi\setminus\ball{x}{l}).
\end{align}
Since $|\Oi\cap\ball{x}{l}|\le \delta |\ball{x}{l}|$ on the event $E^{\delta,\rho}_l(x)$, the cost of environment switching can be estimated as
\begin{equation}
\begin{split}
\frac{\P(E^{\delta,\rho}_l)}{\P(\Oi\cap \ball{x}{l}=\emptyset)}
& \le \frac{\P\left(|\Oi\cap\ball{x}{l}|\le \delta |\ball{x}{l}|\right)}
{\P(\Oi\cap \ball{x}{l}=\emptyset)}\\
& \le \sum_{i=0}^{\delta |\ball{x}{l}|}\binom{|\ball{x}{l}|}{i}
\left(\frac{1-p}{p}\right)^{{i}}\\
& \le e^{c\delta(\log{\frac{1}{\delta}}) l^d}
\end{split}
\label{case(1)-obs}
\end{equation}
by using Stirling's approximation. Alternatively, one can also interpret this as a consequence of Cramer's large deviation principle.  

On the other hand, since the short crossings are unlikely to happen, we gain in the random walk probability by lengthening them. More precisely, {we will see in Lemma~\ref{lem:rwest} that} for $t_k-s_k\le (c_0l)^2$, $u_k\in {\overline{\ball{x}{l/2}}}$ and $v_k\in \partial\ball{x}{l}$, 
\begin{equation}
\begin{split}
 p^{\ball{x}{l}\setminus \Oi}_{t_k-s_k}(u_k,v_k)
 &\le p^{\ball{x}{l}}_{t_k-s_k}(u_k,v_k)\\
 &\le \frac{1}{100} p^{\ball{x}{l}}_{t_k-s_k+l^2}(u_k,v_k),
\end{split}
\label{too-short}
\tag{RW2}
\end{equation}
where $l^2$ is to be understood as $l^2+1$ when $l$ is odd as mentioned in Remark~\ref{rem:parity}. It is easy to see that this change is harmless for the following argument and henceforth we will not mention this parity convention again. 
Setting 
\begin{equation}
{I_k:=1_{\{0<t_k-s_k\le (c_0l)^2, u_k\in \partial\ball{x}{l/2},v_k\in \partial\ball{x}{l}\}},}
\end{equation}
we can bound the product in the right-hand side of \eqref{RW-decomp} by 
\begin{equation}
\begin{split}
&\prod_{k\ge 1} p^{\ball{x}{l}\setminus \Oi}_{t_k-s_k}(u_k,v_k)p^{\Z^d\setminus(\Oi\cup\ball{x}{l/2})}_{s_{k+1}-t_k}(v_k,u_{k+1})\\
&\quad \le 100^{-\sum_{k\ge 1} I_k}
\prod_{k\ge 1} p^{\ball{x}{l}}_{t_k+l^2 I_k -s_k}(u_k,v_k)p^{\Z^d\setminus(\Oi\cup\ball{x}{l/2})}_{s_{k+1}-t_k}(v_k,u_{k+1})\\
&\quad=100^{-\sum_{k\ge 1} I_k}
\prod_{k\ge 1} p^{\ball{x}{l}}_{\tilde{t}_k-\tilde{s}_k}(u_k,v_k)p^{\Z^d\setminus(\Oi\cup\ball{x}{l/2})}_{\tilde{s}_{k+1}-\tilde{t}_k}(v_k,u_{k+1}),
\end{split}
\label{F_1|stuv}
\end{equation}
where $\tilde{s}_k:=s_k+l^2\sum_{m<k}I_m $ and $\tilde{t}_k:=t_k+l^2\sum_{m\le k}I_m$. Let us consider the cases 
\begin{equation}
(\bm{\sigma},\bm{\tau})\in \bm{F}_{1,j}
:=\left\{(\bm{s},\bm{t})\in \bm{F}_1\colon \sum_{k\ge 1}I_k=j\right\},
\end{equation}
that is, exactly $j$ crossings are lengthened, separately for $j\in\{\delta^{1/d}l^d,\delta^{1/d}l^d+1,\dotsc,N\}$. Summing~\eqref{F_1|stuv} multiplied by $p^{{{\Z^d}\setminus (\Oi\cup\ball{x}{l/2})}}_{s_1}(0,u_1)$ over $(\bm{s},\bm{t})\in \bm{F}_{1,j}$ and $(\bm{u},\bm{v})$, we obtain
\begin{equation}
\begin{split}
&\bP(\tau_\Oi>N\textrm{ and }(\bm{\sigma},\bm{\tau})\in \bm{F}_{1,j})\\
&\quad =\sum_{(\bm{s},\bm{t})\in \bm{F}_{1,j}}\sum_{\bm{u},\bm{v}} 
p^{{{\Z^d}\setminus (\Oi\cup\ball{x}{l/2})}}_{s_1}(0,u_1)
 \prod_{k\ge 1} p^{\ball{x}{l}\setminus \Oi}_{t_k-s_k}(u_k,v_k)p^{\Z^d\setminus(\Oi\cup\ball{x}{l/2})}_{s_{k+1}-t_k}(v_k,u_{k+1})\\
&\quad \le  100^{-j}
\sum_{({\bm{s}},{\bm{t}})\in \bm{F}_{1,j}}\sum_{\bm{u},\bm{v}} 
p^{{{\Z^d}\setminus (\Oi\cup\ball{x}{l/2})}}_{\tilde{s}_1}(0,u_1) 
\prod_{k\ge 1} p^{\ball{x}{l}}_{\tilde{t}_k-\tilde{s}_k}(u_k,v_k)p^{\Z^d\setminus(\Oi\cup\ball{x}{l/2})}_{\tilde{s}_{k+1}-\tilde{t}_k}(v_k,u_{k+1}).
\end{split}
\label{case(1)-rw-mult}
\end{equation}
In order to relate this last line to $\bP(\tau_\Oi>N)$, we rewrite~\eqref{case(1)-rw-mult} as a summation over 
\begin{equation}
 (\tilde{\bm{s}},\tilde{\bm{t}})\in\tilde{F}_{1,j}
:=\{(\tilde{\bm{s}},\tilde{\bm{t}})\colon (\bm{s},\bm{t})\in \bm{F}_{1,j}\}.
\end{equation}
Note that each $(\tilde{\bm{s}},\tilde{\bm{t}})$ may come from different $(\bm{s},\bm{t})$'s but with the same number of crossings $K(\bm{s},\bm{t})$ and hence its pre-image has cardinality at most $2^{K(\bm{s},\bm{t})}\le 2^{2+2j}$ on $\bm{F}_{1,j}$. Recalling also that $j\ge \delta^{1/d}l^d$, it follows from~\eqref{case(1)-rw-mult} that
\begin{equation}
\begin{split}
&\bP(\tau_\Oi>N\textrm{ and }(\bm{\sigma},\bm{\tau})\in \bm{F}_{1,j})\\
&\quad \le 
4\cdot 25^{-\delta^{1/d}l^d}
\sum_{(\tilde{\bm{s}},\tilde{\bm{t}})\in\tilde{F}_{1,j}}
\sum_{\bm{u},\bm{v}} 
 p^{{{\Z^d}\setminus (\Oi\cup\ball{x}{l/2})}}_{\tilde{s}_1}(0,u_1)\\
&\quad\quad \times\prod_{k\ge 1} p^{\ball{x}{l}}_{\tilde{t}_k-\tilde{s}_k}(u_k,v_k)p^{\Z^d\setminus(\Oi\cup\ball{x}{l/2})}_{\tilde{s}_{k+1}-\tilde{t}_k}(v_k,u_{k+1}).
\end{split}
\label{case(1)-rw-mult2}
\end{equation}
The sum on the right-hand side is seen to be bounded by $\bP(\tau_{\Oi\setminus\ball{x}{l}}>N+jl^2)$. Indeed, any $(\tilde{\bm{s}},\tilde{\bm{t}})\in \tilde{F}_{1,j}$ has terminal time $\tilde{s}_{K(\tilde{\bm{s}},\tilde{\bm{t}})+1}= N+jl^2$ by construction and hence the above sum represents (a part of) the path decomposition before time $N+jl^2$. 

Taking a sum of~\eqref{case(1)-rw-mult2} over $j$, we obtain
\begin{equation}
\begin{split}
&\bP(\tau_\Oi>N\textrm{ and }(\bm{\sigma},\bm{\tau})\in \bm{F}_1)\\
&\quad\le 
4\cdot 25^{-\delta^{1/d}l^d}\sum_{j=\delta^{1/d}l^d}^N\bP\left(\tau_{\Oi\setminus\ball{x}{l}}>N+jl^2\right)\\
&\quad \le 
4N \cdot 25^{-\delta^{1/d}l^d}\bP\left(\tau_{\Oi\setminus\ball{x}{l}}>N\right).
\end{split}
\end{equation}
Since $\tau_{\Oi\setminus\ball{x}{l}}=\tau_{\Oi}$ on $\{\Oi\cap\ball{x}{l}=\emptyset\}$, recalling~\eqref{case(1)-obs} and $l\ge \log N$ and choosing $\delta>0$ small, we can use~\eqref{eq:switch} to conclude that
\begin{equation}
\mu_N\left((S,\Oi)\in \left(\{(\bm{\sigma},\bm{\tau})\in \bm{F}_1\},E^{\delta,\rho}_l(x)\right)\right)
\le e^{-c\delta^{1/d}l^d}.
\end{equation}
%%%%%%%%%%%%%%%%%%%%%%%%%%%%%%%%%%%%%%%%%%%%%%%%%%%%%%%%%%%%%%%

\medskip
\noindent\underline{Case (2)}:
In this case, we again remove all the obstacles in $\ball{x}{l}$ and leave the crossings unchanged. We apply the same environment and path switching as in the previous case. 
Since the long crossings have higher probability of hitting the obstacles, removing the obstacles gives us a large gain in the random walk probability. In order to make it precise, note first that $(\bm{s},\bm{t})\in \bm{F}_2$ implies 
\begin{equation}
 \sum_{k\ge 1} \left \lfloor \frac{t_k - s_k}{(c_0l)^2} \right\rfloor 
\ge  \frac{1}{2}\sum_{k\ge 1} \frac{t_k - s_k}{(c_0l)^2} 1_{\{t_k-s_k\ge (c_0l)^2\}}
> \frac{\delta^{1/d}}{2}l^d.
\end{equation} 
Given this, we use the aforementioned {(see also Lemma~\ref{lem:rwest})}
\begin{equation}
 p^{\ball{x}{l}\setminus \Oi}_{t_k-s_k}(u_k,v_k)
 \le e^{-c_1(\lfloor (t_k-s_k)/(c_0l)^2\rfloor{-1}) \Gamma(|{\cal X}_l^\circ|)} 
p^{\ball{x}{l}}_{t_k-s_k}(u_k,v_k)
\label{remove-O}\tag{RW1}
\end{equation}
repeatedly to obtain
\begin{equation}
\begin{split}
&\bP(\tau_\Oi>N\textrm{ and }(\bm{\sigma},\bm{\tau})\in \bm{F}_2)\\
&\quad=\sum_{(\bm{s},\bm{t})\in \bm{F}_2}\sum_{\bm{u},\bm{v}}
p^{{{\Z^d}\setminus (\Oi\cup\ball{x}{l/2})}}_{s_1}(0,u_1)
\prod_{k\ge 1} p^{\ball{x}{l}\setminus \Oi}_{t_k-s_k}(u_k,v_k)p^{\Z^d\setminus(\Oi\cup\ball{x}{l/2})}_{s_{k+1}-t_k}(v_k,u_{k+1})\\
&\quad \le e^{-c\delta^{1/d}l^d\Gamma({\cal X}_l^\circ)}
\sum_{(\bm{s},\bm{t})\in \bm{F}_2}\sum_{\bm{u},\bm{v}}
p^{{{\Z^d}\setminus (\Oi\cup\ball{x}{l/2})}}_{s_1}(0,u_1)\\
&\quad\quad\times\prod_{k\ge 1} p^{\ball{x}{l}}_{t_k -s_k}(u_k,v_k)p^{\Z^d\setminus(\Oi\cup\ball{x}{l/2})}_{s_{k+1}-t_k}(v_k,u_{k+1})\\
&\quad\le e^{-c\delta^{1/d}l^d\Gamma(|{\cal X}_l^\circ|)}
\bP(\tau_{\Oi\setminus\ball{x}{l}}>N)
\end{split}
\label{case(2)-rw}\end{equation}
uniformly in $\Oi$, {and we can replace $\tau_{\Oi\setminus\ball{x}{l}}$ by $\tau_{\Oi}$ on $\{\Oi\cap\ball{x}{l}=\emptyset\}$ as before.}

We are going to show that the cost {of removing the obstacles in $\ball{x}{l}$} is much smaller than the above gain in the random walk probability. Recall that $|\mathcal{X}_l^\circ| \geq \rho \min (\delta l^{d-1/2},|\mathcal{X}_l|)$ on the event $E^{\delta,\rho}_l(x)$ and also note that~\eqref{X-separation} implies the bound $|\mathcal{X}_l^\circ| \leq C\delta l^{d-1/2}$ for some $C>0$ depending only on the dimension. In the case $|\mathcal{X}_l^\circ| \in [\rho\delta l^{d-1/2},C\delta l^{d-1/2}]$, we have that $\Gamma(|\mathcal{X}_l^\circ|)$ is bounded below by a positive constant, recalling the definition of $\Gamma$ in~\eqref{eq:Gamma}. Combining~\eqref{case(2)-rw} with~\eqref{case(1)-obs} and~\eqref{eq:switch} and choosing $\delta$ small, we get
\begin{equation}
\begin{split}
& \mu_N\left((S,\Oi)\in\left(\{(\bm{\sigma},\bm{\tau})\in \bm{F}_2\},(E^{\delta,\rho}_l(x)\cap \{|{\cal X}_l^\circ| \geq \rho\delta l^{d-1/2}\})\right)\right)\\
&\quad \le e^{-c\delta^{1/d}l^d}. 
\end{split}
 \label{case(2)-1}
\end{equation}
In the other case $|{\cal X}^\circ_l|\in[\rho|{\cal X}_l|,\rho\delta l^{d-1/2})$, instead, we have {that for sufficiently small $\delta$,}
\begin{equation}
\Gamma(|{\cal X}^\circ_l|)\ge 
c
%\left(\log \frac{(c_0l)^{3/2}}{2|{\cal X}_l^\circ|}\right)^{-1},&d=2,\\
 \delta^{-2/d}l^{1/2-d}|{\cal X}^\circ_l|/|\log \delta| %,&d\ge 3
\end{equation}
recalling~\eqref{eq:Gamma} again. Indeed, for $d\ge 3$, using $|{\cal X}^\circ_l|\le \rho\delta l^{d-1/2}$ in the denominator in~\eqref{eq:Gamma} yields the above bound without $|\log \delta|$; for $d=2$, the argument of $\log$ in~\eqref{eq:Gamma} is large and the above bound follows from the fact that {$1/\log r= r^{-1} (r/\log r)$ and} $r/\log r$ is increasing for $r$ large. %We can make the bound for $d=2$ in almost the same form as $d\ge 3$ and $|{\cal X}^\circ_l|<\rho\delta l^{d-1/2}$ since for sufficiently small $\delta$, 
%\begin{equation}
%\begin{split}
%\left(\log \frac{(c_0l)^{3/2}}{2k}\right)^{-1}
%& \ge \frac{2k}{(c_0l)^{3/2}}\inf_{r\ge c_0^{3/2}(\rho\delta)^{-1}}\frac{r}{\log 2r}\\
%& \ge c\frac{\delta^{-1}}{|\log \delta|} l^{-3/2}k.
%\end{split}
%\end{equation}
Given this lower bound on $\Gamma(|{\cal X}_l^\circ|)$, the gain from the random walk becomes
\begin{equation}
\frac{\bP(\tau_\Oi>N\textrm{ and }(\bm{\sigma},\bm{\tau})\in \bm{F}_2)}
{\bP(\tau_{\Oi\setminus\ball{x}{l}}>N)}
\le 
 e^{-c\delta^{-1/d}l^{1/2}|{\cal X}_l^\circ|/|\log\delta|}.
\label{case(2)-rw2}
\end{equation}
{Note that this gain is much smaller than} the bound~\eqref{case(1)-obs} on the cost of environment switching when $|{\cal X}^\circ_l|$ is small. {Therefore we have to} estimate the environment switching cost more carefully {and this is done} by considering {separately} the events $\{|{\cal X}^\circ_l|=k\}$ for $k <\rho\delta l^{d-1/2}$. 

In the case under consideration, $|{\cal X}^\circ_l|=k$ implies $|{\cal X}_l|\le \rho^{-1}k$. Recall also that all the obstacles in $\ball{x}{l}$ are contained in $\bigcup_{x\in {\cal X}_l}\ball{x}{l^{1/2d}}$ by~\eqref{X-covering}. Therefore on each event $\{|{\cal X}^\circ_l|=k\}$, by counting the possible choices of ${\cal X}_l$ first and then the configurations inside $\bigcup_{x\in {\cal X}_l}\ball{x}{l^{1/2d}}$, we can estimate the cost of environment switching as
\begin{equation}
\begin{split}
&\frac{\P(E^{\delta,\rho}_l\cap\{{\cal X}^\circ_l=k\})}{\P(\Oi\cap \ball{x}{l}=\emptyset)}\\
&\quad \le \sum_{i=k}^{\lfloor\rho^{-1}k\rfloor}\binom{|\ball{x}{l}|}{i}
\sum_{j=0}^{i|\ball{x}{l^{1/2d}}|}\binom{i|\ball{x}{l^{1/2d}}|}{j}\left(\frac{1-p}{p}\right)^j\\
&\quad \le e^{Ck l^{1/2}}.
\end{split}
\label{case(2)-obs}
\end{equation}
Combining this with~\eqref{case(2)-rw2} and choosing $\delta$ small, we obtain 
\begin{equation}
\begin{split}
& \mu_N\left((S,\Oi)\in\left(\{(\bm{\sigma},\bm{\tau})\in \bm{F}_2\},(E^{\delta,\rho}_l(x)\cap \{|{\cal X}_l^\circ| =k\})\right)\right)\\
&\quad \le e^{-c\delta^{-1/d}l^{1/2}k/|\log\delta|}
\end{split}
\label{case(2)-2}
\end{equation}
for each $k < \rho\delta l^{d-1/2}$. 
Finally we sum~\eqref{case(2)-1} and~\eqref{case(2)-2} for $k\in\{1,2,\dotsc,\lfloor\rho\delta l^{d-1/2}\rfloor\}$ to obtain
\begin{equation}
\mu_N\left(E^{\delta,\rho}_l(x) \times \{(\bm{\sigma},\bm{\tau})\in \bm{F}_2\}\right)
\le e^{-c\delta^{1/d}l^{1/2}}.   
\end{equation}
%%%%%%%%%%%%%%%%%%%%%%%%%%%%%%%%%%%%%%%%%%%%%%%%%%%%%%%%%%%%%%%

\medskip
\noindent\underline{Case (3)}: In this case, we remove all the obstacles in $A(x;l/4,l)=\ball{x}{l}\setminus \ball{x}{l/4}$, change the obstacles configuration inside $\ball{x}{l/4}$ to typical configurations and force all the crossings to avoid $\ball{x}{l/4}$ after lengthening them by $l^2$. Complication arises when the origin is close to $\ball{x}{l/4}$ because then it costs a lot to force the first crossing to avoid $\ball{x}{l/4}$. We first deal with the simpler case $0\not\in\ball{x}{l/2}$ by applying the environment and path switching~\eqref{eq:switch} with
\begin{align}
&(A_1,A_2):=(\{(\bm{\sigma},\bm{\tau})\in \bm{F}_3\}, \{\tau_{\Oi}\wedge\tau_{\ball{x}{l/4}}>N\});\\
&(E_1,E_2):=(E^{\delta,\rho}(x),\{\Oi\cap A(x;l/4,l)=\emptyset\}).
%&(\Oi_1,\Oi_2):=(\Oi,(\Oi\setminus\ball{x}{l})\cup (\Oi'\cap\ball{x}{l/4})),
\end{align}
%where $\Oi'$ is an independent copy of $\Oi$.
The gain from the environment switching can be estimated as
\begin{equation}
\begin{split}
\frac{\P(E^{\delta,\rho}_l)}{\P(\Oi\cap A(x;l/4,l)=\emptyset)}
& \le \frac{\P\left(|\Oi\cap\ball{x}{l}|\le \delta|\ball{x}{l}|\right)}{\P(\Oi\cap A(x;l/4,l)=\emptyset)}\\
&\le p^{|\ball{x}{l/4}|}
\sum_{i=0}^{\delta |\ball{x}{l}|}\binom{|\ball{x}{l}|}{i}
\left(\frac{1-p}{p}\right)^{{i}}\\
& \le e^{-cl^d}
\end{split}
\label{GainInsideTypical}
\end{equation}
by using Stirling's approximation (or Cramer's large deviation principle as before).  

On the other hand, if we force the random walk to stay in $A(x;l/4,l)$ instead of $\ball{x}{l}$, the extra cost per step should be measured by the difference of the principal Dirichlet eigenvalues of the discrete Laplacian in $A(x;l/4,l)$ and $\ball{x}{l}$, which is of order $l^{-2}$. In fact, we will see in Lemma~\ref{lem:rwest} that uniformly in $u_k\in {\partial\ball{x}{l/2}}$ and $v_k\in \partial\ball{x}{l}$, 
\begin{equation}
\begin{split}
 p^{\ball{x}{l}\setminus \Oi}_{t_k-s_k}(u_k,v_k)
 &\le p^{\ball{x}{l}}_{t_k-s_k}(u_k,v_k)\\
 &\le e^{c_2((t_k-s_k)l^{-2}+1)} p^{A(x;l/4,l)}_{t_k-s_k+l^2}(u_k,v_k).
\end{split}
\label{aviod-center}
\tag{RW3}
\end{equation}
If $s_k<N$ and $t_k=N$ for some $k\in\N$, then this (last) crossing may be incomplete and its endpoint $v_k$ may be in $\ball{x}{l/4}$. In that case, the path switching should be done differently and we change the endpoint of the last crossing to $\tilde{v}_k:=v_k+(5l/8)\mathbf{e}_1$. The cost is bounded similarly as 
\begin{equation}
p^{\ball{x}{l}\setminus \Oi}_{t_k-s_k}(u_k,v_k) 
\le e^{c_2((t_k-s_k)l^{-2}+1)} p^{A(x;l/4,l)}_{t_k-s_k+2l^2}(u_k,\tilde{v}_k).
\label{switch-end}
\tag{RW4}
\end{equation}
We define $(\tilde{s}_k,\tilde{t}_k)_{k\ge 1}$ as the starting and ending times of switched crossings, similarly to Case~(1), and also
\begin{equation}
\begin{split}
& (\tilde{v}_k,\tilde{u}_{k+1})\\
&\quad :=
 \begin{cases}
 (v_k+(5l/8)\mathbf{e}_1,v_k+(5l/8)\mathbf{e}_1),&\textrm{if }s_k<N, t_k=N \textrm{ and }v_k \in \ball{x}{l/4},\\
 (v_k,u_{k+1}), &\textrm{otherwise.}  
 \end{cases}
\end{split}
\end{equation}
Then using the above estimates and recalling the definition of $\bm{F}_3$, we can bound the product in the right-hand side of \eqref{RW-decomp} by 
\begin{equation}
\begin{split}
&\prod_{k\ge 1} p^{\ball{x}{l}\setminus \Oi}_{t_k-s_k}(u_k,v_k)p^{\Z^d\setminus(\Oi\cup\ball{x}{l/2})}_{s_{k+1}-t_k}(v_k,u_{k+1})\\
&\quad \le \exp\left\{c_2\sum_{k\ge 1}\frac{t_k-s_k}{l^2}+ K(\bm{s},\bm{t})\right\}\\
&\qquad\times\prod_{k\ge 1} p^{A(x;l/4,l)}_{\tilde{t}_k -\tilde{s}_k}(\tilde{u}_k,\tilde{v}_k)p^{\Z^d\setminus(\Oi\cup\ball{x}{l/2})}_{\tilde{s}_{k+1}-\tilde{t}_k}(\tilde{v}_k,\tilde{u}_{k+1})\\
&\quad=e^{c\delta^{1/d}l^d}
\prod_{k\ge 1} p^{A(x;l/4,l)}_{\tilde{t}_k-\tilde{s}_k}(\tilde{u}_k,\tilde{v}_k)p^{\Z^d\setminus(\Oi\cup\ball{x}{l/2})}_{\tilde{s}_{k+1}-\tilde{t}_k}(\tilde{v}_k,\tilde{u}_{k+1}).
\end{split}
\label{F_3|stuv}
\end{equation}
Note that each $(\tilde{\bm{s}},\tilde{\bm{t}})$ has pre-image of cardinality at most $2^{2\delta^{1/d}l^d}$. Therefore summing~\eqref{F_3|stuv} over $(\bm{s},\bm{t},\bm{u},\bm{v})$ separately according to the number of crossings as in Case (1), we can obtain
\begin{equation}
\begin{split}
&\bP(\tau_\Oi>N\textrm{ and }(\bm{\sigma},\bm{\tau})\in \bm{F}_3)\\
&\quad =\sum_{(\bm{s},\bm{t})\in \bm{F}_3}\sum_{\bm{u},\bm{v}} 
p^{{{\Z^d}\setminus (\Oi\cup\ball{x}{l/2})}}_{s_1}(0,u_1)
 \prod_{k\ge 1} p^{\ball{x}{l}\setminus \Oi}_{t_k-s_k}(u_k,v_k)p^{\Z^d\setminus(\Oi\cup\ball{x}{l/2})}_{s_{k+1}-t_k}(v_k,u_{k+1})\\
&\quad \le  e^{c\delta^{1/d}l^d}
\sum_{(\tilde{\bm{s}},\tilde{\bm{t}})}\sum_{\tilde{\bm{u}},\tilde{\bm{v}}} 
 p^{{{\Z^d}\setminus (\Oi\cup\ball{x}{l/2})}}_{\tilde{s}_1}(0,u_1)
\prod_{k\ge 1} p^{\ball{x}{l}}_{\tilde{t}_k-\tilde{s}_k}(\tilde{u}_k,\tilde{v}_k)p^{\Z^d\setminus(\Oi\cup\ball{x}{l/2})}_{\tilde{s}_{k+1}-\tilde{t}_k}(\tilde{v}_k,\tilde{u}_{k+1})\\
&\quad\le CN e^{c\delta^{1/d}l^d}\bP(\tau_{\Oi\cup\ball{x}{l/4}}>N)
\end{split}
\end{equation}
uniformly in $\Oi$ in the case $0\not\in\ball{x}{l/2}$. Recalling~\eqref{GainInsideTypical} and $l\ge \log N$ and using~\eqref{eq:switch}, we conclude that in this case
\begin{equation}
\mu_N\left((S,\Oi)\in \left(\{(\bm{\sigma},\bm{\tau})\in \bm{F}_3\},E^{\delta,\rho}_l(x)\right)\right)
\le e^{-c l^d}.
\label{prob-F_3}
\end{equation}

Finally, we deal with the case $0\in\ball{x}{l/2}$. In this case, the starting point of first crossing may be close to (or even inside) $\ball{x}{l/4}$ and we want to ensure that the random walk gets away from that ball quickly. To this end, we fix a path $\pi(x;l)\subset \ball{x}{l/2}$ of length $l$ from $0$ to $n\mathbf{e}_1\in\partial\ball{x}{l/2}$ ($n\in\N$) and modify the environment and path switching as follows (see Figure~\ref{fig:case(3)}): 
\begin{align}
&(A_1,A_2):=\left(\{(\bm{\sigma},\bm{\tau})\in \bm{F}_3\}, \left\{S_{[0,l]}=\pi(x;l),\tau_{\Oi}\wedge(l+\tau_{\ball{x}{l/4}}\circ\theta_l)>N\right\}\right);\\
&(E_1,E_2):=(E^{\delta,\rho}(x),\{\Oi\cap (A(x;l/4,l)\cup \pi(x;l))=\emptyset\}),
\end{align}
where $l+\tau_{\ball{x}{l/4}}\circ\theta_l$ is the first hitting time to $\ball{x}{l/4}$ after time $l$. Let us explain the difference from the previous case {$0\not\in\ball{x}{l/2}$}. For the environment, we need to keep $\pi(x;l)$ empty, which has a cost of $e^{-cl}$, but is negligible compared with the original gain $e^{cl^d}$ in~\eqref{GainInsideTypical}. 

For the random walk, {only the first crossing, that is $S_{[0,\tau_1]}$ in this case, is switched differently}. 
%, we switch the $S_{[0,\tau_1]}$ which corresponds to the following part in~\eqref{RW-decomp}:
%\begin{equation}
%\sum_{s_1,u_1}p_{s_1}^{\ball{x}{l}\setminus\Oi}(0,u_1)p_{t_1-s_1}^{\ball{x}{l}\setminus\Oi}(u_1,v_1)
%\le p_{t_1}^{\ball{x}{l}}(0,v_1). 
%\end{equation}
In the present case, note that $v_1\in\partial\ball{x}{l}$ since by the total duration constraint on $\bm{F}_3$, we have $t_1\le 2\delta^{1/d}c_0^2 l^{d+2}<N$ for small $\delta$. We {switch the paths with $\tau_1=t_1, S_{\tau_1}=v_1$ to those} go from 0 to $n\mathbf{e}_1$ following $\pi(x;l)$ {in $l$ steps} and then to go from $n\mathbf{e}_1$ to $v_1$ inside $A(x;l/4,l)$ in {$t_1+2l^2$ steps} afterward. {The probability to follow $\pi(x;l)$ in the first $l$ steps is $(2d)^{-l}$ and combining this with the estimate (see Lemma~\ref{lem:rwest})}
\begin{equation}
 p_{t_1}^{\ball{x}{l}}(0,v_1)\le e^{c_2(\lfloor t_1l^{-2}\rfloor+1)}p_{t_1+2l^2}^{A(x;l/4,l)}(n\mathbf{e}_1,v_1),
\label{switch-start}
\tag{RW5}
\end{equation}
{we obtain the following bound on the switching cost of the first crossing:}
\begin{equation}
\begin{split}
p_{t_1}^{\ball{x}{l}}(0,v_1)
&\le {(2d)^l}
p_l^{\pi(x;l)}(0,n\mathbf{e}_1) e^{c_2(\lfloor t_1l^{-2}\rfloor+1)}
p_{t_1+2l^2}^{A(x;l/4,l)}(n\mathbf{e}_1,v_1)\\
&\le {(2d)^l}e^{c_2(\lfloor t_1l^{-2}\rfloor+1)}
p_{t_1+2l^2+l}^{A(x;l/4,l)\cup \pi(x;l)}(0,v_1).
\end{split}
\label{first-cross}
\end{equation}
{Note that the term $c_2\lfloor t_1l^{-2}\rfloor$ already appeared in~\eqref{F_3|stuv}. The extra cost of $(2d)^l$} is again negligible compared with {the} $e^{cl^d}$ gain from the environment. 

Therefore simply by {setting} $\tilde{t}_1:=t_1+2l^2+l$ {and changing $(\tilde{s}_k,\tilde{t}_k)_{k\ge 2}$ accordingly}, we can {use~\eqref{first-cross} to} follow the same argument as before to extend~\eqref{prob-F_3} to the case {$0\in\ball{x}{l/2}$}. 
\qed

\begin{figure}[h]
 \includegraphics[width=70mm]{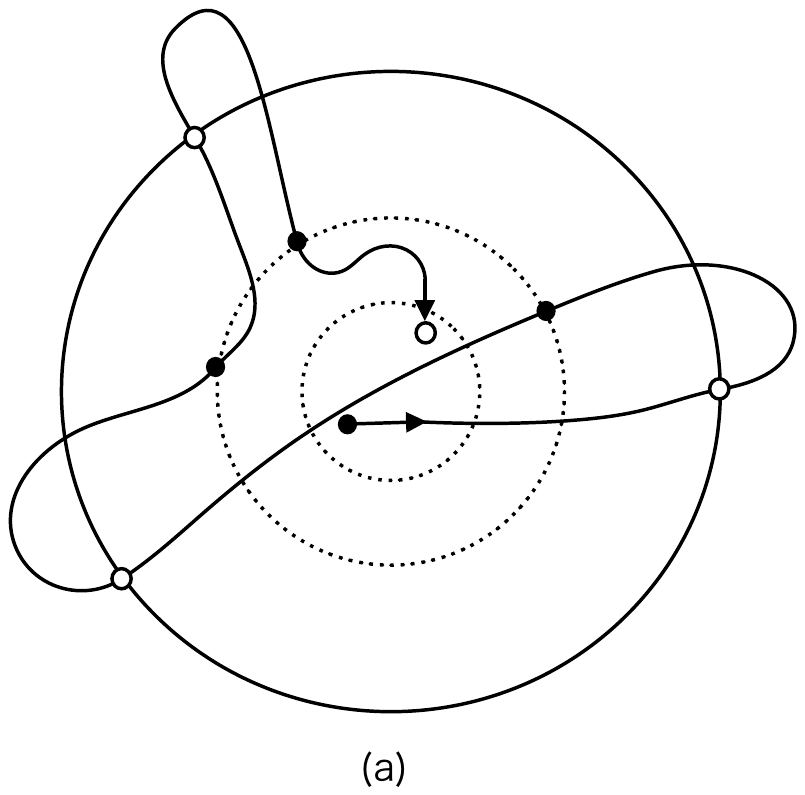}\hspace{10mm}
 \includegraphics[width=70mm]{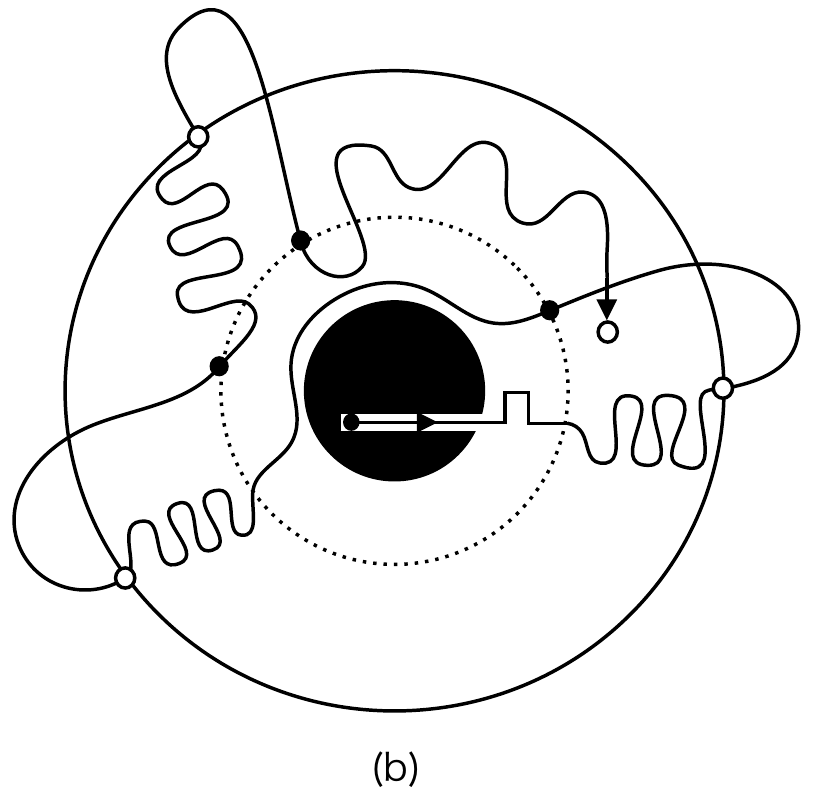} 
\begin{caption}
{A schematic figure of the switching configuration from (a) to (b) in Case (3). The balls are $\ball{x}{l/4}$, $\ball{x}{l/2}$ and $\ball{x}{l}$ from inside. There are 4 crossings from $\ball{x}{l/2}$ to $\ball{x}{l}^c$, including the last incomplete crossing. Both $S_0=0$ and $S_N$ are in $\ball{x}{l/4}$ in (a). The paths from $\bullet$ to $\circ$ are crossings and are lengthened. Observe that we cannot make the second crossing avoid $\ball{x}{l/4}$ without lengthening it as illustrated in (b). The paths from $\circ$ to $\bullet$ are unchanged. The first polygonal segment of the path in (b) represents $\pi(x;l)$.}
\label{fig:case(3)}
\end{caption}
\end{figure}

\subsection{Random walk estimates}\label{SS:tech-proof}
In this subsection, we restate and prove the random walk estimates used in the proof of Lemma~\ref{lem:odensity}. Recall {the notation $\overline{\ball{x}{l}}=\ball{x}{l}\cup\partial\ball{x}{l}$ and} that $p_n(u,v)$ is understood to be $p_{n+1}(u,v)$ if $n+|u+v|_1$ is odd by the convention introduced in Remark~\ref{rem:parity}. 
\bl\label{lem:rwest}
Let ${\cal X}_l^\circ(x)$ and $\Gamma$ be defined as in~\eqref{def:calXo} and~\eqref{eq:Gamma}, respectively. There exist $c_0, c_1, c_2>0$ such that the following hold for all sufficiently large $l$:
\begin{enumerate}
\item Uniformly in $m \ge (c_0l)^2$, {$u \in \ball{x}{l/2}$, $v \in\overline{\ball{x}{l}}$, }
\begin{align}
p^{\ball{x}{l} \setminus {\Oi}}_m(u, v)  
\leq e^{-c_1 (\lfloor m /(c_0l)^2\rfloor{-1}) \Gamma(|{\cal X}_l^\circ|)}p_m^{\ball{x}{l}}(u,v). 
\tag{RW1}
\label{lem:remove-O}
\end{align}
\item Uniformly in $m\leq (c_0l)^2$ and $u\in{\overline{\ball{x}{l/2}}}, v\in \partial\ball{x}{l}$,
\begin{align}
p_m^{\ball{x}{l}}(u, v) 
\leq \frac{1}{100}p_{{m+l^2}}^{\ball{x}{l}}(u,v).
\tag{RW2}
\label{lem:too-short}
\end{align}
\item Uniformly in $m>0$ and $u\in\partial\ball{x}{l/2}, v\in \partial \ball{x}{l}$, 
\begin{align}
p_m^{\ball{x}{l}}(u, v) 
\leq e^{c_2 (ml^{-2}+1)}p_{{m+l^2}}^{A(x;l/4,l)}(u,v).
\label{lem:avoid-center}
\tag{RW3}
\end{align}
\item Uniformly in $m>0$ and $u\in\partial\ball{x}{l/2}, v\in \ball{x}{l/4}$, 
\begin{align}
p_m^{\ball{x}{l}}(u, v)  
\leq e^{c_2 (ml^{-2}+1)}p_{{m+2l^2}}^{A(x;l/4,l)}(u,v+(5l/8)\mathbf{e}_1).
\label{lem:switch-end}
\tag{RW4}
\end{align}
\item Suppose $0\in\ball{x}{l/2}$ and let $n\in\N$ be such that $n\mathbf{e}_1\in\partial\ball{x}{l/2}$. Then uniformly in $m>0$, and $v\in \partial \ball{x}{l}$, 
\begin{align}
p_m^{\ball{x}{l}}(0, v)  
\leq e^{c_2 (ml^{-2}+1)}p_{{m+{2l^2}}}^{A(x;l/4,l)}(n\mathbf{e}_1,v).
\label{lem:switch-start}
\tag{RW5}
\end{align}
\end{enumerate}
\el
In the proof of this lemma, we will use the following estimate on the Dirichlet heat kernel: for any $c\in(0,1)$, there exists $C>0$ such that uniformly in $r\in\N$, $k\in[cr^2,r^2/c]$ and $u,w\in \ball{x}{r}$, 
\begin{equation}
C\le\frac{p_k^{\ball{x}{r}}(u,w)}{r^{-d-2}{\rm dist}(u,\partial\ball{x}{r}){\rm dist}(w,\partial\ball{x}{r})}
\le \frac{1}{C}.\label{LL10}
\end{equation}
This can be found for example in~\cite[Proposition~6.9.4]{LL10}, where it is stated only for the case $k=r^2$ but the argument therein can easily be adapted to the above uniform estimate.

The first assertion~\eqref{remove-O} will be a direct consequence of the following lemma.

\bl\label{lem:calXo}
For any $c_0\in(0,1)$, there exists $c_1>0$ independent of $\Oi$ such that for any $l\in\N$ and $u,w\in\ball{x}{l}$,
\begin{equation}
p^{\ball{x}{l}\setminus \Oi}_{(c_0l)^2}(u,w)
\le e^{-c_1\Gamma(|\mathcal{X}_l^\circ|))} p^{\ball{x}{l}}_{(c_0l)^2}(u,w).
\label{eq:calXo}
\end{equation}
\el
\begin{remark}
 This lemma holds for arbitrary $c_0\in(0,1)$. In Lemma~\ref{lem:rwest}, it is only~\eqref{lem:too-short} which imposes a restriction on $c_0$. 
\end{remark}
%%%%%%%%%%%%%%%%%%%%%%%%%%%%%%%%%%%%%%%%%%%%%%%%%%%%%%%%%%%%%

\noindent{\bf Proof of Lemma \ref{lem:calXo}.}
We write $l_0$ for $c_0l$ in this proof to ease the notation. It suffices to prove
\be\label{eq:calXo2}
\min_{u, {w}\in \ball{x}{l}} \bP_u\big(\tau_\Oi \leq l_0^2 \,\big|\, S_{l_0^2}={w}, \tau_{\ball{x}{l}^c}>l_0^2\big)
\geq c_1\Gamma(|{\cal X}_l^\circ|).
\ee
for some $c_1>0$, since $1-\lambda \le e^{-\lambda}$. 
The proof of this relies on the so-called \emph{second moment method}. Let us introduce
\begin{equation}
 T:=\sum_{m\in [{l_0^2}/{4},{3l_0^2}/{4}]}
1_{\{S_m\in \mathcal{X}_l^\circ\}}.
\end{equation}
We will show the following:
\begin{align}
&\bE_u\left[T \relmiddle| S_{l_0^2}={w}, \tau_{\ball{x}{l}^c}>l_0^2\right]
\ge c|\mathcal{X}_l^\circ|l_0^{2-d},\label{1stmoment}\\
& \bE_u\left[T^2 \relmiddle| S_{l_0^2}={w}, \tau_{\ball{x}{l}^c}>l_0^2\right]
\le {\frac{C(|\mathcal{X}_l^\circ|l_0^{2-d})^2}{\Gamma(|\mathcal{X}_l^\circ|)}}.
\label{2ndmoment}
\end{align}
Given these bounds, the desired bound follows via the Paley--Zygmunnd inequality as
\begin{equation}
\begin{split}
&\bP_u\left(\tau_\Oi\le l_0^2 \relmiddle| S_{l_0^2}={w}, \tau_{\ball{x}{l}^c}>l_0^2\right)\\
&\quad \ge \bP_u\left(T>0 \relmiddle| S_{l_0^2}={w}, \tau_{\ball{x}{l}^c}>l_0^2\right)\\
&\quad \ge \frac{\bE_u\left[T \relmiddle| S_{l_0^2}={w}, \tau_{\ball{x}{l}^c}>l_0^2\right]^2}
 {\bE_u\left[T^2 \relmiddle| S_{l_0^2}={w}, \tau_{\ball{x}{l}^c}>l_0^2\right]}.
\end{split}
\end{equation}
\noindent\underline{First moment~\eqref{1stmoment}}:
Note first that uniformly in $m\in [{l_0^2}/{4},{3l_0^2}/{4}]$, $z\in \ball{x}{l/2}$ and $u,{w}$ as in the statement,
\begin{equation}
\begin{split}
 \bP_u\left(S_m=z,S_{l_0^2}={w}, \tau_{\ball{x}{l}^c}>l_0^2\right)
&= p_m^{\ball{x}{l}}(u,z)p_{l_0^2-m}^{\ball{x}{l}}(z,{w})\\
&\ge \frac{c}{l_0^d}p^{\ball{x}{l}}_{l_0^2}(u,{w}).
 \label{visit-in-mid}
\end{split}
\end{equation}
Indeed, {by~\eqref{LL10},} there exists ${C}>0$ such that uniformly in $m,u,{w}$ and $z$ as above, 
\begin{align}
p_m^{\ball{x}{l}}(u,z)&\ge \frac{C}{l_0^{d+1}}{{\rm dist}(u,\partial\ball{x}{l})},\\
{p^{\ball{x}{l}}_{l_0^2-m}(z,w)}&{\ge \frac{C}{l_0^{d+1}}{\rm dist}(w,\partial\ball{x}{l})},
\end{align}
and
\begin{equation}
{p^{\ball{x}{l}}_{l_0^2}(u,{w})\le \frac{1}{Cl_0^{d+2}}{\rm dist}(u,\partial\ball{x}{l}){\rm dist}(w,\partial\ball{x}{l}).}
\end{equation}
{The bound~\eqref{visit-in-mid} follows from these three bounds.} Summing~\eqref{visit-in-mid} over $m\in [{l_0^2}/{4},{3l_0^2}/{4}]$ and $z\in \mathcal{X}_l^\circ$ yields the following equivalent of~\eqref{1stmoment}: 
\begin{equation}
\begin{split}
&\bE_u\left[T\colon S_{l_0^2}={w}, \tau_{\ball{x}{l}^c}>l_0^2\right]\\
&\quad \ge c|\mathcal{X}_l^\circ|l_0^{2-d}\bP_u\left(S_{l_0^2}={w}, \tau_{\ball{x}{l}^c}>l_0^2\right).
\end{split}
\end{equation}
\noindent\underline{Second moment~\eqref{2ndmoment}}:
We begin with
\begin{equation}
\label{binomial}
\begin{split}
&\bE_u \left[T^2\colon S_{l_0^2} = {w},\tau_{\ball{x}{l}^c}>l_0^2\right]\\
&=2\sum_{l_0^2/4 \leq i\leq j\le 3l_0^2/4}\sum_{z_1,z_2 \in \mathcal{X}_l^\circ}\bP_u\left(S_{i} = z_1,
S_{j} = z_2,S_{l_0^2} = {w},\tau_{\ball{x}{l}^c}>l_0^2\right)\\
&= 2\sum_{\substack{ k, m\ge l_0^2/4,\\k+m \le 3l_0^2/4}}\sum_{z_1,z_2 \in \mathcal{X}_l^\circ}p^{\ball{x}{l}}_k(u,z_1)p^{\ball{x}{l}}_{m}(z_1,z_2) p^{\ball{x}{l}}_{l_0^2-k-m}(z_2,{w}).\\
\end{split}
\end{equation}
It follows from{~\eqref{LL10} as before} that for the parameters appearing above, 
\begin{equation}
 p^{\ball{x}{l}}_k(u,z_1)p^{\ball{x}{l}}_{l_0^2-k-m}(z_2,v)\le \frac{c}{l_0^d} p^{\ball{x}{l}}_{l_0^2}(u,v).
\end{equation}
Substituting this into~\eqref{binomial} and performing the summation over $z_2$ and $k$, we find that
\begin{equation}
\begin{split}
 &\bE_u \left[T^2\colon S_{l_0^2} = v,\tau_{\ball{x}{l}^c}>l_0^2\right]\\
 &\quad\le \frac{c}{l_0^{d-2}} p^{\ball{x}{l}}_{l_0^2}(u,v)
 \sum_{m\le l_0^2/2}\sum_{z \in \mathcal{X}_l^\circ}
 \bP_z(S_m \in \mathcal{X}_l^\circ).
\end{split}
\label{XtoX}
\end{equation}
For the probability appearing in the summation, we claim
\begin{equation}
\sup_{z \in \mathcal{X}_l^\circ} \bP_z(S_m \in \mathcal{X}_l^\circ) 
\le c
\begin{cases}
 m^{-d/2},&m < l_0^{1/d},\\
 l_0^{-1/2}, &m\in [{l_0}^{1/d},|\mathcal{X}_l^\circ|^{2/d}{l_0}^{1/d}],\\
 |\mathcal{X}_l^\circ| m^{-d/2}, &m\in (|\mathcal{X}_l^\circ|^{2/d}{l_0}^{1/d}, l_0^2/2].
\end{cases}
\label{fall-in-X}
\end{equation}
Substituting this bound {into}~\eqref{XtoX}, we obtain
\begin{equation}
 \bE_u\left[T^2 \colon S_{l_0^2}=v, \tau_{\ball{x}{l}^c}>l_0^2\right]
\le {\frac{C(|\mathcal{X}_l^\circ|l_0^{2-d})^2}{\Gamma(|\mathcal{X}_l^\circ|)}}
p^{\ball{x}{l}}_{l_0^2}(u,v)
\end{equation}
which is equivalent to~\eqref{2ndmoment}. 

It remains to show~\eqref{fall-in-X}. First, the \emph{on-diagonal} term $\bP_z(S_m=z)\le cm^{-d/2}$ is always smaller than the right-hand side of~\eqref{fall-in-X}. Henceforth, we shall focus on the points $w\in\mathcal{X}_l^\circ\setminus\{z\}$ which are at least $l^{1/2d}$ away from $z$. By a standard Gaussian estimate on the transition probability of the symmetric simple random walk~\cite[Theorem~6.28]{B17},
\begin{equation}
\begin{split}
 & \sum_{w\in\mathcal{X}_l^\circ\setminus\{z\}}\bP_z(S_m=w)\\
 &\quad \le \sum_{n=1}^{\infty} \sum_{w\in \mathcal{X}_l^\circ\cap A(z;n l^{1/2d},(n+1)l^{1/2d})}\frac{C}{m^{d/2}}\exp\left\{-\frac{|w-z|^2}{Cm}\right\}.
\end{split}
\end{equation}
This right-hand side is maximized when {the} annuli are filled from inside out but since the points in $\mathcal{X}_l^\circ$ are at least $l^{1/2d}$ away from each other, the $n$-th annulus contains at most $Cn^{d-1}$ points. This leads us to the bound
\begin{equation}
 \sum_{w\in\mathcal{X}_l^\circ\setminus\{z\}}\bP_z(S_m=w)
 \le \frac{C}{m^{d/2}} \sum_{n=1}^{C|\mathcal{X}_l^\circ|^{1/d}} 
n^{d-1}\exp\left\{-\frac{n^2l^{1/d}}{Cm}\right\}.  
\end{equation}
The desired bound~\eqref{fall-in-X} follows by a simple computation considering the cases $m<l^{1/d}$, $m\in [l^{1/d},|\mathcal{X}_l^\circ|^{2/d}l^{1/d}]$ and $m> |\mathcal{X}_l^\circ|^{2/d}l^{1/d}$ separately. \qed

\noindent{\bf Proof of Lemma \ref{lem:rwest}.}
{We only consider the case when $l, 5l/8$ and $n$ are all even.} 

The first assertion~\eqref{lem:remove-O} follows immediately from Lemma~\ref{lem:calXo}. Indeed, using \eqref{eq:calXo} in the Chapman--Kolmogorov identity, we have
\begin{equation}
\begin{split}
 p^{\ball{x}{l} \setminus \Oi}_m(u, v)
& =\sum_{w\in\ball{x}{l}} 
 p^{\ball{x}{l} \setminus \Oi}_{(c_0l)^2}(u,w)
 p^{\ball{x}{l} \setminus \Oi}_{m-(c_0l)^2}(w, v)\\
& \le e^{-c\Gamma(|\mathcal{X}_l^\circ|)}
 \sum_{w\in\ball{x}{l}} 
 p^{\ball{x}{l}}_{(c_0l)^2}(u,w)
 p^{\ball{x}{l} \setminus \Oi}_{m-(c_0l)^2}(w,v)
\end{split}
\end{equation}
and~\eqref{remove-O} follows by iteration. 

Let us proceed to prove the second assertion~\eqref{lem:too-short}. Note first that by a standard Gaussian heat kernel bound~\cite[Theorem~6.28]{B17}, for any $u,w\in \ball{x}{l}$ with $|u-w|\ge cl$, 
\begin{equation}
\begin{split}
p_k^{\ball{x}{l}}(u,w)
&\le p_k^{\Z^d}(u,w)\\
&\le C k^{-d/2}\exp\left\{-\frac{|u-w|^2}{C k}\right\}\\
&\le Cl^{-d} \left(\frac{l^2}{k}\right)^{d/2}\exp\left\{-c\frac{l^2}{k}\right\}.
\end{split}
\label{Gaussian-upper}
\end{equation}

On the other hand, for $m\in[l^2,2l^2]$, {$u\in\overline{\ball{x}{l/2}}$ and $w\in\partial\ball{x}{3l/4}$}, we have
%\begin{equation}
% \sum_{z\in\ball{w}{l/10}}p^{\ball{x}{l}}_m(u,z)
%=P_u\left(S_m\in \ball{w}{l/10}, \tau_{\ball{x}{l}^c}>m\right)
%\ge c>0
%\end{equation}
%by Donsker's invariance principle applied to $(l^{-1}S_{l^2k})_{k\ge 1}$. (Alternatively, one can prove it by the parabolic Harnack inequality and the chaining argument.) 
%By the parabolic Harnack inequality the values of heat kernel on $z\in\ball{w}{l/10}$ are comparable, that is, for $m\in[l^2,2l^2]$, 
%\begin{equation}
% \frac{\max_{z\in\ball{w}{l/10}}p^{\ball{x}{l}}_m(u,z)}
%{\min_{z\in\ball{w}{l/10}}p^{\ball{x}{l}}_m(u,z)}\le C
%\end{equation}
%and therefore we get the following uniform lower bound
\begin{equation}
p^{\ball{x}{l}}_m(u,w)
%\ge \min_{z\in\ball{w}{l/10}}p^{\ball{x}{l}}_m(u,z)
\ge {Cl^{-d}}
\label{uniform-lower}
\end{equation}
{by~\eqref{LL10}.}
Combining this with~\eqref{Gaussian-upper}, we obtain the comparison
\begin{equation}
\max_{k\le (c_0l)^2}
{\max_{u\in\overline{\ball{x}{l/2}},w\in\partial\ball{x}{3l/4}}}
\frac{p^{\ball{x}{l}}_k(u,w)}{p^{\ball{x}{l}}_{k+l^2}(u,w)}\le \frac{1}{100}
\label{too-short2}
\end{equation}
for sufficiently small $c_0$. Suppose $m\le (c_0l)^2$ for $c_0$ satisfying the above. By decomposing the random walk path upon the last visit to $\partial\ball{x}{3l/4}$, we get 
\begin{equation}
\begin{split}
 p^{\ball{x}{l}}_m(u,v)
& =\sum_{k=1}^m\sum_{w\in\partial\ball{x}{3l/4}} p_k^{\ball{x}l}(u,w)
 p_{m-k}^{A(x;3l/4,l)}(w,v)\\
& \le \frac{1}{100}
\sum_{k=1}^m\sum_{w\in\partial\ball{x}{3l/4}} p_{k+l^2}^{\ball{x}l}(u,w)
 p_{m-k}^{A(x;3l/4,l)}(w,v)\\
&= \frac{1}{100} p^{\ball{x}{l}}_{m+l^2}(u,v).
\end{split}
\end{equation}
This concludes the proof of~\eqref{too-short}. 

The proofs of~\eqref{lem:avoid-center}--\eqref{lem:switch-start} rely on the fact that there exists $c_2>0$ such that for any $m\in\N$, $u\in{\partial\ball{x}{l/2}}$ and $w\in A(x;3l/8,7l/8)$, 
\begin{equation}
p^{A(x;l/4,l)}_{m+l^2}(u,w)
 \ge \exp\{-c_2(\lfloor ml^{-2}\rfloor+1)\}p_m^{\ball{x}{l}}(u,w).
\label{avoid-center2}
\end{equation}
In order to prove this, we bound the left-hand side from below by the probability that $S_{[0,l^2]}\subset A(x;l/4,l)$, $S_{[l^2,m+l^2]}\subset\ball{w}{l/8}$ and $S_{kl^2}\in\ball{w}{l/16}$ for each $k\in\{1,2,\dotsc,\lfloor ml^{-2}\rfloor\}$, which can be written as
\begin{equation}
\begin{split}
\sum_{z_1,\dotsc,z_{\lfloor ml^{-2}\rfloor}\in\ball{w}{l/16}}
p_{l^2}^{A(x;l/4,l)}(u,z_1)
\left(\prod_{j=1}^{\lfloor ml^{-2}\rfloor-1}p_{l^2}^{\ball{w}{l/8}}(z_j,z_{j+1})\right)\\
\times p_{m-\lfloor ml^{-2}\rfloor l^2+l^2}^{\ball{w}{l/8}}(z_{\lfloor ml^{-2}\rfloor},w).
\end{split}
\end{equation}
Since $m-\lfloor ml^{-2}\rfloor l^2+l^2\in[l^2,2l^2]$ and all the points $u$, $z_j$'s and $w$ are at least $l/16$ away from the corresponding boundaries, by~\eqref{LL10}, all the heat kernels appearing in this expression is bounded from below by $Cl^{-d}$ regardless where $z_j$'s are in $\ball{w}{l/16}$. Therefore we find the bound
\begin{equation}
p^{A(x;l/4,l)}_{m+l^2}(u,w)
 \ge \exp\{-c_2(\lfloor ml^{-2}\rfloor+1)\}l^{-d} 
\end{equation}
for some $c_2>0$. Recalling~\eqref{Gaussian-upper}, we have $p_m^{\ball{x}{l}}(u,w)\le cl^{-d}$ and we conclude the proof of~\eqref{avoid-center2}.

Given~\eqref{avoid-center2}, the third bound~\eqref{lem:avoid-center} can be proved in the same way as in the {proof of~\eqref{lem:too-short}} via the last visit decomposition. In order to prove the {bound}~\eqref{lem:switch-end}, we first replace $m$ by $m+l^2$ and choose $u\in \partial\ball{x}{l/2}$ and $w=v+(5l/8)\mathbf{e}_1$ for $v\in\ball{x}{l/4}$ in~\eqref{avoid-center2} {to obtain}
\begin{equation}
\begin{split}
&\exp\{c_2(\lfloor ml^{-2}\rfloor+1)\} p^{A(x;l/4,l)}_{m+l^2}(u,v+(5l/8)\mathbf{e}_1) \\
&\quad \ge  p^{\ball{x}{l}}_{m+l^2}(u,v+(5l/8)\mathbf{e}_1).
%&\quad \ge cp^{\ball{x}{l}}_m(u,v).
\end{split}
\end{equation}
We can further bound the right-hand side from below by $cp^{\ball{x}{l}}_m(u,v)$ using either~\eqref{too-short2} ($m\le (c_0l)^2$) or the parabolic Harnack inequality from~\cite{D99} ($m>(c_0l)^2$). This yields~\eqref{lem:switch-end} by making $c_2$ larger. The proof of~\eqref{lem:switch-start} is almost the same and left to the reader. 
 \qed

%%%%%%%%%%%%%%%%%%%%%%%%%%%%%%%%%%%%%%%%%%%%%%%%%%%%%%%%%%%%%

\section{Random Walk Range and ``Truly''-Open Sites}\label{S:Topen}
In this section, we prove various properties of $\cal T$ and its relation with the random walk range $S_{[0,N]}$, which will pave the way for the proof of Theorems~\ref{th:cover} and~\ref{th:Tiboundary}. First, we prove Lemma~\ref{lem:TinS} in Subsection~\ref{SS:TinS}, which shows that the random walk must visit the interior of $\Ti$, and sites in $\Ti$ are well-approximated by sites in $S_{[0,N]}$. We then explain in Subsection~\ref{SS:Tiboundaryw} how Lemma~\ref{lem:SinTi}, Proposition~\ref{prop:Tiboundaryw}, and the deduction of Theorem~\ref{th:Tiboundary} from Theorem~\ref{th:Ticover} all follow from the same key Lemma~\ref{lem:Glbd} on the probability of visiting certain sites that are costly for survival. The proof of Lemma~\ref{lem:Glbd} is then given in Subsection~\ref{SS:Glbd-proof} using path decomposition and switching, with the basic setup presented earlier in Section~\ref{SS:path-decomp}.
%
%In this section, we prove various properties of $\Ti$ except for Theorem~\ref{th:Ticover}. Let us recall the statements.\footnote{Is this really necessary?} We first prove Lemma~\ref{lem:TinS} which states that
%\begin{equation}
% \lim_{N\to\infty} \mu_N \left(S_{[0,N]} \supset \left\{x\in\Ti\colon {\rm dist}(x,\partial\Ti)\ge{(\log N)^3}\right\}\right) = 1
%\tag{\ref{eq:TinS}}
%\end{equation}
%and
%\begin{equation}
%\lim_{N\to\infty} \mu_N \left(\Ti \subset \left\{x\in\Z^d\colon {\rm dist}(x,S_{[0,N]})\le(\log N)^5\right\}\right) = 1.
%\tag{\ref{eq:Ti2S}}
%\end{equation}
%Then we prove Lemma~\ref{lem:SinTi} and Proposition~\ref{prop:Tiboundaryw}, which are
%\begin{equation}
% \tag{\ref{eq:SinTi}}
%\lim_{N\to\infty} \mu_N \big(S_{[0,N]} \subset \Ti) = 1,
%\end{equation}
%and
%\begin{equation}
%  \tag{\ref{eq:Tiboundaryw}}
% \lim_{N\to\infty} \mu_N \big(|\partial \Ti| \leq \varrho_N^{d-1+b}\big) = 1,
%\end{equation}
%respectively. The proofs of these two statements and also the derivation of Theorem~\ref{th:Tiboundary} from Theorem~\ref{th:Ticover} have much in common and we summarize the key technical part as a single lemma (Lemma~\ref{lem:Glbd}). 

\subsection{Proof of Lemma~\ref{lem:TinS}}
\label{SS:TinS}
In this subsection, we give the proof of Lemma~\ref{lem:TinS}. First we show that ``truly''-open sites are rare in the following sense.
\bl[``truly''-open sites are rare] \label{lem:trare}
For any $v\in \Z^d$  
and all sufficiently large $N$, 
\be
\P(v \mbox{ is ``truly''-open}) \leq \exp\left\{- (\log N)^2\right\}.
\ee
\el
\noindent{\bf Proof of Lemma~\ref{lem:trare}.}
Recall that $x\in \Z^d$ is ``truly''-open if
\begin{equation}
 \bP_x\left(\tau_\Oi > (\log N)^5\right) \geq \exp\left\{-(\log N)^2\right\}.
\end{equation}
By using Donsker--Varadhan's asymptotics~\eqref{eq:surv}, we obtain
\begin{equation}
\begin{split}
&  \P\left(\bP_x\left(\tau_\Oi > (\log N)^5\right) \geq \exp\left\{-(\log N)^2\right\}\right)\\
 &\quad \le \exp\left\{(\log N)^2\right\}\E\left[\bP\left(\tau_{\Oi}>(\log N)^5\right)\right]\\
 &\quad \le \exp\left\{(\log N)^2\right\} \exp\left\{-c(\log N)^{\frac{5d}{d+2}}\right\}.
\end{split}
\end{equation}
Since the power of $5d/(d+2)>2$ for $d\ge 2$, we are done.
\qed

\noindent\textbf{Proof of Lemma~\ref{lem:TinS}.}
The proof of the first assertion~\eqref{eq:TinS} is simple. Indeed, since a ``truly''-open site is open by definition, for any site in
\begin{equation}
 \left\{x\in\Ti\colon {\rm dist}(x,\partial\Ti)\ge(\log N)^5\right\}, 
\end{equation}
its $(\log N)^5$ neighborhood is free of obstacles. Therefore~\eqref{eq:TinS} follows from Lemma~\ref{lem:hit-centers}. 

The second assertion~\eqref{eq:Ti2S} can be restated as
\begin{equation}
\lim_{N\to\infty} \mu_N \left(\bigcap_{w\in\Ti}\{ \tau_{\ball{w}{(\log N)^5}}\le N\}\right) =1.
\label{eq:Ti2S-restate}
\end{equation}
We are going to show that for any $w\in\ball{0}{3\varrho_N}$, 
\begin{equation}
\mu_N\left(w\textrm{ is ``truly''-open}, \tau_{\ball{w}{(\log N)^5}}> N\right)
\le\exp\left\{-\left(\log N\right)^2\right\}, 
\end{equation} 
from which~\eqref{eq:Ti2S-restate} follows by the union bound. But since whether $w$ is ``truly''-open or not depends only on the configuration of obstacles inside $\ball{w}{(\log N)^5}$ and hence is independent of $\bP(\tau_{\ball{w}{(\log N)^5}}\wedge\tau_\Oi> N)$, we have  
\begin{equation}
\begin{split}
&\P\otimes\bP\left(w\textrm{ is ``truly''-open}, \tau_{\ball{w}{(\log N)^5}}\wedge\tau_\Oi> N\right)\\
&\quad =\P(w\textrm{ is ``truly''-open})\E\left[\bP\left(\tau_{\ball{w}{(\log N)^5}}\wedge\tau_\Oi> N\right)\right]\\
&\quad \le \exp\left\{-(\log N)^2\right\}\P\otimes\bP(\tau_\Oi>N)
\end{split}
\end{equation}
by using Lemma~\ref{lem:trare}.
\qed

\subsection{Proof of Lemma~\ref{lem:SinTi}, Proposition~\ref{prop:Tiboundaryw} and Theorem~\ref{th:Tiboundary}} \label{SS:Tiboundaryw}
The proof of Lemma~\ref{lem:SinTi} and Proposition~\ref{prop:Tiboundaryw} turn out to be quite similar, both involving random walk path switching to avoid sites that are costly for survival. As explained in Subsection~\ref{SS:outline}, to bound $|\partial\Ti|$ and prove Proposition~\ref{prop:Tiboundaryw}, it suffices to give an upper bound on the expected total number of visits to $\bigcup_{x\in\partial\Ti}\ball{x}{(\log N)^6}$, as well as a uniform lower bound on the expected number of visits to $\ball{x}{(\log N)^6}$ over all $x\in\partial\Ti$. %The upper bound is fairly simple since the random walk is typically killed soon after visiting $v\in\partial\Ti$. The challenge is to prove the lower bound. If the expected number of visits to} $x\in \partial \Ti$ is too low, then we expect to gain a lot by forcing the random walk to avoid $x$ and closing a ``truly''-open site next to $x$. However, a ``truly''-open site is defined by the obstacle configuration in its $(\log N)^5$ neighborhood and hence we have to force the random walk to avoid that neighborhood. Therefore, we should replace the harmonic measure at $x$ by the probability of visiting $\ball{x}{(\log N)^6}$, which can in turn be bounded by the expected number of visits to $\ball{x}{(\log N)^6}$, defined as
More precisely, define 
\be
G_{{\Oi}}(u, x) := \bE_u\left[\sum_{n=0}^{\tau_{{\Oi}}} 1_{S_n\in \ball{x}{(\log N)^6}}\right], \quad u\in \Ti, x\in \partial \Ti,
\label{eq:Green}
\ee
which is the expected number of visits to $\ball{x}{(\log N)^6}$ before the walk is killed. We also introduce the set where the above expected number of visits is too small: 
\begin{equation}
 \Li(l):=\bigcup_{x\in \partial\Ti, G_{{\Oi}}(\mathcal{x}_N,x)\le \varrho_N^{1-d}\varphi(N,l)}\ball{x}{(\log N)^6},
\label{eq:defL}
\end{equation}
where
\begin{equation}
 \varphi(N,l):=
\begin{cases}
\varrho_N^{-c_4\epsilon},&\textrm{if }l=\epsilon\varrho_N,\\
(\log N)^{-c_5},&\textrm{if }l=\varrho_N/\log N
\end{cases}
\end{equation}
with $c_4\in(0,1)$ and $c_5>0$ to be chosen later.

We first claim that the expected number of visits to the neighborhood of $\partial\Ti$ is not too large.
\bl \label{lem:Gubd}
There exists $c_6>0$ such that
\be
%\lim_{N\to\infty} \mu_N\left( 
\sum_{x\in \partial \Ti}  G_{{\Oi}}(\mathcal{x}_N, x) \leq (\log N)^{c_6}.
% \right) = 1.
\label{eq:Gubd}\ee
\el

We will then show that on the event of confinement in $\ball{\mathcal{x}_N}{\varrho_N+l}$ and $\ball{\mathcal{x}_N}{\varrho_N-l/4}$ being free of obstacles, the probability for the random walk to visit $\Ti^c$ or $\Li(l)$ is asymptotically negligible. 
\bl \label{lem:Glbd}
There exists $\epsilon_0>0$ such that the following holds: let $l:=\epsilon\varrho_N$ with $\epsilon\le \epsilon_0$ or $l:=\varrho_N/\log N$, and assume that
\begin{equation}
\lim_{N\to\infty} \mu_N\left(
\tau_{\ball{\mathcal{x}_N}{\varrho_N+l}}>N,
\Oi\cap\ball{\mathcal{x}_N}{\varrho_N-l/4}=\emptyset
% \min_{x\in \partial \Ti}  G_{{\Oi}}(0, x) \geq  \varrho_N^{1-d} N^{-1/2} 
\right) = 1.
\label{eq:assum-Glbd} 
\end{equation}
Then
\begin{equation}
\lim_{N\to\infty} \mu_N\left(
\tau_{\Ti^c\cup
%\Ti_r^c\cap
\Li(l)}\le N
% \min_{x\in \partial \Ti}  G_{{\Oi}}(0, x) \geq  \varrho_N^{1-d} N^{-1/2} 
\right) = 0.
\label{eq:Glbd}
\end{equation}
\el
Let us present three consequences of these two lemmas before giving proofs.

\medskip
\noindent
\textbf{Proof of Lemma~\ref{lem:SinTi}.}
Since Theorem~\ref{th:confine} and Proposition~\ref{prop:Ticoverw} imply 
\begin{equation}
\lim_{N\to\infty} \mu_N\left(\tau_{\ball{\mathcal{x}_N}{(1+\epsilon)\varrho_N}}>N, \Oi\cap\ball{\mathcal{x}_N}{(1-\epsilon/4)\varrho_N}=\emptyset\right) = 1
\end{equation}
for any $\epsilon>0$, Lemma~\ref{lem:SinTi} immediately follows from Lemma~\ref{lem:Glbd} with $l=\epsilon\varrho_N$.\qed 

\medskip
\noindent
\textbf{Proof of Proposition~\ref{prop:Tiboundaryw}.}
%Proposition~\ref{prop:Tiboundaryw} also follows from Lemmas~\ref{lem:Gubd} and~\ref{lem:Glbd}. Indeed, 
By Lemma~\ref{lem:TinS} (cf.~\eqref{eq:Ti2S-restate}), we know that the random walk does visit $\ball{v}{(\log N)^6}$ for each $v\in\partial\Ti$, and together with Lemma~{\ref{lem:Glbd}}, this implies that we must have $\Li(\epsilon\varrho_N)=\emptyset$. This means that we have a uniform lower bound $\min_{x\in\partial\Ti}G_{{\Oi}}(\mathcal{x}_N,x)\ge \varrho_N^{1-d-c_4\epsilon}$ and hence
\begin{equation}
 \sum_{x\in \partial \Ti}  G_{{\Oi}}(\mathcal{x}_N, x) \ge |\partial\Ti|\varrho_N^{1-d-c_4\epsilon}.
\end{equation}
Combining with Lemma~\ref{lem:Gubd}, we conclude that $|\partial\Ti|\le \varrho_N^{d-1+c_4\epsilon}(\log N)^{c_6}$, and since $\epsilon>0$ can be taken arbitrarily small, Proposition~\ref{prop:Tiboundaryw} follows.\qed 

\medskip
\noindent
\textbf{Proof of Theorem~\ref{th:Tiboundary} assuming Theorem~\ref{th:Ticover}.}
Observe that once we have proved Theorem~\ref{th:Ticover}, we may take $l=\varrho_N/\log N$ in Lemma~\ref{lem:Glbd}. Then the same argument as above yields Theorem~\ref{th:Tiboundary}.\qed

\medskip

We close this subsection with the proof of Lemma~\ref{lem:Gubd} which is fairly simple. The proof of Lemma~\ref{lem:Glbd} is much more involved and will take up the next two subsections.

\medskip
\noindent\textbf{Proof of Lemma~\ref{lem:Gubd}.}
Let us define the stopping times 
\begin{equation}
\xi_1:=\inf\left\{n\ge 0\colon {\rm dist}(S_n,\partial\Ti)<(\log N)^6\right\}
\end{equation}
and for $k\ge 1$, 
\begin{equation}
 \xi_{k+1}:=\inf\left\{n\ge \xi_k+2(\log N)^{10}\colon {\rm dist}(S_n,\partial\Ti)<(\log N)^6\right\}.
\end{equation}
We can then bound the left-hand side of~\eqref{eq:Gubd} by
\begin{equation}
\sum_{x\in {\partial\Ti}}G_{{\Oi}}(\mathcal{x}_N,x)\le (\log N)^C\bE\left[\max\{k\colon \xi_k<\tau_{\Oi}\}\right]
\end{equation}
for some $C>0$. Observe that whenever the random walk visits $(\log N)^6$ neighborhood of $\partial\Ti$, there is more than $c(\log N)^{-6d}$ probability of exiting $\Ti$ within the next $(\log N)^{12}$ steps by the local central limit theorem. And once the random walk exits $\Ti$, it will hit $\Oi$ in the next $(\log N)^{12}$ steps with high probability by the definition of $\Ti$. Therefore $\max\{k\colon \xi_k<\tau_{\Oi}\}$ is stochastically dominated by a geometric random variable with parameter $c(\log N)^{-6d}$ and the desired bound follows. \qed

\subsection{Path decomposition}
\label{SS:path-decomp}
In order to prove Lemma~\ref{lem:Glbd}, what will be relevant is the behavior of the random walk near $\partial\Ti$. Since $v\in \Z^d$ is ``truly''-open if its $(\log N)^5$ neighborhood is open and we assume $\Oi\cap\ball{\mathcal{x}_N}{\varrho_N-l/4}=\emptyset$, we know that $\partial\Ti$ lies near $\partial\ball{\mathcal{x}_N}{\varrho_N}$. This motivates us to decompose the random walk paths according to the crossings of a thin annulus near the boundary of the confinement ball $\ball{\mathcal{x}_N}{\varrho_N}$. 

Similarly to~\eqref{eq:PathDecompB}--~\eqref{eq:PathDecompE}, we decompose a random walk path $(S_n)_{0\ge n\le N}$ by using successive crossings between the inner and outer shells of the annulus 
\begin{equation}
A(\mathcal{x}_N;\varrho_N-2l,\varrho_N-l)
=\ball{\mathcal{x}_N}{\varrho_N-l}\setminus \overline{\ball{\mathcal{x}_N}{\varrho_N-2l}},
\end{equation}
where we will choose $l>0$ to be either $\epsilon\varrho_N$ or $\varrho_N/\log N$. To this end, we introduce the stopping times
\begin{align}
\sigma_1 & := \min\{n\geq 0\colon S_n \in {\overline{\ball{\mathcal{x}_n}{\varrho_N-2l}}}\}{\wedge N},\label{eq:sigma_1}
\end{align}
and for $k\in\N$, 
\begin{align}
\tau_k & := \min\{n> \sigma_k \colon S_n \in B(\mathcal{x}_n,\varrho_N-l)^c\}{\wedge N}, \label{eq:tau_k}\\
\sigma_{k+1} & := \min\{n> \tau_k \colon S_n \in {\overline{\ball{\mathcal{x}_n}{\varrho_N-2l}}}\}{\wedge N}.\label{eq:sigma_{k+1}}
\end{align}
In what follows, we will decompose the random walk paths into the pieces $(S_{[\tau_k,\tau_{k+1}]})_{k\ge 1}$ and the role of $(\sigma_k)_{k\ge 1}$ is auxiliary. More precisely,  the paths that visit a costly site $v\in\Ti^c\cup\Li(l)$ (cf.~\eqref{eq:defL}) during ${[\tau_k,\tau_{k+1}]}$ are going to be switched to the paths that stay inside $\ball{\mathcal{x}_N}{\varrho_N-l/2}$ during ${[\tau_k,\tau_{k+1}]}$. 

We use bold face letters to denote sequences of numbers as in Subsection~\ref{SS:odenproof}. For a non-decreasing sequence $\bm{t}=(t_k)_{k\ge 1}$ of integers, by slightly abusing our notation in Subsection~\ref{SS:odenproof}, we write 
\begin{equation}
K({\bm{t}}):=\sup\{k\ge 1\colon {t_{k+1}-t_k}>0\}
\end{equation}
which represents the number of crossings from $\partial\ball{\mathcal{x}_N}{\varrho_N-l}$ to $\partial\ball{\mathcal{x}_N}{\varrho_N-2l}$ and back to $\partial\ball{\mathcal{x}_N}{\varrho_N-l}$ when $\bm{t}=\bm{\tau}$. We have the following control on $K(\bm{\tau})$. %and the duration of each crossing.
\begin{lemma}
\label{lem:num-length}
There exists $c_3>0$ depending only on the dimension $d$ such that if
\begin{equation}
\lim_{N\to\infty} \mu_N\left(\tau_{\ball{\mathcal{x}_N}{\varrho_N+l}^c}>N, \Oi\cap\ball{\mathcal{x}_N}{\varrho_N-l/4}=\emptyset\right)=1
\label{eq:assum-confine}
\end{equation}
for some $l\in[\varrho_N/\log N,c_3\varrho_N]$, then
\begin{equation}
\lim_{N\to\infty} \mu_N\left(K({\bm{\tau}})\le Nl^{-2}
\right)=1.
\label{eq:num-length}
\end{equation}
\end{lemma}
%\begin{remark}
%We shall use this lemma twice: first for $l={\epsilon}\varrho_N$ for some small but fixed $\epsilon\in(0,c_3)$, and a second time for $l=\varrho_N/\log N$. In the first case, we know~\eqref{eq:assum-confine} holds by Theorem~\ref{th:confine} and Proposition~\ref{prop:Ticoverw}. We will then bootstrap and use~\eqref{eq:num-length} with $l={\epsilon}\varrho_N$ to prove that~\eqref{eq:assum-confine} holds for $l=\varrho_N/\log N$. 
%\end{remark}
\noindent{\bf Proof of Lemma~\ref{lem:num-length}.}
Let us fix $l\in [\varrho_N/\log N,c_3\varrho_N]$ and suppose that $K(\bm{\tau})>Nl^{-2}$, or equivalently, $\sigma_{\lfloor Nl^{-2}\rfloor}<N$. Then we find that
\begin{equation}
\begin{split}
&\bP\left(\tau_{\ball{\mathcal{x}_N}{\varrho_N+l}^c}>N,K(\bm{\tau})> Nl^{-2}\right)\\
&\quad\le \sup_{u\in\partial\ball{\mathcal{x}_N}{\varrho_N-l}}
\bP_u\left(\tau_{\ball{\mathcal{x}_N}{\varrho_N-2l}}<\tau_{\ball{\mathcal{x}_N}{\varrho_N+l}^c}\right)^{Nl^{-2}-1}
\end{split}
\label{eq:many-crossings}
\end{equation}
by using the strong Markov property at each $\tau_k$ with $k\le Nl^{-2}-1$. Since
\begin{equation}
\sup_{u\in\partial\ball{\mathcal{x}_N}{\varrho_N-l}}
\bP_u\left(\tau_{\ball{\mathcal{x}_N}{\varrho_N-2l}}<\tau_{\ball{\mathcal{x}_N}{\varrho_N+l}^c}\right)
\end{equation}
is bounded away from one for all large $N$, by choosing $c_3$ sufficiently small and recalling~\eqref{eq:surv}, we find that the right-hand side of~\eqref{eq:many-crossings} decays faster than $\P\otimes\bP(\tau_\Oi>N)$. \qed

It is also useful to know that the random walk does not end up near the boundary of the confinement ball at time $N$. 
\begin{lemma}
\label{lem:end-inside}
\begin{equation}
 \lim_{\epsilon\to 0}\limsup_{N\to\infty}\mu_N\left(S_N \in \ball{\mathcal{x}_N}{(1-\epsilon)\varrho_N}\right)=1.
\end{equation} 
\end{lemma}
\noindent\textbf{Proof of Lemma~\ref{lem:end-inside}.} 
By Theorem~\ref{th:confine} and Proposition~\ref{prop:Ticoverw}, it suffices to show that
\begin{equation}
\begin{split}
 \P\otimes\bP\Bigl(S_N \not\in \ball{\mathcal{x}_N}{(1-\epsilon)\varrho_N}\mid\tau_{\Oi\cup\ball{\mathcal{x}_N}{\varrho_N+\varrho_N^{\epsilon_1}}^c}>N,\\\Oi\cap\ball{\mathcal{x}_N}{(1-\epsilon)\varrho_N}=\emptyset\Bigr)
\end{split}
\end{equation} 
tends to zero as $N\to \infty$ and $\epsilon\to 0$. Let us write $\Bi={\ball{\mathcal{x}_N}{\varrho_N+\varrho_N^{\epsilon_1}}\setminus\Oi}$ in this proof to ease the notation. We use the eigenfunction expansion to get
\begin{equation}
\bE\left[f(S_N)\colon\tau_{{\Bi^c}}>N\right]
 =\sum_{k\ge 1}\left(1-\lambda^{{\rm RW},(k)}_{\Bi}\right)^N\left\langle \phi^{{\rm RW},(k)}_{\Bi},f\right\rangle \phi^{{\rm RW},(k)}_{\Bi}(0)
\label{eq:EFexpansion}
\end{equation}
for any bounded function $f$, where $\lambda^{{\rm RW},(k)}_{\Bi}$ and $\phi^{{\rm RW},(k)}_{\Bi}$ are the $k$-th smallest eigenvalue and corresponding eigenfunction with $\|\phi^{{\rm RW},(k)}_{\Bi}\|_2=1$ for the generator of the random walk killed upon exiting $\Bi$. On the event $\Oi\cap\ball{\mathcal{x}_N}{(1-\epsilon)\varrho_N}=\emptyset$, each of the eigenvalues $\lambda^{{\rm RW},(k)}_{\Bi}$ and eigenfunctions $\phi^{{\rm RW},(k)}_{\Bi}$ ($k\in \N$), after proper rescaling, should be close to the eigenvalues of the Dirichlet Laplacian on the unit ball. Based on this observation, one can in fact prove (see Lemma~\ref{lem:specest} in Appendix~\ref{app:eigen}) that 
\begin{align}
%\lambda^{{\rm RW},(1)}_{\Bi}&\le C\varrho_N^{-2},
%\tag{EV1}\label{eq:ev1}\\
\lambda^{{\rm RW},(2)}_{\Bi}-\lambda^{{\rm RW},(1)}_{\Bi}
&\ge c_7\varrho_N^{-2},
\tag{EV}\label{eq:ev}\\
\left\|\phi^{{\rm RW},(1)}_{\Bi}\right\|_\infty
&\le c_8\varrho_N^{-d/2},
%\left(\lambda^{{\rm RW},(k)}_{\Bi}\right)^{d/4}
\tag{EF}\label{eq:ef}
\end{align}
which are well-known for the eigenvalues and eigenfunctions for the continuum Laplacian. It follows from~\eqref{eq:ev} that the terms with $k\ge 2$ in \eqref{eq:EFexpansion} are negligible for bounded $f$ by a standard argument, (see, for example, the proof of Lemma~\ref{lem:surv-asymp} in Appendix~\ref{app:eigen}.) By setting $f=1_{\ball{\mathcal{x}_N}{(1-\epsilon)\varrho_N}^c}$ and $f=1$ in~\eqref{eq:EFexpansion}, we find that
\begin{equation}
 \bP\left(S_N \not\in \ball{\mathcal{x}_N}{(1-\epsilon)\varrho_N}\mid\tau_{\Bi^c}>N\right)=\frac{\left\langle \phi^{{\rm RW},{(1)}}_{\Bi},1_{\ball{\mathcal{x}_N}{(1-\epsilon)\varrho_N}^c}\right\rangle}{\left\langle \phi^{{\rm RW},{(1)}}_{\Bi},1\right\rangle}{(1+o(1))}
\label{eq:1st-term}
\end{equation}
as $N\to\infty$. Since~\eqref{eq:ef} and the fact $\phi^{{\rm RW},(1)}_{\Bi}\ge 0$ imply that
\begin{align}
&\left\langle \phi^{{\rm RW},(1)}_{\Bi},1_{\ball{\mathcal{x}_N}{(1-\epsilon)\varrho_N}^c}\right\rangle
\le c\epsilon\varrho_N^{d/2},\\
&\left\langle\phi^{{\rm RW},(1)}_{\Bi},1\right\rangle
\ge c_8^{-1}\varrho_N^{d/2}\|\phi^{{\rm RW},(1)}_{\Bi}\|_2^2=c_8^{-1}\varrho_N^{d/2},
\end{align}
the right-hand side of~\eqref{eq:1st-term} vanishes as $\epsilon\to 0$.
\qed
\begin{remark}
With a little more effort, it is possible to show that the eigenfunction $\phi^{{\rm RW},(1)}_{\Bi}$ converges, after proper rescaling, to the eigenfunction of the Dirichlet Laplacian on the unit ball in $L^2$. See, for example,~\cite{Flucher95} for a further discussion on related problems. 
\end{remark}

\subsection{Proof of Lemma~\ref{lem:Glbd}}
\label{SS:Glbd-proof}
\noindent\textbf{Proof of Lemma~\ref{lem:Glbd}.}
Referring to~\eqref{eq:assum-Glbd} and Lemmas~\ref{lem:num-length} and~\ref{lem:end-inside}, let us fix $\epsilon\in(0,c_3)$, $l\in\{\varrho_N/\log N,\epsilon\varrho_N\}$ and introduce \emph{good events}
\begin{align}
 E_{\ref{eq:assum-Glbd}}&:=\left\{\Oi\cap\ball{\mathcal{x}_N}{\varrho_N-l/4}=\emptyset\right\},\\
 E_{\ref{lem:num-length}}&:=\left\{K(\bm{\sigma},\bm{\tau})\le Nl^{-2}%, \max_{k\le Nl^{-2}}(\sigma_{k+1}-\tau_k)\le c_4l^2\log N
\right\},\\
 E_{\ref{lem:end-inside}}&:=\left\{\mathcal{x}_N\in \ball{0}{(1-2\epsilon)\varrho_N}, S_N\in \ball{\mathcal{x}_N}{(1-2\epsilon)\varrho_N}\right\} 
\end{align}
and define $E_{\rm good}:=E_{\ref{eq:assum-Glbd}}\cap E_{\ref{lem:num-length}}\cap E_{\ref{lem:end-inside}}$, for which we have $\lim_{N\to\infty}\mu_N(E_{\rm good})= 1$. We are going to prove 
\begin{equation}
\P\otimes\bP\left(\tau_{\Ti^c%\cap\Ti_r^c
}\le N<\tau_\Oi, E_{\rm good}\right)
\le c(\log N)^{-1}\P\otimes\bP\left(\tau_{\Oi}>N\right)
\label{eq:better-avoid1}
\end{equation}
and for any $v\in\ball{0}{3\varrho_N}$, 
\begin{equation}
\begin{split}
&\P\otimes\bP\left(v\in\partial\Ti, G_{{\Oi}}(\mathcal{x}_N,v)\le\varrho_N^{1-d}\varphi(N,l), \tau_{\ball{v}{(\log N)^6}}\le N<\tau_\Oi, E_{\rm good}\right)\\
&\quad\le c(\log N)^{-1}
\P\otimes\bP\left(v\in\partial\Ti, \tau_{\Oi}\wedge\tau_{\ball{v}{(\log N)^6}}>N\right). 
\end{split}
\label{eq:better-avoid2}
\end{equation}
%%%%%%%%%%%%%%%%%%5
First, the bound~\eqref{eq:better-avoid1} implies that $\lim_{N\to\infty}\mu_N(\tau_{\Ti^c}\le N)=0$. Second, since $\bP(\tau_{\Oi}\wedge\tau_{\ball{v}{(\log N)^6}}>N)$ is independent of the obstacle configuration in $\ball{v}{(\log N)^6}$, in particular whether each site $w\in\ball{v}{1}$ is ``truly''-open or not, Lemma~\ref{lem:trare} implies 
\begin{equation}
\begin{split}
&\P\otimes\bP\left(v\in\partial\Ti,\tau_{\Oi}\wedge\tau_{\ball{v}{(\log N)^6}}>N\right)\\ 
&\quad \le \sum_{w\in\ball{v}{1}}
\P\otimes\bP\left(w\textrm{ is ``truly''-open, } \tau_{\Oi}\wedge\tau_{\ball{v}{(\log N)^6}}>N\right)\\
&\quad \le \exp\left\{-(\log N)^2\right\}\P\otimes\bP\left(\tau_{\Oi}>N\right).
\end{split}
\end{equation}
Therefore, by substituting this bound into~\eqref{eq:better-avoid2} and summing over $v\in\ball{0}{3\varrho_N}$, it follows that $\lim_{N\to\infty}\mu_N(\tau_{\Li(l)}\le N)=0$.

The strategy of the proofs of~\eqref{eq:better-avoid1} and~\eqref{eq:better-avoid2} is to show by a path switching argument that it is better for the random walk not to visit $\Ti^c$, or $\ball{v}{(\log N)^6}$ with $v\in\partial\Ti$ and $G_{{\Oi}}(\mathcal{x}_N,v)\le\varrho_N^{1-d}\varphi(N,l)$. Note that on the event $E_{\ref{eq:assum-Glbd}}$, we have 
\begin{equation}
\Ti^c\cup \ball{v}{(\log N)^6}\subset\ball{\mathcal{x}_N}{\varrho_N-l/2}^c
\label{eq:TLoutside}
\end{equation}
and hence it is natural to use the path decomposition in terms of the crossings from $\partial\ball{\mathcal{x}_N}{\varrho_N-l}$ to $\ball{\mathcal{x}_N}{\varrho_N-2l}$ introduced in Subsection~\ref{SS:path-decomp}. The random walk can visit $\Ti^c\cup\ball{v}{(\log N)^6}$ only on a crossing $[\tau_k,\sigma_{k+1}]$ ($k\in\N$) and if it happens, we want to switch such a crossing to one that avoids $\Ti^c\cup\ball{v}{(\log N)^6}$. However, it turns out to be easier to switch the path on the entire interval $[\tau_k,\tau_{k+1}]$. More precisely, for $u\in\partial\ball{\mathcal{x}_N}{\varrho_N-l}$ and $u'\in\partial\ball{\mathcal{x}_N}{\varrho_N-l}\cup \ball{\mathcal{x}_N}{(1-2\epsilon)\varrho_N}$, we are going to compare
\begin{equation}
\begin{split}
& p^{\rm visit}_t(u,u')\\
&\quad:=
 \begin{cases}
 \bP_u\left(\tau_{\Ti^c\cup\ball{v}{(\log N)^6}}<\tau_1=t<\tau_{\Oi}, S_t=u'\right),
 &\textrm{if }u'\in\partial\ball{\mathcal{x}_N}{\varrho_N-l},\\
 \bP_u\left(\tau_{\Ti^c\cup\ball{v}{(\log N)^6}}<t<\tau_{\Oi}\wedge\tau_1, S_t=u'\right),
 &\textrm{if }u'\in\ball{\mathcal{x}_N}{(1-2\epsilon)\varrho_N}
 \end{cases}
\end{split}
\label{eq:pvisit-def}
\end{equation}
with 
\begin{equation}
\begin{split}
 & p^{\rm avoid}_t(u,u')\\
 &:=
 \begin{cases}
 \bP_u\left(\tau_{\Ti^c\cup\ball{v}{(\log N)^6}}> \tau_1=t, S_t=u'\right),
 &\textrm{if }u'\in\partial\ball{\mathcal{x}_N}{\varrho_N-l},\\
 \bP_u\left(\tau_{\Ti^c\cup\ball{v}{(\log N)^6}}\wedge\tau_1> t, S_t=u'\right),
 &\textrm{if }u'\in\ball{\mathcal{x}_N}{(1-2\epsilon)\varrho_N}.
 \end{cases}
\end{split}
\label{eq:pavoid-def}
\end{equation}
These are the probabilities that conditionally on $S_{\tau_k}=u$, the random walk path during $[\tau_k,\tau_{k+1}]$ either visits or avoids $\Ti^c$ and $B(v; (\log N)^6)$ and ends at $u'$ at time $\tau_{k+1}=\tau_k+t$. The two cases $u'\in\partial\ball{\mathcal{x}_N}{\varrho_N-l}$ and $u'\in\ball{\mathcal{x}_N}{(1-2\epsilon)\varrho_N}$ correspond to $k<K(\bm{\tau})$ and $k=K(\bm{\tau})$ respectively, where for the latter case recall that we are working on the event $E_{\ref{lem:end-inside}}$. The key comparison estimate we will prove is the following: if $v\in\partial\Ti$, $G_{{\Oi}}(\mathcal{x}_N,v)\le \varrho_N^{1-d}\varphi(N,l)$ and $E_{\ref{eq:assum-Glbd}}$ holds, then
\begin{equation}
p^{\rm visit}_{t}(u,u') 
\le \varrho_N^{-d}(\log N)^{-3}p^{\rm avoid}_{t+2\varrho_N^2}(u,u').
\label{eq:key-compar}
\end{equation}
%which says that we can gain a lot by switching the random walk paths visiting $\Ti^c\cup\ball{v}{(\log N)^6}$ to those avoiding it with spending $2\varrho_N^2$ extra time.

Let us first see how we can deduce the desired bounds~\eqref{eq:better-avoid1} and~\eqref{eq:better-avoid2} from~\eqref{eq:key-compar}. We assume \eqref{eq:TLoutside} and $G_{{\Oi}}(\mathcal{x}_N,v)\le \varrho_N^{1-d}\varphi(N,l)$. % and write $K=K(\bm{t})$ for simplicity. 
For $j\ge 1$ and $\bm{k}=(k_i)_{i=1}^j\subset\N^j$, consider the event that the crossings with indices ${m}\in\bm{k}$ visit $\Ti^c\cup\ball{v}{(\log N)^6}$ and others do not. Its probability can be bounded as
\begin{equation}
\begin{split}
& \sum_K\sum_{t_1<t_2<\dotsc<t_{K+1}=N}\sum_{u_1,\dotsc,u_{K+1}} p^{\ball{\mathcal{x}_N}{\varrho_N-l}}_{t_1}(0,u_1)\\
&\qquad\times\prod_{{m}=1}^K \left(p^{\rm avoid}_{t_{{m}+1}-t_{{m}}}(u_{{m}},u_{{m}+1})1_{{m}\not\in\bm{k}}+p^{\rm visit}_{t_{{m}+1}-t_{{m}}}(u_{{m}},u_{{m}+1})1_{{m}\in\bm{k}}\right)\\
%&\quad \le \sum_{t_1<t_2<\dotsc<t_K}\sum_{u_1,\dotsc,u_K}p^{\ball{\mathcal{x}_N}{\varrho_N-l}}_{t_1}(0,u_1)
% \prod_{k=1}^K \left(1_{k\not\in\{k_i\}_{i=1}^j} +\varrho_N^{-d}(\log N)^{-3}1_{k\in\{k_i\}_{i=1}^j}\right) 
%p^{\rm avoid}_{t_{k+1}-t_k+2\varrho_N^2 1_{k\in\{k_i\}_{i=1}^j}}(u_k,u_{k+1})\\
&\quad \le \varrho_N^{-dj}(\log N)^{-3j}\bP\left(\tau_{\Ti^c\cup\ball{v}{(\log N)^6}}>N+2j\varrho_N^2\right), 
\end{split}
\end{equation}
where in the first line, $u_1,\dotsc,u_K\in\partial\ball{\mathcal{x}_N}{\varrho_N-l}$, $u_{K+1}\in\ball{\mathcal{x}_N}{(1-2\epsilon)\varrho_N}$ and for each ${m}\in \bm{k}$, we have used~\eqref{eq:key-compar} to get the extra multiplicative factor $\varrho_N^{-d}(\log N)^{-3}$ by lengthening the crossing duration by $2\varrho_N^2$. Recalling that we are assuming~$E_{\ref{lem:num-length}}$, we may restrict our consideration to $K\le Nl^{-2}\le \varrho_N^d(\log N)^2$. Therefore there are at most $K^j\le \varrho_N^{dj}(\log N)^{2j}$ choices of $\bm{k}=(k_i)_{i=1}^j$ and thus it follows that
\begin{equation}
\begin{split}
& \bP\left(\textrm{exactly $j$ crossings visit }\Ti^c\cup \ball{v}{(\log N)^6},E_{\ref{lem:num-length}},\tau_\Oi>N\right)\\
&\quad \le (\log N)^{-j} \bP\left(\tau_{\Ti^c\cup\ball{v}{(\log N)^6}}>N\right).
\end{split}
\end{equation}
Since $\Oi\subset\Ti^c$, summing over $j\ge 1$, we obtain~\eqref{eq:better-avoid1} and~\eqref{eq:better-avoid2}.

It remains to prove~\eqref{eq:key-compar}. Recall first that~\eqref{eq:TLoutside} holds on the event $E_{\ref{eq:assum-Glbd}}$. In particular, during $[\sigma_1,\tau_1]$, the random walk stays inside $\ball{\mathcal{x}_N}{\varrho_N-l}$ and can visit neither $\Oi$ nor $\Ti^c\cup\ball{v}{(\log N)^6}$. Based on this observation, both cases in~\eqref{eq:pvisit-def} can be described as follows: the random walk starting from $u\in\partial\ball{\mathcal{x}_N}{\varrho_N-l}$ visits $\Ti^c\cup\ball{v}{(\log N)^6}$ and $\overline{\ball{\mathcal{x}_N}{\varrho_N-2l}}$ in this order without hitting $\Oi$, and then stays inside $\ball{\mathcal{x}_N}{\varrho_N-l}$ before it ends at $u'$ at time $t$. Therefore, using the strong Markov property at $\sigma_1$, the first hitting time of $\overline{\ball{\mathcal{x}_N}{\varrho_N- 2l}}$, we obtain
\begin{equation}
\begin{split}
&p^{\rm visit}_{t}(u,u')\\ 
%&=\bP_u\left(\tau_{\Ti^c\cup\ball{v}{(\log N)^6}}<\tau_1=t<\tau_{\Oi}, S_t=u'\right)\\
&\quad=\bE_u\left[p^{\ball{\mathcal{x}_N}{\varrho_N-l}}_{t-\sigma_1}(S_{\sigma_1},u')\colon\tau_{\Ti^c\cup\ball{v}{(\log N)^6}}<\sigma_1<t\wedge\tau_{\Oi}\right].
\end{split}
\label{eq:t-s-t-visit}
\end{equation}
Similarly, by~\eqref{eq:TLoutside}, for the random walk to avoid $\Ti^c\cup\ball{v}{(\log N)^6}$, it suffices to stay inside $\ball{\mathcal{x}_N}{\varrho_N-l/2}$ and hence using the Markov property at time $\varrho_N^2$, we obtain
\begin{equation}
\begin{split}
&p^{\rm avoid}_{t+2\varrho_N^2}(u,u')\\
%&=\bP_u\left(\tau_{\Ti^c\cup\ball{v}{(\log N)^6}}\wedge \tau_\Oi>\tau_1=t+2\varrho_N^2, S_{t+\varrho_N^2}=u'\right)\\
&\quad
\ge \bE_u\left[p^{\ball{\mathcal{x}_N}{\varrho_N-l}}_{t+\varrho_N^2}(S_{\varrho_N^2},u')\colon\tau_{\ball{\mathcal{x}_N}{\varrho_N-l/2}}\wedge \tau_1>\varrho_N^2, S_{\varrho_N^2}\in\ball{\mathcal{x}_N}{\varrho_N/2}\right].
\end{split}
\label{eq:t-s-t-avoid} 
\end{equation}
When we replace $p^{\rm visit}$ by $p^{\rm avoid}$, we basically switch the path $S_{[0,\sigma_1]}$ to paths of length $\rho_N^2$ that stays inside $\ball{\mathcal{x}_N}{\varrho_N-l/2}$, does not exit $\ball{\mathcal{x}_N}{\varrho_N-l}$ after hitting $\overline{\ball{\mathcal{x}_N}{\varrho_N-2l}}$ and ends in $\ball{\mathcal{x}_N}{\varrho_N/2}$ at time $\varrho_N^2$. After time $\varrho_N^2$, we let the random walk continue to evolve until it first exits (after time $\varrho_N^2$) from $\ball{\mathcal{x}_N}{\varrho_N-l}$ at time $t+2\varrho_N^2$.

We will prove in Lemma~\ref{lem:rwest2} below the following four estimates~\eqref{eq:exit-Ti}--\eqref{eq:rwest2} in order to control the gain from this switching. The first three estimates show that we gain a lot by switching the first piece $S_{[0,\sigma_1]}$: On the event $\{\mathcal{x}_N\in\ball{0}{(1-2\epsilon)\varrho_N}\}\cap\{\Oi\cap\ball{\mathcal{x}_N}{\varrho_N-l/4}=\emptyset\}$, for any $v\in\partial\Ti$, we have
\begin{align}
\sup_{u\in\partial\ball{\mathcal{x}_N}{\varrho_N-l}}
\bP_u\left(\tau_{\Ti^c}<\sigma_1<\tau_{\Oi}\right) 
&\le \exp\left\{-(\log N)^2\right\},\tag{RW6}\label{eq:exit-Ti}\\
\sup_{u\in\partial\ball{\mathcal{x}_N}{\varrho_N-l}}
\bP_u\left(\tau_{\ball{v}{(\log N)^6}}<\sigma_1<\tau_{\Oi}\right)
&\le G_{{\Oi}}({\mathcal{x}_N},v)^2 \frac{\varrho_N^{d-1}}{l}(\log N)^{8d},
\tag{RW7}\label{eq:visit-v}
\end{align}
and there exists $c_9>0$ such that 
\begin{equation}
\begin{split}
 &\inf_{u\in\partial\ball{\mathcal{x}_N}{\varrho_N-l}}
 \bP_u\left(\tau_{\ball{\mathcal{x}_N}{\varrho_N-l/2}^c}\wedge\tau_1\ge \varrho_N^2,S_{\varrho_N^2}\in\ball{\mathcal{x}_N}{\varrho_N/2}\right)\\
 &\quad \ge c_9\frac{l}{\varrho_N}.
\end{split}
\tag{RW8}\label{eq:stay-in-B2}
\end{equation}
Substituting the last estimate~\eqref{eq:stay-in-B2} into~\eqref{eq:t-s-t-avoid}, we find that uniformly in $u\in\partial\ball{\mathcal{x}_N}{\varrho_N-l}$,
\begin{equation}
p^{\rm avoid}_{t+2\varrho_N^2}(u,u') 
\ge c_9\frac{l}{\varrho_N}\inf_{y\in\ball{\mathcal{x}_N}{\varrho_N/2}}p^{\ball{\mathcal{x}_N}{\varrho_N-l}}_{t+\varrho_N^2}(y,u'). 
\label{eq:avoid-lbd}
\end{equation}
The fourth estimate controls the possible cost caused by switching the second piece $S_{[\sigma_1,\tau_1]}$: There exists $c_{10}>0$ such that for all $t>\sigma_1$, $u'\in\partial\ball{\mathcal{x}_N}{\varrho_N-l}$ and $w\in\partial\ball{\mathcal{x}_N}{\varrho_N-2l}$, 
\begin{equation}
\begin{split}
 &p^{\ball{\mathcal{x}_N}{\varrho_N-l}}_{t-\sigma_1}(w,u')\\
 &\quad\le c_{10}\left(\frac{\varrho_N}{l}\right)^{c_{10}}\exp\left\{c_{10}\frac{\sigma_1}{\varrho_N^{2}}\right\}\inf_{y\in\ball{\mathcal{x}_N}{\varrho_N/2}}p^{\ball{\mathcal{x}_N}{\varrho_N-l}}_{t+\varrho_N^2}(y,u').
\end{split}
\tag{RW9}\label{eq:rwest2}
\end{equation}

We defer the proofs of \eqref{eq:exit-Ti}--\eqref{eq:rwest2} to the next subsection and now complete the proof of~\eqref{eq:key-compar}. Note that the cost in~\eqref{eq:rwest2} becomes large if $\sigma_1$ is large. We first exclude the case where $\sigma_1$ is atypically large by using a tail bound for $\sigma_1$. For typical values of $\sigma_1$, the cost in~\eqref{eq:rwest2} is not too large and we can use~\eqref{eq:exit-Ti}--\eqref{eq:stay-in-B2} to prove~\eqref{eq:key-compar}. Let us first consider the case $\sigma_1\ge c_{11}l^2\log N$ in~\eqref{eq:t-s-t-visit}, where $c_{11}>1$ is to be determined later. By using~\eqref{eq:rwest2}, we find that
\begin{equation}
\begin{split}
&\bE_{u}\left[p^{\ball{\mathcal{x}_N}{\varrho_N-l}}_{t-\sigma_1}(S_{\sigma_1},u')\colon c_{11}l^2\log N\le\sigma_1<  t\wedge\tau_{\Oi}\right]\\
&\quad \le c_{10}\left(\frac{\varrho_N}{l}\right)^{c_{10}}
\inf_{y\in\ball{\mathcal{x}_N}{\varrho_N/2}} p^{\ball{\mathcal{x}_N}{\varrho_N-l}}_{t+\varrho_N^2}(y,u')\\
&\qquad\times\bE_{u}\left[\exp\left\{c_{10}\frac{\sigma_1}{\varrho_N^2}\right\}\colon c_{11}l^2\log N\le\sigma_1< t\wedge\tau_{\Oi}\right].
\end{split}
\label{eq:long-crossing}
\end{equation} 
%The exponential factor in the last line can be large but instead of~\eqref{eq:exit-Ti} or~\eqref{eq:visit-v}, we note 
Observe that up to time $\sigma_1$, the walk is confined in an annulus of width $3l$ and hence $\bP_{u}(\sigma_1\ge n)\le \exp\{-cnl^{-2}\}$. This bound yields that if $l\le \epsilon\varrho_N$ for sufficiently small $\epsilon>0$, then
\begin{equation}
\begin{split}
& \bE_{u}\left[\exp\left\{c_{10}\frac{\sigma_1}{\varrho_N^2}\right\}\colon c_{11}l^2\log N\le\sigma_1< t<\tau_{\Ti^c}\right]\\
&\quad \le \sum_{n\ge c_{11}l^2\log N}\exp\left\{c_{10}\frac{n}{\varrho_N^2}-c\frac{n}{l^2}\right\}\\
&\quad\le \sum_{n\ge c_{11}l^2\log N}\exp\left\{-(c-c_{10}\epsilon^2){\frac{n}{l^2}}\right\}\\
&\quad\le N^{-cc_{11}/2}.
\end{split}
\end{equation}
We choose $c_{11}$ so large that the above right-hand side is less than $\varrho_N^{-3d}$. Then by substituting the above into~\eqref{eq:long-crossing} and comparing with~\eqref{eq:avoid-lbd}, we obtain
\begin{equation}
\begin{split}
&\bE_{u}\left[p^{\ball{\mathcal{x}_N}{\varrho_N-l}}_{t-\sigma_1}(S_{\sigma_1},u')\colon c_{11}l^2\log N\le\sigma_1<  t<\tau_{\Oi}\right]\\
&\quad \le \varrho_N^{-2d}\inf_{y\in\ball{\mathcal{x}_N}{\varrho_N/2}}
p^{\ball{\mathcal{x}_N}{\varrho_N-l}}_{t+\varrho_N^2}(y,u')\\
&\quad \le \frac12\varrho_N^{-d}(\log N)^{-3}
p^{\rm avoid}_{t+2\varrho_N^2}(u,u').
\end{split}
\label{eq:bound-long}
\end{equation}

Next, we consider the case $\sigma_1<c_{11}l^2\log N$ in~\eqref{eq:t-s-t-visit}% for a constant $c_{11}>0$ to be determined later
. In this case, we have a deterministic upper bound on the exponential factor in~\eqref{eq:rwest2} and hence
\begin{equation}
\begin{split}
&\bE_{u}\left[p^{\ball{\mathcal{x}_N}{\varrho_N-l}}_{t-\sigma_1}(S_{\sigma_1},u')\colon\tau_{\Ti^c\cup\ball{v}{(\log N)^6}}<\sigma_1< c_{11}l^2\log N\wedge t\wedge\tau_{\Oi}\right]\\
&\quad \le c_{10}\left(\frac{\varrho_N}{l}\right)^{c_{10}}
\exp\left\{c_{10}c_{11}\left(\frac{l}{\varrho_N}\right)^2\log N\right\}
\inf_{y\in\ball{\mathcal{x}_N}{\varrho_N/2}} p^{\ball{\mathcal{x}_N}{\varrho_N-l}}_{t+\varrho_N^2}(y,u')\\
&\qquad\times\bP_{u}\left(\tau_{\Ti^c\cup\ball{v}{(\log N)^6}}<\sigma_1< c_{11}l^2\log N\wedge t\wedge\tau_{\Oi}\right).
\end{split}
\label{eq:short-crossing}
\end{equation}
Now we use~\eqref{eq:exit-Ti} and~\eqref{eq:visit-v} to see that on the event $\{G_{{\Oi}}(\mathcal{x}_N,v)<\varrho_N^{1-d}\varphi(N,l)\}$, 
\begin{equation}
\begin{split}
& \bP_{u}\left(\tau_{\Ti^c\cup\ball{v}{(\log N)^6}}<\sigma_1< c_{11}l^2\log N\wedge t\wedge\tau_{\Oi}\right)\\
&\quad \le \bP_{u}\left(\tau_{\Ti^c}<\sigma_1<  t\wedge\tau_{\Oi}\right)+\bP_{u}\left(\tau_{\ball{v}{(\log N)^6}}<\sigma_1<  t\wedge\tau_{\Oi}\right)\\
&\quad \le 2\varrho_N^{2-d}\frac{\varphi(N,l)^2}{l^2}(\log N)^{8d}.
%\bP_{u}\left(\tau_{\Ti^c\cup\ball{v}{(\log N)^6}}\wedge\tau_1\ge \varrho_N^2, S_{\varrho_N^2}\in\ball{\mathcal{x}_N}{\varrho_N/2}\right).
\end{split}
\end{equation}
Substituting this into~\eqref{eq:short-crossing} and comparing with~\eqref{eq:avoid-lbd} as in the previous case, we arrive at
\begin{equation}
\begin{split}
 &\bE_{u}\left[p^{\ball{\mathcal{x}_N}{\varrho_N-l}}_{t-\sigma_1}(S_{\sigma_1},u')\colon\tau_{\Ti^c\cup\ball{v}{(\log N)^6}}<\sigma_1< c_{11}l^2\log N\wedge t\wedge\tau_{\Oi}\right]\\
&\quad \le c_{12}\frac{(\log N)^{c_{12}}\varphi(N,l)^2}{\varrho_N^{d}}\exp\left\{c_{12}\left(\frac{l}{\varrho_N}\right)^2\log N\right\}
p^{\rm avoid}_{t+2\varrho_N^2}(u,u')\\
&\quad \le \frac12\varrho_N^{-d}(\log N)^{-3}
p^{\rm avoid}_{t+2\varrho_N^2}(u,u')
\end{split}
\label{eq:bound-short}
\end{equation}
by setting% \blue{(it was said after (4.25) that the constants $c_4,c_5$ in the definition of $\varphi$ are ``to be chosen later''. I prefer to make it explicit where they are specified.)}
\begin{equation}
\varphi(N,l)=
\begin{cases}
N^{-c_{12}\epsilon}, &\textrm{if }l=\epsilon\varrho_N,\\
(\log N)^{-c_{12}-4},&\textrm{if }l=\varrho_N/\log N. 
\end{cases}
\end{equation} 
Gathering~\eqref{eq:bound-short} and~\eqref{eq:bound-long}, we get~\eqref{eq:key-compar} and we are done. \qed

\medskip
\subsection{Random walk estimates II}
In this subsection, we prove the random walk estimates~\eqref{eq:exit-Ti}--\eqref{eq:rwest2} used in Subsection~\ref{SS:Tiboundaryw}. Recall the definition of the stopping times $\sigma_k$ and $\tau_k$ ($k\in\N$) in~\eqref{eq:sigma_1}--\eqref{eq:sigma_{k+1}}.
\begin{lemma}
\label{lem:rwest2}
Suppose that $l\in[\varrho_N/\log N,c_3\varrho_N]$, $\mathcal{x}_N\in\ball{0}{(1-2\epsilon)\varrho_N}$ and $\Oi\cap\ball{\mathcal{x}_N}{\varrho_N-l/4}=\emptyset$. Then the following hold:
\begin{enumerate}
\item For $u\in\partial\ball{\mathcal{x}_N}{\varrho_N-l}$ and $v\in\partial\Ti$, 
 \begin{align}
 \bP_u\left(\tau_{\Ti^c}<\sigma_1 <\tau_{\Oi}\right) 
 &\le \exp\left\{-(\log N)^2\right\},\tag{RW6}\label{eq:exit-Ti}\\
 \bP_u\left(\tau_{\ball{v}{(\log N)^6}}<\sigma_1<\tau_{\Oi}\right) 
 &\le G_{{\Oi}}({\mathcal{x}_N},v)^2 \frac{\varrho_N^{d-1}}{l}(\log N)^{8d}.
 \tag{RW7}\label{eq:visit-v}
 \end{align}
\item  There exists $c_9>0$ such that for $u\in\partial\ball{\mathcal{x}_N}{\varrho_N-l}$, 
 \begin{equation}
 \bP_u\left(\tau_{\ball{\mathcal{x}_N}{\varrho_N-l/2}^c}\wedge\tau_1\ge \varrho_N^2,S_{\varrho_N^2}\in\ball{\mathcal{x}_N}{\varrho_N/2}\right)
 \ge c_9\frac{l}{\varrho_N}.
 \tag{RW8}\label{eq:stay-in-B2}
 \end{equation}
\item There exists $c_{10}>0$ such that uniformly in $0\le m<n$, $w\in\partial\ball{\mathcal{x}_N}{\varrho_N-2l}$  and $u'\in\partial\ball{\mathcal{x}_N}{\varrho_N-l}$, 
 \begin{equation}
\begin{split}
&  p^{\ball{\mathcal{x}_N}{\varrho_N-l}}_{n-m}(w,u')\\
&\quad \le c_{10}\left(\frac{\varrho_N}{l}\right)^{c_{10}}\exp\left\{c_{10}m\varrho_N^{-2}\right\}
 \inf_{y\in\ball{\mathcal{x}_N}{\varrho_N/2}}p^{\ball{\mathcal{x}_N}{\varrho_N-l}}_{n+\varrho_N^2}(y,u').
\end{split} 
\tag{RW9}\label{eq:rwest2}
 \end{equation}
\end{enumerate}
\end{lemma}
Let us explain the intuitions behind these bounds before delving into the proof. The first assertion~\eqref{eq:exit-Ti} follows readily from the definition of the ``truly''-open set. The second assertion~\eqref{eq:visit-v} is based on the following observation. The probability for the random walk to visit $\ball{v}{(\log N)^6}$ without hitting $\Oi$ is bounded by $G_{{\Oi}}(u,v)$. Then it has to come back to $w$ but by the time reversal, the probability is again bounded by $G_{{\Oi}}(w,v)$. Finally, the factor $l^{-1}$ comes from the fact that the random walk hits $\ball{x}{\varrho_N-2l}$ for the first time at $w$. We need the extra poly-logarithmic factor to change the starting points in $G_{{\Oi}}(u,v)$ and $G_{{\Oi}}(w,v)$ to $\mathcal{x}_N$ by using the elliptic Harnack inequality. The third assertion~\eqref{eq:stay-in-B2} is a slight modification of~\eqref{LL10}. The fourth assertion~\eqref{eq:rwest2} basically says that it is easier for the random walk to go from $y\in\ball{\mathcal{x}_N}{\varrho_N/2}$ to $u'\in\partial\ball{\mathcal{x}_N}{\varrho_N-l}$ than from $w\in\partial\ball{\mathcal{x}_N}{\varrho_N-2l}$, without exiting $\ball{\mathcal{x}_N}{\varrho_N-l}$. There are two reasons why we have a large factor on the right-hand side. First, if $w$ and $u'$ are close to each other and $n-m$ is of order $l^2$, then it is in fact better to start from $w$; second, if both $m$ and $n$ are large, then we have to include the cost for the random walk to stay in $\ball{\mathcal{x}_N}{\varrho_N-l}$ for extra time $m+\varrho_N^2$. 

\medskip
\noindent\textbf{Proof of Lemma~\ref{lem:rwest2}.} 
The left-hand side of~\eqref{eq:exit-Ti} is bounded by
\begin{equation}
\bP_u\left(\tau_{\Ti^c}<\sigma_1<\tau_\Oi\right)
\le \sup_{x\in\Ti^c}\bP_x(\sigma_1<\tau_\Oi)
\end{equation}
by the strong Markov property applied at $\tau_{\Ti^c}$.
Since we assume $\Oi\cap\ball{\mathcal{x}_N}{\varrho_N-l/4}=\emptyset$, we have
$\Ti\subset \ball{\mathcal{x}_N}{\varrho_N-l/2}$ and hence $\sigma_1>(\log N)^5$ whenever the random walk starts from $\Ti^c$. The bound~\eqref{eq:exit-Ti} follows from the definition of $\Ti$.

Next, the left-hand side of~\eqref{eq:visit-v} is bounded by
\begin{equation}
\sum_{w\in\partial\ball{\mathcal{x}_N}{\varrho_N-2l}} \sum_{y\in \ball{v}{(\log N)^6}}\bP_u\left(\tau_y<\sigma_1<\tau_{\Oi}, S_{\sigma_1}=w\right). 
\end{equation}
By reversing the time on $[\tau_y,\sigma_1]$, we have that for each $y\in \ball{v}{(\log N)^6}$, 
\begin{equation}
\begin{split}
&\bP_u\left(\tau_y< \sigma_1<\tau_{\Oi}, S_{\sigma_1}=w\right)\\
&\quad \le \bP_u\left(\tau_{{\ball{v}{(\log N)^6}}}<\sigma_1\wedge\tau_{\Oi}\right)\bP_w\left(\tau_{{\ball{v}{(\log N)^6}}}< \sigma_1\wedge\tau_{\Oi}\right).
\end{split} 
\label{eq:time-reverse}
\end{equation}
%The first factor on the right-hand side is bounded by $G_{{\Oi}}(u,v)$ uniformly in $y$. On the other hand, 
We further bound the second factor on the right-hand side by using the strong Markov property at $\tau_1$ as
\begin{equation}
\begin{split}
&\sum_{z\in\partial\ball{\mathcal{x}_N}{\varrho_N-l}}\bP_w\left(S_{\tau_1}=z, \tau_1<\sigma_1\right)
\bP_z\left(\tau_{{\ball{v}{(\log N)^6}}}< \sigma_1\wedge\tau_{\Oi}\right)\\
&\quad\le \bP_w\left(\tau_1<\sigma_1\right)
\max_{z\in \partial\ball{\mathcal{x}_N}{\varrho_N-l}}\bP_z\left(\tau_{{\ball{v}{(\log N)^6}}}< \sigma_1\wedge\tau_{\Oi}\right)\\
&\quad\le \frac{c}{l}
\max_{z\in \partial\ball{\mathcal{x}_N}{\varrho_N-l}}\bP_z\left(\tau_{{\ball{v}{(\log N)^6}}}< \sigma_1\wedge\tau_{\Oi}\right),
\end{split}
\end{equation} 
where in the last inequality we have used a gambler's ruin type estimate (see~\cite[(6.14)]{LL10} for a similar estimate). Substituting this into~\eqref{eq:time-reverse} and summing over $y\in \ball{v}{(\log N)^6}$, we find that 
\begin{equation}
\begin{split}
 &\bP_u\left(\tau_{\ball{v}{(\log N)^6}}<\sigma_1<\tau_{\Oi}, S_{\sigma_1}=w\right) \\
 &\quad \le \frac{c}{l}(\log N)^{{6d}}
 \max_{z\in \partial\ball{\mathcal{x}_N}{\varrho_N-l}}\bP_z\left(\tau_{{\ball{v}{(\log N)^6}}}< \sigma_1\wedge\tau_{\Oi}\right)^2\\ 
&\quad \le \frac{c}{l}(\log N)^{{6d}}
\max_{z\in \partial\ball{\mathcal{x}_N}{\varrho_N-l}}G_{{\Oi}}(z,v)^2.
\end{split}
\label{eq:P^2}
\end{equation}
We are going to shift the variable $z\in \partial\ball{\mathcal{x}_N}{\varrho_N-l}$ to $\mathcal{x}_N$ by applying the following elliptic Harnack inequality to the function $G_{{\Oi}}(\cdot,v)$, which is harmonic in $\ball{\mathcal{x}_N}{\varrho_N-l/2}$: There exists $c_{13}>0$ such that for any $x\in\Z^d$, $r\in\N$ sufficiently large and any non-negative harmonic function $f$ on $\ball{x}{r}$, 
\begin{equation}
 \sup_{\ball{x}{0.9r}}f(y)
\le c_{13}\inf_{\ball{x}{0.9r}}f(y).\label{eq:EHI}
\end{equation}
See~\cite[Theorem~6.3.9]{LL10}, for example. 

To compare $G_{{\Oi}}(z,v)$ with $G_{{\Oi}}(\mathcal{x}_N,v)$, we will apply~\eqref{eq:EHI} iteratively as follows. First, let $l_1=l$ and $z_1$ be the point on $\partial\ball{z}{l_j}$ closest to $\mathcal{x}_N$. Applying~\eqref{eq:EHI} to $G_{{\Oi}}(\cdot,v)$ on the ball $\ball{z_1}{l_1}$ gives $G_{{\Oi}}(z,v)\le cG_{{\Oi}}(z_1,v)$. We can now iterate this procedure. For $j\ge 2$, %we construct a sequence $(z_j)_{j=1}^J$ as follows (see Figure~\ref{fig:GrowingBallsE}): 
let $l_j:=2^{j-1}l$ and $z_{j+1}$ be the point on $\partial\ball{z_{j}}{l_j}$ closest to $\mathcal{x}_N$, and apply~\eqref{eq:EHI} to $G_{{\Oi}}(\cdot,v)$ on the ball $\ball{z_j}{l_j}$. The iteration is stopped at the first $J$ such that $z_J\in\ball{\mathcal{x}_N}{2\varrho_N/3}\}$. See Figure~\ref{fig:GrowingBallsE}.
\begin{figure}[!b]
\centering
\includegraphics{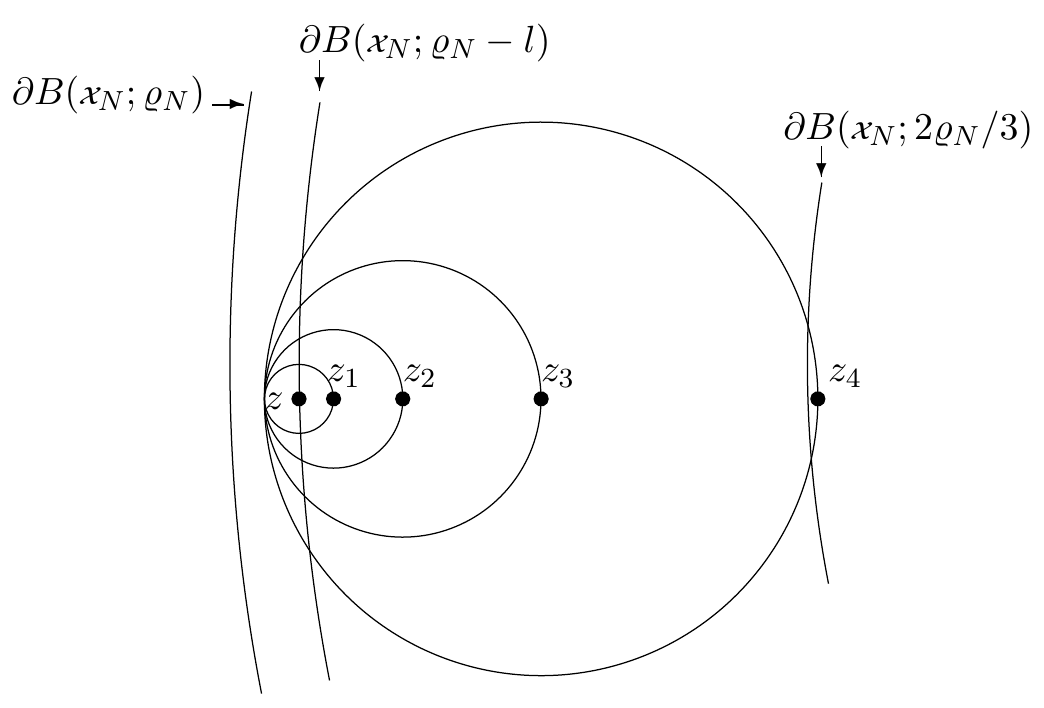}
\caption{The sequence $(z_j)_{j=1}^J$ constructed in the proof of~\eqref{eq:visit-v}. The balls around $z_j$ have geometrically growing radii and the construction terminates at $J=4$ in this picture.}
\label{fig:GrowingBallsE}
\end{figure}
Noting that $J\le c\log (\rho_N/l)$, we have
\begin{equation}
\begin{split}
%\max_{z\in \partial\ball{\mathcal{x}_N}{\varrho_N-l}}
G_{{\Oi}}(z,v)
& \le c_{13} G_{{\Oi}}(z_1,v)\\
& \le \dotsb\\
& \le c_{13}^J G_{{\Oi}}(z_J,v)\\
& \le \left(\frac{\varrho_N}{l}\right)^c G_{{\Oi}}(\mathcal{x}_N,v), 
\end{split}
\label{eq:P<G}
\end{equation}
where in the last inequality, we have applied~\eqref{eq:EHI} in $\ball{\mathcal{x}_N}{\varrho_N-l/2}$ to bound $G_{{\Oi}}(z_J,v)$ by $c_{13}G_{{\Oi}}(\mathcal{x}_N,v)$. Since $l\ge \varrho_N/\log N$, by substituting~\eqref{eq:P<G} into~\eqref{eq:P^2} and summing over $w\in\partial\ball{\mathcal{x}_N}{\varrho_N-2l}$, we get the desired bound~\eqref{eq:visit-v}. 

In order to prove the lower bound~\eqref{eq:stay-in-B2}, we let the random walk obey the following strategy: pick $u'\in\partial\ball{\mathcal{x}_N}{\varrho_N-3l/2}\cap\ball{u}{l}$ and
\begin{enumerate}
 \item $S_{l^2}\in\ball{u'}{l/4}$ without exiting $A(\mathcal{x}_N;\varrho_N-2l,\varrho_N-l/2)$;
 \item $S_{\varrho_N^2}\in\ball{\mathcal{x}_N}{\varrho_N/2}$ without exiting $\ball{\mathcal{x}_N}{\varrho_N-l}$.
\end{enumerate}
In this way, the condition $\tau_{\ball{\mathcal{x}_N}{\varrho_N-l/2}^c}\wedge\tau_1\ge \varrho_N^2$ holds and hence the left-hand side of~\eqref{eq:stay-in-B2} is bounded from below by the probability of the above strategy. With the help of~\eqref{LL10}, one can find $c>0$ such that
\begin{align}
\bP_u\left(S_{l^2}\in\ball{u'}{l/4}, S_{[0,l^2]}\subset A(\mathcal{x}_N;\varrho_N-2l,\varrho_N-l/2)\right)
&\ge c,\\
\inf_{y\in\ball{u'}{l/4}}\bP_y\left(S_{\varrho_N^2-l^2}\in\ball{\mathcal{x}_N}{\varrho_N/2}, S_{[0,\varrho_N^2-l^2]}\subset\ball{\mathcal{x}_N}{\varrho_N-l}\right)
&\ge c\frac{l}{\varrho_N}.
\end{align}
Collecting these bounds, we get~\eqref{eq:stay-in-B2}.

Finally we prove~\eqref{eq:rwest2}. In the case $n+ \varrho_N^2\le 2\varrho_N^2$, we have
\begin{equation}
\begin{split}
&p^{\ball{\mathcal{x}_N}{\varrho_N-l}}_{n+\varrho_N^2}(y,u')\\
& \quad\ge\bE_u\Bigl[p^{\ball{\mathcal{x}_N}{\varrho_N-l}}_{n+\varrho_N^2-\tau_{{\partial\ball{\mathcal{x}_N}{\varrho_N-3l/2}}}}
\left(S_{\tau_{
{\partial\ball{\mathcal{x}_N}{\varrho_N-3l/2}}}}
, y\right)\colon\\
&\qquad\qquad\tau_{
{\partial\ball{\mathcal{x}_N}{\varrho_N-3l/2}}
}<(n-m)\wedge\tau_{\ball{\mathcal{x}_N}{\varrho_N-l}^c}\Bigr]\\
&\quad\ge cl\varrho_N^{-d-1}\bP_u\left(\tau_{
{\partial\ball{\mathcal{x}_N}{\varrho_N-3l/2}}}
<(n-m)\wedge\tau_{\ball{\mathcal{x}_N}{\varrho_N-l}^c}\right),
\end{split} 
\end{equation}
where we have used the strong Markov property at $\tau_{{\partial\ball{\mathcal{x}_N}{\varrho_N-3l/2}}}$ and applied~\eqref{LL10} to the transition probability appearing inside the expectation, noting that $n+\varrho_N^2-\tau_{\partial\ball{\mathcal{x}_N}{\varrho_N-3l/2}}\in[\varrho_N^2,2\varrho_N^2]$ under the conditions $n+\varrho_N^2<2\varrho_N^2$ and $\tau_{\partial\ball{\mathcal{x}_N}{\varrho_N-3l/2}}<(n-m)\wedge\tau_{\ball{\mathcal{x}_N}{\varrho_N-l}^c}$. Similarly, we have
\begin{equation}
\begin{split}
&p^{\ball{\mathcal{x}_N}{\varrho_N-l}}_{n-m}(w,u')\\
&\quad =\bE_u\Bigl[p^{\ball{\mathcal{x}_N}{\varrho_N-l}}_{n-m-\tau_{
{\partial\ball{\mathcal{x}_N}{\varrho_N-3l/2}}
}}
\left(S_{\tau_{
{\partial\ball{\mathcal{x}_N}{\varrho_N-3l/2}}
}}
,w\right)\colon\\
&\qquad\qquad\tau_{
{\partial\ball{\mathcal{x}_N}{\varrho_N-3l/2}}
}<(n-m)\wedge\tau_{\ball{\mathcal{x}_N}{\varrho_N-l}^c}\Bigr]\\
&\quad\le cl^{-d}\bP_u\left(\tau_{
{\partial\ball{\mathcal{x}_N}{\varrho_N-3l/2}}}
<(n-m)\wedge\tau_{\ball{\mathcal{x}_N}{\varrho_N-l}^c}\right),
\end{split} 
\end{equation}
where we have used the strong Markov property at $\tau_{\ball{\mathcal{x}_N}{\varrho_N-3l/2}}$ and the estimate
\begin{equation}
\begin{split}
 &p^{\ball{\mathcal{x}_N}{\varrho_N-l}}_{n-m-\tau_{\partial\ball{\mathcal{x}_N}{\varrho_N-3l/2}}}\left(S_{\tau_{\partial\ball{\mathcal{x}_N}{\varrho_N-3l/2}}},w\right)\\
 &\quad\le  \sup_{k\in\N,|x-y|\ge l/2}p^{{\ball{\mathcal{x}_N}{\varrho_N-l}}}_k(x,y)\\
 &\quad\le Cl^{-d},
\end{split}
\label{Gasussian-upper2}
\end{equation} 
which follows in the same way as in~\eqref{Gaussian-upper}. Combining the above two bounds, we get~\eqref{eq:rwest2} in this case. 

In the other case $n+\varrho_N> 2\varrho_N^2$, we use the following parabolic Harnack inequality from~\cite[$H(C_H)$ in Theorem~1.7]{D99}: For all $x_0\in\Z^d$, $s\in\R$, $r>200$ and every non-negative $u(t,x)$ that satisfies the discrete heat equation on $\Z\cap[s,s+100r^2]\times \ball{x_0}{r}$, 
\begin{equation}
\begin{split}
&\sup_{(t_1,x_1)\in \Z\cap[s+0.01r^2,s+0.1r^2]\times \ball{x_0}{0.99r}}u(t_1,x_1)\\
&\quad \le 100
 \inf_{(t_2,x_2)\in \Z\cap[s+0.11r^2,s+100r^2]\times \ball{x_0}{0.99r}}u(t_2,x_2).
\end{split}
\label{eq:PHI}
\end{equation}
We use this to first shift the spatial variable $w$ to $y\in\ball{\mathcal{x}_N}{\varrho_N/2}$ and then the time variable in the transition probability kernel in~\eqref{eq:rwest2}. To this end, we construct a sequence $(w_j)_{j=1}^J$ in the same way as in the proof of~\eqref{eq:visit-v} (see Figure~\ref{fig:GrowingBalls}): let $w_0:=w$, and for $j\ge 0$ let $l_j:=2^{j-1}l$ and $w_{j+1}$ be the point on $\partial\ball{w_{j}}{l_j}$ closest to $\mathcal{x}_N$ and
\begin{equation}
 J:=\min\{j\ge 0\colon w_j\in\ball{\mathcal{x}_N}{2\varrho_N/3}\}.
\end{equation}
\begin{figure}[!b]
\centering
\includegraphics{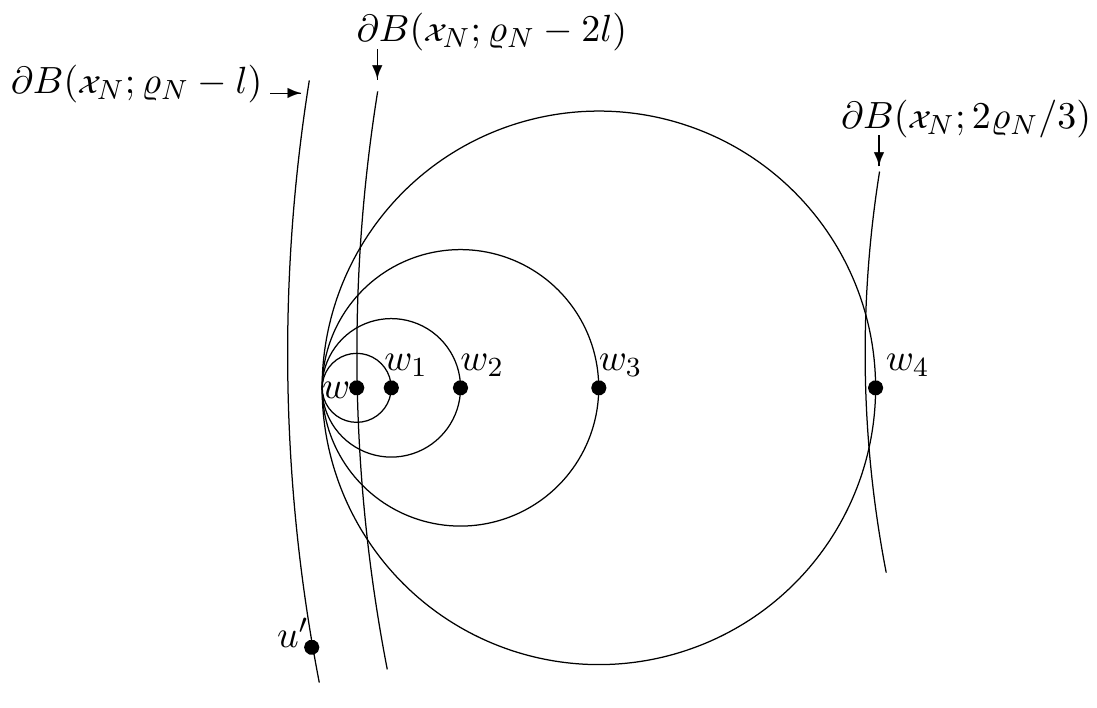}
\caption{The sequence $(w_j)_{j=1}^J$ constructed in the proof of~\eqref{eq:rwest2}. The balls around $w_j$ have geometrically growing radii and the construction terminates at $J=4$ in this picture.}
\label{fig:GrowingBalls}
\end{figure}
Note that $\varrho_N/3\le l_J\le 2\varrho_N/3$ and therefore $J\le c\log (\rho_N/l)$. As a first step, we switch from $w=w_0$ to $w_1$ and use the bound
\begin{equation}
p^{\ball{\mathcal{x}_N}{\varrho_N-l}}_{n-m}(w,u')
\le c_{14}\frac{\varrho_N}{l} p^{\ball{\mathcal{x}_N}{\varrho_N-l}}_{n-m+l_1^2}(w_1,u')\label{eq:w_1}
\end{equation}
with some $c_{14}>0$, which can be verified by applying either~\eqref{Gasussian-upper2} to the left-hand side and~\eqref{LL10} to the right-hand side when $n-m<\varrho_N^2$, or~\eqref{eq:PHI} with $x_0=w_1$, $s=n-m-0.05l_1^2$ and $r=l_1$ to the function $u(t,x):=p^{\ball{\mathcal{x}_N}{\varrho_N-l}}_t(x,u')$ restricted to $\ball{w_1}{l_1}$, when $n-m>\varrho_N^2$. Next, noting that $w_j$ keeps distance at least $l_{j+1}/2$ from $\partial\ball{w_{j+1}}{l_{j+1}}$ for any $j< J$, we can apply~\eqref{eq:PHI} with $x_0=w_j$, $s=n-m+0.95l_j^2$ and $r=l_j$ to the function $u(t,x):=p^{\ball{\mathcal{x}_N}{\varrho_N-l}}_t(x,u')$ restricted to $\ball{w_j}{l_j}$ to obtain
\begin{align}
% p^{\ball{\mathcal{x}_N}{\varrho_N-l}}_{n-m}(w,u')&\le c\frac{\varrho_N}{l} p^{\ball{\mathcal{x}_N}{\varrho_N-l}}_{n-m+l_1^2}(w_1,u'),\\
 p^{\ball{\mathcal{x}_N}{\varrho_N-l}}_{n-m+l_j^2}(w_j,u')
&\le 100 p^{\ball{\mathcal{x}_N}{\varrho_N-l}}_{n-m+l_{j+1}^2}(w_{j+1},u')\quad\textrm{for }j\ge 1, \label{eq:w_j}\\
 p^{\ball{\mathcal{x}_N}{\varrho_N-l}}_{n-m+l_J^2}(w_J,u')
&\le 100 p^{\ball{\mathcal{x}_N}{\varrho_N-l}}_{n-m+\varrho_N^2}(y,u'),
\label{eq:w_J}
\end{align}
where in the second bound, we have used $l_J^2\le 4\varrho_N^2/9$. By using~\eqref{eq:w_1}--\eqref{eq:w_J} and recalling the upper bound on $J$, we get
\begin{equation}
\begin{split}
 p^{\ball{\mathcal{x}_N}{\varrho_N-l}}_{n-m}(w,u')
 &\le c^J \frac{\varrho_N}{l} p^{\ball{\mathcal{x}_N}{\varrho_N-l}}_{n-m+l_J^2}(w_J,u')\\
 &\le \left(\frac{\varrho_N}{l}\right)^c p^{\ball{\mathcal{x}_N}{\varrho_N-l}}_{n-m+\varrho_N^2}(y,u').
\end{split}
\label{eq:adjust-point}
\end{equation}
Now another iteration of~\eqref{eq:PHI} leads us to
\begin{equation}
 \begin{split}
p^{\ball{\mathcal{x}_N}{\varrho_N-l}}_{n-m+\varrho_N^2}(y,u')
&\le 100 p^{\ball{\mathcal{x}_N}{\varrho_N-l}}_{n-m+2\varrho_N^2}(y,u')\\
&\le \dotsb\\
&\le 100^{{\lfloor m\varrho_N^{-2}\rfloor-1}} p^{\ball{\mathcal{x}_N}{\varrho_N-l}}_{n-m+[m\varrho_N^{-2}]\varrho_N^2}(y,u')\\
&\le 100^{\lfloor m\varrho_N^{-2}\rfloor} p^{\ball{\mathcal{x}_N}{\varrho_N-l}}_{n+\varrho_N^2}(y,u'),
 \end{split}
\label{eq:adjust-time}
\end{equation}
where in the $k$-th inequality, we choose $x_0=\mathcal{x}_N$, $s=n-m+(k-0.05)\varrho_N^2$, $r=\varrho_N-l$ and $u(t,x):=p^{\ball{\mathcal{x}_N}{\varrho_N-l}}_t(x,u')$ in~\eqref{eq:PHI}. The desired estimate follows from~\eqref{eq:adjust-point} and~\eqref{eq:adjust-time}. 
\qed

\section{Proof of the Refined Version of Ball Covering Theorem}% and \ref{th:Tiboundary}
\label{S:Thms}
In this section, we prove Theorems~\ref{th:Ticover}. % and~\ref{th:Tiboundary}. 
As we discussed in Subsection~\ref{SS:Tiboundaryw}, below Lemma~\ref{lem:Glbd}, Theorem~\ref{th:Tiboundary} follows from Theorem~\ref{th:Ticover}. To strengthen Proposition~\ref{prop:Ticoverw} to Theorem~\ref{th:Ticover}, we first show that the volume of $\Ti$ is very close to $|\ball{0}{\varrho_N}|$ with the help of the Faber--Krahn inequality and a control on the number of the possible shapes of $\Ti$ provided by~Proposition~\ref{prop:Tiboundaryw}. Now if there is a non-``truly''-open site $x\in\ball{\mathcal{x}_N}{\varrho_N-\varrho_N^{\epsilon_2}}$, then there exists a closed site near $x$. Lemma~\ref{lem:odensity} then implies that there are in fact many closed sites around $x$, which leads us to a contradiction. 

%
%\bp[Quantitative Faber-Krahn inequality]\label{prop:FK}
%Let $\Omega \subset \R^d$ be an open set with finite measure, and let $\lambda_\Omega$ be the principal eigenvalue of $-\Delta$ on $\Omega$ with Dirichlet boundary condition. Then there exists $c>0$ such that
%\be\label{FK}
%\lambda_\Omega - \lambda_{B_\Omega} \geq c|\Omega|^{-\frac{2}{d}} \mathcal{A}(\Omega)^2,
%\ee
%where $B_\Omega$ is the ball with $|B_\Omega|=|\Omega|$, and $\mathcal{A}(\Omega)$ is the Fraenkel asymmetry of $\Omega$ defined by
%\be\label{Fraenkel}
%\mathcal{A}(\Omega) := \inf_{x\in\R^d} \frac{|\Omega \triangle (x+B_\Omega)|}{|\Omega|},
%\ee
%where $\triangle$ denotes the symmetric difference.
%\ep
In order to carry out the proof, we will compare $\lambda^{{\rm RW},(k)}_\Ti$, the $k$-th smallest eigenvalue for the generator of the random walk killed upon exiting $\Ti$, with its counterpart for the continuum Laplacian, and then apply the Faber--Krahn inequality. We define the \emph{continuous hull} of $T\subset\Z^d$ by
\begin{equation}
 \tilde{T}:=\{x\in\R^d\colon {\rm dist}_\infty(x,T)<2\}
\label{eq:cont-hull}
\end{equation}
and denote the $k$-th smallest Dirichlet eigenvalue of $-\frac{1}{2d}\Delta$ in $\tilde{T}$ and the corresponding eigenfunction by $\lambda_{\tilde{T}}^{(k)}$ and $\phi_{\tilde{T}}^{(k)}$, respectively. We will prove the following comparison in Appendix~\ref{app:eigen}. 
\begin{lemma}
\label{lem:disc-cont}
There exists $c>0$ such that for any $x\in\Z^d$ and $T\supset\ball{x}{\varrho_N/2}$, 
\begin{equation}
 \lambda^{{\rm RW},(1)}_T\ge\lambda_{\tilde{T}}^{{(1)}}-c\varrho_N^{-4}.
\label{eq:disc>cont}
\end{equation}
%on the event $\{\ball{\mathcal{x}_N}{(1-\eps)\varrho_N} \cap \Oi=\emptyset\}$\footnote{The formulation of Proposition~\ref{prop:Ticoverw} is modified.} with $\epsilon<1/2$, \blue{the following bound holds:}\footnote{I don't want to use ``we'' in statements of lemma. By the way, I don't really understand why I need to add something like ``we have'' here. Does this correction applies to other statements?}
%\begin{equation}
% \lambda^{{\rm RW},(1)}_\Ti\ge \lambda_{\tilde{\Ti}}^{(1)}-c\varrho_N^{-4}.
%\label{eq:disc>cont}
%\end{equation}
Moreover, for each $k\in\N$, there exists $\gamma_k>0$ such that 
\begin{equation}
\left| \lambda^{{\rm RW},(k)}_{\ball{0}{\varrho_N}}-\lambda^{(k)}_{\ball{0}{\varrho_N}}\right|\le \gamma_k\varrho_N^{-3}.
\label{eq:disc<cont}
\end{equation}
\end{lemma}
Let us recall a lower bound on the survival probability which is well-known in the continuum setting~\cite[(33)]{P99} and for the two dimensional continuous time random walk~\cite[Proposition~2.1]{B94}. Lemma~\ref{lem:surv-asymp} below is the analogue in our discrete time setting, which we prove in Appendix~\ref{app:eigen} for completeness. Recall that we denote the Euclidean volume of $G\subset\R^d$ by $\vol{G}$. 
\begin{lemma}
\label{lem:surv-asymp}
There exists $c>0$ such that for all sufficiently large $N>0$, 
\begin{equation}
\P\otimes\bP\left(\tau_\Oi>N\right)
\ge \exp\left\{\vol{\ball{0}{\varrho_N}}\log p-N\lambda_{\ball{0}{\varrho_N}}^{(1)}-c\varrho_N^{d-1}\right\}.
\end{equation}
\end{lemma}
Using this lemma, we can show that the volume of $\tilde{\Ti}$ is very close to $\vol{\ball{0}{\varrho_N}}$.
\begin{lemma}
\label{lem:vol-control}
For any $b\in(0,1)$, 
\begin{equation}
\lim_{N\to\infty} \mu_N\left(\left|\vol{\tilde{\Ti}}-\vol{\ball{0}{\varrho_N}}\right|< \varrho_N^{(d-1)/2+2b}%\textrm{ and }\mathcal{A}(\tilde{\Ti})< \varrho_N^{-1/2+2b}
\right)=1.
\end{equation}
\end{lemma}
\noindent\textbf{Proof of Lemma~\ref{lem:vol-control}.}
Referring to Theorem~\ref{th:confine} and Propositions~\ref{prop:Ticoverw} and~\ref{prop:Tiboundaryw}, we introduce the set of possible shapes of $\Ti$ for $b\in(0,1)$: 
\begin{equation}
\begin{split}
 \mathbb{T}_b:=
 \bigl\{T\subset\Z^d\colon  \ball{x}{0.9\varrho_N}\subset T\subset \ball{x}{\varrho_N+\varrho_N^{\epsilon_1}}\textrm{ for}\\ 
\textrm{some }
x\in\ball{0}{\varrho_N},\textrm{ and } |\partial T|\le \varrho_N^{d-1+b}\bigr\}
\end{split}
\end{equation}
so that $\lim_{N\to\infty}\mu_N(\Ti\in\mathbb{T}_b)=1$. The cardinality of this set is bounded by
\begin{equation}
|\mathbb{T}_b|\le \exp\{c\varrho_N^{d-1+2b}\}
\label{eq:Tentropy}
\end{equation}
simply by considering the choice of $\varrho_N^{d-1+b}$ points from $\ball{0}{3\varrho_N}$. Having controlled the entropy, we estimate next the probability $\mu_N(\Ti=T)$ for each $T\in\mathbb{T}_b$. With the help of the eigenfunction expansion, one finds the upper bound 
\begin{equation}
\begin{split}
\P\otimes\bP\left(\tau_{T^c}>N, \Ti=T\right)
 &\le \P(\Oi\cap T=\emptyset)\bP(\tau_{T^c}>N)\\
 &\le |T|^{1/2}\exp\left\{|T|\log p-N\lambda^{{\rm RW},(1)}_T\right\}
\end{split}
\label{eq:T=Tubd}
\end{equation}
for general $T\subset\Z^d$. See, for example,~\cite[(2.21)]{K16}. Observe that since $\tilde{T}\subset \bigcup_{x\in T}(x+[-2,2]^d)$, we have $|T|\ge \vol{\tilde{T}}-c|\partial T|$ and hence for $T\in\mathbb{T}_b$, 
\begin{equation}
 |T|\ge \vol{\tilde{T}}-c\varrho_N^{d-1+b}.
\label{eq:vol-approx}
\end{equation}
On the other hand, by~\eqref{eq:disc>cont} and the Faber--Krahn inequality (see, for example,~\cite[pp.87--92]{Chavel84}), for any $T\in\mathbb{T}_b$, we have
\begin{equation}
\begin{split}
\lambda^{{\rm RW},(1)}_T
&\ge \lambda_{\tilde{T}}^{(1)}-c\varrho_N^{-4}\\
&\ge \lambda_{|\tilde{T}|}%+ c\varrho_N^{-2} \mathcal{A}(\tilde{T})^2
-c\varrho_N^{-4},
\end{split}
\label{eq:ev-approx}
\end{equation}
where for $r>0$, we denote by $\lambda_r$ the principal Dirichlet eigenvalue of $-\frac{1}{2d}\Delta$ in a ball with volume $r$. Substituting~\eqref{eq:vol-approx} and~\eqref{eq:ev-approx} into~\eqref{eq:T=Tubd}, we obtain 
\begin{equation}
\begin{split}
& \P\otimes\bP\left(\tau_{T^c}>N, \Ti=T\right)\\
&\quad \le |T|^{1/2}\exp\left\{\vol{\tilde{T}}\log p-N{\lambda_{|\tilde{T}|}}
 %- c\varrho_N^{d} \mathcal{A}(\tilde{T})^2
 +c\varrho_N^{d-1+b}\right\}. 
\end{split}
\label{eq:T=Tubd2}
\end{equation}
%We use this and~\eqref{eq:ball-strategy} to exclude the following situation:
%\begin{align}
% \left||\tilde{T}|-|\ball{0}{\varrho_N}|\right|&\ge \varrho_N^{(d-1)/2+2b},\label{eq:vol-dev}\\ 
% \mathcal{A}(\tilde{T})&\ge \varrho_N^{-1/2+2b}.\label{eq:A-large}
%\end{align}
Suppose that $T\in\mathbb{T}_b$ satisfies $|\vol{\tilde{T}}-\vol{\ball{0}{\varrho_N}}|\ge\varrho_N^{(d-1)/2+2b}$. Then, since %${\lambda_{|\tilde{T}|}}=\vol{\tilde{T}}^{-2/d}{\lambda_1}$ and 
the function
\begin{equation}
 r\mapsto r\log p-N\lambda_r%{\ball{0}{r}}^{(1)}
=r\log p-\frac{N\lambda_1}{r^{2/d}}
\end{equation}
is twice-differentiable and maximized at $\vol{\ball{0}{\varrho_N}}$ (cf.~\eqref{eq:Rn}), one finds by the Taylor expansion that
\begin{equation}
\begin{split}
&\vol{\tilde{T}}\log p-N{\lambda_{|\tilde{T}|}}\\
&\quad\le \vol{\ball{0}{\varrho_N}}\log p-N\lambda_{\ball{0}{\varrho_N}}^{(1)}-c\left|\vol{\tilde{T}}|-\vol{\ball{0}{\varrho_N}}\right|^2\\
&\quad\le \vol{\ball{0}{\varrho_N}}\log p-N\lambda_{\ball{0}{\varrho_N}}^{(1)}-c\varrho_N^{d-1+4b}.
\end{split}
\end{equation}
Substituting this into~\eqref{eq:T=Tubd2} and comparing with Lemma~\ref{lem:surv-asymp}, we obtain $\mu_N(\Ti=T)\le \exp\{-c\varrho_N^{d-1+4b}\}$. 
%Similarly, if we suppose that $T\in\mathbb{T}_b$ satisfies~\eqref{eq:A-large}, then recalling that $\ball{0}{\varrho_N}$ is a maximizer of the set function
%\begin{equation}
% \Omega \mapsto |\Omega|\log p-N\lambda_\Omega,
%\end{equation}
%we obtain $\mu_N(\Ti=T)\le \exp\{-c\varrho_N^{d-1+3b}\}$. 
Thanks to~\eqref{eq:Tentropy}, we can use the union bound to conclude the proof of Lemma~\ref{lem:vol-control}.\qed
%that 
%\begin{equation}
%\lim_{N\to\infty} \mu_N\left( \left||\tilde{\Ti}|-|\ball{0}{\varrho_N}|\right|< \varrho_N^{(d-1)/2+2b}%\textrm{ and }\mathcal{A}(\tilde{\Ti})< \varrho_N^{-1/2+2b}
%\right)=1.
%\end{equation}

\medskip
\noindent{\textbf{Proof of Theorems \ref{th:Ticover}}.} 
Thanks to Lemma~\ref{lem:odensity}, we can restrict our consideration to the event
\begin{equation}
 \bigcap_{x\in \ball{0}{2\varrho_N}} \, \bigcap_{{(\log N)^3}\leq l\leq {\varrho_N}} \left\{x\in \Oi \textrm{ and }{\frac{|\Oi \cap \ball{x}{l}|}{|\ball{x}{l}|}} < \delta \right\}, 
\label{eq:lem3.1}
\end{equation}
that is, any $x\in\ball{0}{2\varrho_N}$ is either open or has $\delta$-fraction of closed sites in its $l$-neighborhood for all $l\in[(\log N)^3,\varrho_N]$. In addition, by Lemma~\ref{lem:vol-control} with $b=\epsilon_1/2$ where $\epsilon_1 $ is as in Theorem~\ref{th:confine}, we can further assume that
\begin{equation}
 \left|\vol{\tilde{\Ti}}-\vol{\ball{0}{\varrho_N}}\right|< \varrho_N^{(d-1)/2+\epsilon_1}. 
\label{eq:lem5.3}
\end{equation}

Let $\epsilon_2>(d-1+\epsilon_1)/d$ and suppose that there exists $x\in\ball{\mathcal{x}_N}{\varrho_N-\varrho_N^{\epsilon_2}}\setminus\Ti$. Then there exists at least one closed site $y\in\ball{x}{(\log N)^5}$, and by~\eqref{eq:lem3.1}, more than $\delta$ fraction of the points in the ball $\ball{y}{\varrho_N^{\epsilon_2}/2}\subset \ball{\mathcal{x}_N}{\varrho_N}$ must be closed. Recalling~\eqref{eq:vol-approx} and that $\tilde{\Ti}\subset \ball{\mathcal{x}_N}{\varrho_N+2\varrho_N^{\epsilon_1}}$ by the definition of $\Ti$, we find that 
\begin{equation}
\begin{split}
\vol{\tilde{\Ti}}
&\le |\Ti|+\varrho_N^{d-1+b}\\
&\le |\ball{\mathcal{x}_N}{\varrho_N+2\varrho_N^{\epsilon_1}}|-\delta|\ball{y}{\varrho_N^{\epsilon_2}/2}|+\varrho_N^{d-1+\epsilon_1}\\
&\le \vol{\ball{0}{\varrho_N}}+c\varrho_N^{d-1+\epsilon_1}-c'\varrho_N^{d\epsilon_2}\\
&\le \vol{\ball{0}{\varrho_N}}-c\varrho_N^{d-1+\epsilon_1},
\end{split}
\end{equation}
which contradicts~\eqref{eq:lem5.3}. \qed

\medskip  

\begin{remark}[Finer asymptotics of survival probability]
There is a conjecture on the precise second order asymptotics of the survival probability in the literature: there exists $a_1>0$ such that
\begin{equation}
\P\otimes\bP(\tau_\Oi>N) =\exp \left\{-c(d,p)N^{\frac{d}{d+2}}-a_1N^{\frac{d-1}{d+2}}+o(N^{\frac{d-1}{d+2}})\right\}. 
\end{equation}
See~\cite{L84} and the bottom of~p.~76 in~\cite{B02} for more information. Lemma~\ref{lem:surv-asymp} gives a lower bound of this form while the currently best known upper bound is
\begin{equation}
\P\otimes\bP(\tau_\Oi>N)
\le\exp\left\{-c(d,p)N^{\frac{d}{d+2}}+N^{\frac{d-\epsilon}{d+2}}\right\}
\end{equation} 
for some $\epsilon\in(0,1)$. See~\cite[(2.40)]{B02} or~\cite[Theorem~5.6 on page 208]{S98}. 

Based on what we have proved, we can get a refined upper bound on the survival probability. Theorem~\ref{th:Tiboundary} implies that $\lim_{N\to\infty}\mu_N(\Ti\in\mathbb{T}_{0+})=1$ with
\begin{equation}
\begin{split}
 \mathbb{T}_{0+}:=\Bigl\{T\subset\Z^d\colon  \ball{x}{0.9\varrho_N}\subset T\subset \ball{x}{\varrho_N+\varrho_N^{\epsilon_1}}\textrm{ for}\\ 
\textrm{some }x\in\ball{0}{\varrho_N}\textrm{ and }
|\partial T|\le \varrho_N^{d-1}(\log N)^a\Bigr\}.
\end{split}
\end{equation}
Just as in~\eqref{eq:Tentropy}, we have 
\begin{equation}
|\mathbb{T}_{0+}|\le \exp\{c\varrho_N^{d-1}(\log N)^{a+1}\}
%\label{eq:Tentropy2}
\end{equation}
and then by Lemma~\ref{lem:SinTi} and a variant of~\eqref{eq:T=Tubd2}, we obtain
\begin{equation}
\begin{split}
&\P\otimes\bP(\tau_\Oi>N)\\
&\quad\sim \P\otimes\bP(\tau_{\Ti^c}>N, \Ti\in\mathbb{T}_{0+})\\
&\quad\le |\mathbb{T}_{0+}|\sup_{T\in\mathbb{T}_{0+}}|T|^{1/2}\exp\left\{\vol{\tilde{T}}\log p-N\lambda_{B_{\tilde{T}}}+c\varrho_N^{d-1}(\log N)^{a+1}\right\}\\
&\quad\le \exp\left\{\vol{\ball{0}{\varrho_N}}\log p-N\lambda_{\ball{0}{\varrho_N}}+c\varrho_N^{d-1}(\log N)^{a+1}\right\}\\
&\quad= \exp\left\{-c(d,\log(1/p))N^{\frac{d}{d+2}}+cN^{\frac{d-1}{d+2}}(\log N)^{a+1}\right\}.
\end{split}
\end{equation}
\end{remark}

%%%%%%%%%%%%%%%%%%%%%%%%%%%%%%%%%%%%%%%%%%%%%%%%%%%%%%%%%%%%%%%%%%%%%%%%%%%%%%
\appendix
\section{Estimates for eigenvalues and eigenfunctions}
\label{app:eigen}
In this section, we collect some estimates on eigenvalues and eigenfunctions, including Lemma~\ref{lem:disc-cont}, and then prove~\eqref{eq:ev},~\eqref{eq:ef} (used in the proof of Lemma~\ref{lem:end-inside}) and Lemma~\ref{lem:surv-asymp}. Recall that $\lambda^{(k)}_{T}$ and $\phi^{(k)}_{T}$ are the $k$-th smallest Dirichlet eigenvalue and corresponding eigenfunction with $\|\phi^{(k)}_{T}\|_2=1$ for $-\frac{1}{2d}\Delta$ in $T\subset\R^d$ and $\lambda^{{\rm RW},(k)}_{T}$ and $\phi^{{\rm RW},(k)}_{T}$ are their discrete space counterparts. 

We begin with the following comparison lemma which includes Lemma~\ref{lem:disc-cont}.
\begin{lemma}
\label{lem:a.1}
For any $d\ge 1$ and $k\ge 1$, there exists $\gamma_k>0$ such that for all sufficiently large $R>0$, the following bounds hold:
\begin{align}
\left|\lambda^{{\rm RW},(k)}_{\ball{0}{R}}-\lambda^{(k)}_{\ball{0}{R}}\right|
&\le \gamma_k R^{-3},
\label{eq:disc-cont-ball}\\
{\min_{|y|\le 1}}\phi_{\ball{0}{R}}^{{\rm RW},(1)}(y)&\ge \gamma_1^{-1} R^{-d/2},
\label{eq:EFlbd}
\end{align}
and for any $x\in\Z^d$ and $T\supset\ball{x}{\varrho_N/2}$, 
\begin{equation}
 \lambda^{{\rm RW},(1)}_T\ge\lambda_{\tilde{T}}^{(1)}-{\gamma_1} \varrho_N^{-4}.
%\label{eq:disc>cont}
\end{equation}
\end{lemma}
\noindent
\textbf{Proof of Lemma~\ref{lem:a.1}.}
The first assertion can be found in~\cite[(3.27) and (6.11)]{Weinberger1958}. The second assertion follows from~\cite[Lemma~2.1(b)]{B94}, which states that
\begin{equation}
{\sup_{y\in \ball{0}{R}}}\left|\phi_{\ball{0}{R}}^{(1)}(y)-\phi_{\ball{0}{R}}^{{\rm RW},(1)}(y)\right|
\le cR^{-d/2-1},
\end{equation}
and the fact that ${\min_{y\in\Z^d:|y|\le 1}}\phi_{\ball{0}{R}}^{(1)}(y)\ge c^{-1}R^{-d/2}$. 
The third assertion follows from~\cite[(6.9)]{Weinberger1958}, which states that
\begin{equation}
\lambda_{\tilde{T}}^{(1)}\le \lambda^{{\rm RW},(1)}_T(1+c\varrho_N^{-2}),
\end{equation}
and the bound $\lambda^{{\rm RW},(1)}_T\le c\varrho_N^{-2}$ that follows from the assumption $T\supset\ball{x}{\varrho_N/2}$. 
\qed 

Next we restate and prove~\eqref{eq:ev} %,~\eqref{eq:ev2} 
and~\eqref{eq:ef}.  
\begin{lemma}
\label{lem:specest}
There exist $c_7,c_8>0$ such that if $\ball{x}{(1-\eps)\varrho_N} \subset\Bi \subset\ball{x}{\varrho_N+\varrho_N^{\epsilon_1}}$ for some $x\in\Z^d$ and $\epsilon>0$ sufficiently small depending only on the dimension $d$, then the following bounds hold: 
\begin{align}
%\lambda^{{\rm RW},(1)}_{\Bi}&\le C\varrho_N^{-2},
%\tag{EV1}\label{eq:ev1}\\
\lambda^{{\rm RW},(2)}_{\Bi}-\lambda^{{\rm RW},(1)}_{\Bi}
&\ge c_7\varrho_N^{-2},
\tag{EV}\\
\left\|\phi^{{\rm RW},(1)}_{\Bi}\right\|_\infty
&\le c_8\varrho_N^{-d/2}.
%&\le C\left(\lambda^{{\rm RW},(k)}_{\Bi}\right)^{d/4}.
\tag{EF}%\label{eq:ef}
\end{align}
\end{lemma}
\noindent\textbf{Proof of Lemma~\ref{lem:specest}.}
In order to show the first assertion~\eqref{eq:ev}, recall first that the continuum eigenvalue satisfies the scaling relation $\lambda^{(k)}_{\ball{0}{R}}=R^{-2}\lambda^{(k)}_{\ball{0}{1}}$. Combining this with~\eqref{eq:disc-cont-ball}, we get the following bounds:
\begin{equation}
\begin{split}
 \lambda^{{\rm RW},(1)}_{\Bi}&\le \lambda^{{\rm RW},(1)}_{\ball{0}{\varrho_N+\varrho_N^{\epsilon_1}}}\\
&\le \lambda^{(1)}_{\ball{0}{\varrho_N+\varrho_N^{\epsilon_1}}}+{\gamma_1}\varrho_N^{-3}\\
 &\le \varrho_N^{-2}\lambda^{(1)}_{\ball{0}{1}}+c\varrho_N^{-3+\epsilon_1},\\
\end{split}
\label{eq:EVubd}
\end{equation}
%which implies~\eqref{eq:ev1}, 
and
\begin{equation}
\begin{split}
 \lambda^{{\rm RW},(2)}_{\Bi}
&\ge \lambda^{{\rm RW}, (2)}_{\ball{0}{(1-\epsilon)\varrho_N}}\\
&\ge \lambda^{(2)}_{\ball{0}{(1-\epsilon)\varrho_N}}-{\gamma_2}\varrho_N^{-3}\\
 &\ge \varrho_N^{-2}\lambda^{(2)}_{\ball{0}{1}}-c\epsilon\varrho_N^{-2}.
\end{split}
\end{equation}
Since $\lambda^{(1)}_{\ball{0}{1}}<\lambda^{(2)}_{\ball{0}{1}}$, the desired bound~\eqref{eq:ev} follows for sufficiently small $\epsilon>0$. 

Next, we show the second assertion~\eqref{eq:ef}. By the eigenvalue equation for the semigroup, it follows for any $x\in\Bi$ that
\begin{equation}
\begin{split}
\phi^{{\rm RW},(1)}_{\Bi}(x)
&=\left(1-\lambda^{{\rm RW},(1)}_{\Bi}\right)^{-\lfloor1/\lambda^{{\rm RW},(1)}_{\Bi}\rfloor}\sum_{y\in\Bi}p^\Bi_{\lfloor 1/\lambda^{{\rm RW},(1)}_{\Bi}\rfloor}(x,y)\phi^{{\rm RW},(1)}_{\Bi}(y)\\
&\le \left(1-\lambda^{{\rm RW},(1)}_{\Bi}\right)^{-\lfloor 1/\lambda^{{\rm RW},(1)}_{\Bi}\rfloor}\left\|p^\Bi_{\lfloor 1/\lambda^{{\rm RW},(1)}_{\Bi}\rfloor}(x,\cdot)\right\|_2
\left\|\phi^{{\rm RW},(1)}_{\Bi}\right\|_2,
%\le \left(1-\lambda^{{\rm RW},(k)}_{\Bi}\right)^{-1/\lambda^{{\rm RW},(k)}_{\Bi}}\left(\sum_{y\in\Bi}p^\Bi_{1/\lambda^{{\rm RW},(k)}_{\Bi}}(x,y)^2\right)^{1/2}\left(\sum_{y\in\Bi}\phi^{{\rm RW},(k)}_{\Bi}(y)^2\right)^{1/2},
\label{eq:Nash}
\end{split}
\end{equation}
where we have used the Schwarz inequality in the second line. Then by the symmetry of the transition kernel $p^\Bi$, the Chapman--Kolmogorov identity and the normalization $\|\phi_\Bi^{(1)}\|_2=1$, we can further rewrite~\eqref{eq:Nash} as
\begin{equation}
\begin{split}
|\phi^{{\rm RW},(1)}_{\Bi}(x)|
 &\le\left(1-\lambda^{{\rm RW},(1)}_{\Bi}\right)^{-\lfloor 1/\lambda^{{\rm RW},(1)}_{\Bi}\rfloor}p^\Bi_{2\lfloor 1/\lambda^{{\rm RW},(1)}_{\Bi}\rfloor}(x,x)^{1/2}\\
 &\le c\left(\lambda^{{\rm RW},(1)}_{\Bi}\right)^{d/4},
\end{split}
\end{equation}
where we used the local limit theorem as an upper bound.
Since we have $\lambda^{{\rm RW},(1)}_{\Bi}\ge c\varrho_N^{-2}$ similarly to~\eqref{eq:EVubd}, the last line is bounded by $c\varrho_N^{-d/2}$.
\qed

In order to prove Lemma~\ref{lem:surv-asymp}, we use the eigenfunction expansion for the semigroup generated by the random walk killed upon exiting $\ball{0}{\varrho_N}$, whose generator we denote by $Q_{\varrho_N}$. Due to the periodicity of the random walk, it is convenient to consider the semigroup generated by $Q_{\varrho_N}^2$. Let $\Z^d_{\rm e}$ ($\Z^d_{\rm o}$) and $1_{\rm e}$ ($1_{\rm o}$) denote the set of even (odd) sites and its indicator function, respectively. 
\begin{lemma}
\label{lem:parity}
For any positive integer $k\le |\ball{0}{\varrho_N}|/2$, the $k$-th largest eigenvalues of $Q_{\varrho_N}^2$ is $(1-\lambda_{\ball{0}{\varrho_N}}^{{\rm RW},(k)})^2$. Moreover the eigenspace corresponding to $(1-\lambda_{\ball{0}{\varrho_N}}^{{\rm RW},(1)})^2$ is spanned by $\phi_{\ball{0}{\varrho_N}}^{{\rm RW},(1)}1_{\rm e}$ and $\phi_{\ball{0}{\varrho_N}}^{{\rm RW},(1)}1_{\rm o}$. 
\end{lemma}
\begin{proof}
For any eigenvalue $\zeta$ and corresponding eigenfunction $\phi_\zeta$ of $Q_{\varrho_N}$, we have 
\begin{equation}
Q_{\varrho_N}1_{\rm e} \phi_\zeta
=\zeta 1_{\rm o}\phi_\zeta
\end{equation}
and the same holds with $1_{\rm e}$ and $1_{\rm o}$ interchanged. It follows that 
\begin{equation}
Q_{\varrho_N}(1_{\rm e} \phi_\zeta-1_{\rm o} \phi_\zeta)
=-\zeta (1_{\rm e} \phi_\zeta-1_{\rm o} \phi_\zeta),
\end{equation}
that is, $-\zeta$ is also an eigenvalue. Since $Q_{\varrho_N}$ has $|\ball{0}{\varrho_N}|$ eigenvalues counting multiplicity, the first assertion about the eigenvalues follows. 

The second assertion about the eigenfunction is a consequence of the following two facts: $\lambda_{\ball{0}{\varrho_N}}^{{\rm RW},(1)}$ is a simple eigenvalue and $Q_{\varrho_N}^2$ leaves $\ell^2(\Z^d_{\rm e})$ and $\ell^2(\Z^d_{\rm o})$ invariant.
\end{proof}
\medskip
\noindent\textbf{Proof of Lemma~\ref{lem:surv-asymp}.}
Let us start with the case that $N$ is an even integer. In this case, $S_N\in\Z^d_{\rm e}$ and hence
\begin{equation}
\begin{split}
&\P\otimes\bP\left(\tau_\Oi>N\right)\\
&\quad\ge \P(\Oi\cap\ball{0}{\varrho_N}=\emptyset)\bP\left(\tau_{\ball{0}{\varrho_N}^c}>N, S_N\in\Z^d_{\rm e}\right)\\
&\quad =p^{|\ball{0}{\varrho_N}|}Q_{\varrho_N}^{N/2}1_{\rm e}(0).
\end{split}
\label{eq:ball-strategy2}
\end{equation}
Let us denote by $P$ the orthogonal projection onto the first eigenspace of $Q_{\varrho_N}^2$. Then 
\begin{equation}
Q_{\varrho_N}^{N/2}1_{\rm e}(0)
=\left(1-\lambda_{\ball{0}{\varrho_N}}^{{\rm RW},(1)}\right)^NP1_{\rm e}(0)
+Q_{\varrho_N}^{N/2}({\rm id}-P)1_{\rm e}(0).
\label{eq:proj}
\end{equation}
Using $1-\lambda\ge \exp\{-\lambda-\lambda^2\}$ for small $\lambda>0$, $\phi^{{\rm RW},(1)}_{\ball{0}{\varrho_N}}\ge 0$,~\eqref{eq:EFlbd} and Lemma~\ref{lem:parity}, we can bound the first term below by
\begin{equation}
\begin{split}
 &\left(1-\lambda^{{\rm RW},(1)}_{\ball{0}{\varrho_N}}\right)^{N}
 \left\langle\phi^{{\rm RW},(1)}_{\ball{0}{\varrho_N}}, 1_{\rm e}\right\rangle \phi^{{\rm RW},(1)}_{\ball{0}{\varrho_N}}(0)\\
 &\quad \ge \left(1-\lambda^{{\rm RW},(1)}_{\ball{0}{\varrho_N}}\right)^N\phi^{{\rm RW},(1)}_{\ball{0}{\varrho_N}}(0)^2\\
 &\quad \ge \exp\left\{-N\lambda^{{\rm RW},(1)}_{\ball{0}{\varrho_N}}-cN\varrho_N^{-3}\right\}.
\end{split}
\label{eq:1stlbd2}
\end{equation}
On the other hand, it also follows from Lemma~\ref{lem:parity} that the operator norm of $Q_{\varrho_N}^{N} ({\rm id}-P)$ is bounded by $(1-\lambda_{\ball{0}{\varrho_N}}^{{\rm RW},(2)})^N$. Combining this with~\eqref{eq:ev} and $1-\lambda\le \exp\{-\lambda\}$, we can bound the second term in~\eqref{eq:proj} by
\begin{equation}
\begin{split}
\left|Q_{\varrho_N}^{N}({\rm id}-P)1_{\rm e}(0)\right|
& \le \left(1-\lambda_{\ball{0}{\varrho_N}}^{{\rm RW},(2)}\right)^N\|1_{\rm e}\|_{\ell^2(\ball{0}{\varrho_N})}\\
& \le \exp\left\{-N\lambda^{{\rm RW},(1)}_{\ball{0}{\varrho_N}}-cN\varrho_N^{-2}\right\}.
\end{split}
\label{eq:2ndubd2}
\end{equation}
Since~\eqref{eq:2ndubd2} is negligible compared with~\eqref{eq:1stlbd2} for large $N$, we obtain
\begin{equation}
\P\otimes\bP\left(\tau_\Oi>N\right)
\ge \exp\left\{|\ball{0}{\varrho_N}|\log(1/p)-N\lambda^{{\rm RW},(1)}_{\ball{0}{\varrho_N}}-cN\varrho_N^{-3}\right\}.
\end{equation}
Substituting the bound $\left||\ball{0}{\varrho_N}|-\vol{\ball{0}{\varrho_N}}\right|\le c\varrho_N^{d-1}$ and~\eqref{eq:disc-cont-ball} into the above, we arrive at the desired bound. 

Finally when $N$ is an odd integer, we start with
\begin{equation}
\begin{split}
& \P(\Oi\cap\ball{0}{\varrho_N}=\emptyset)\bP\left(\tau_{\ball{0}{\varrho_N}^c}>N\right)\\
&\quad =\sum_{|y|=1}\P(\Oi\cap\ball{0}{\varrho_N}=\emptyset)\bP_y\left(\tau_{\ball{0}{\varrho_N}^c}>N-1\right).
\end{split}
\end{equation}
Then the rest of the argument is the same as before. We have the sum of $\frac1{2d}\phi^{{\rm RW},(1)}_{\ball{0}{\varrho_N}}(y)$ over $\{|y|=1\}$ instead of $\phi^{{\rm RW},(1)}_{\ball{0}{\varrho_N}}(0)$ in~\eqref{eq:1stlbd2}, but \eqref{eq:EFlbd} gives us the same lower bound. \qed

\section{On the proof of Theorem~\ref{th:confine}}
\label{app:Povel}
In this section, we briefly explain how to prove Theorem~\ref{th:confine} by adapting the argument in~\cite{P99}. Roughly speaking, it consists of the following five steps: 
\begin{enumerate}
 \item use the method of enlargement of obstacles in~\cite{S98} to define a \emph{clearing set} $\mathscr{U}_{\rm cl}$, 
 \item prove volume and eigenvalue constraints for $\mathscr{U}_{\rm cl}$, 
 \item apply a quantitative Faber--Krahn inequality to show that $\mathscr{U}_{\rm cl}$ is close to a ball with radius $\varrho_N$,
 \item  prove a sharp lower bound on the partition function in terms of a random eigenvalue, 
 \item use the fact that the region outside $\mathscr{U}_{\rm cl}$ has much larger eigenvalue to deduce that it is too costly for the random walk to get away from the ball. 
\end{enumerate}
In what follows, we explain these steps in more detail using the same parameters as in~\cite{P99} as much as possible. It turns out that it is only Step 3 that requires an extra argument in the discrete setting. 

Steps 1 and 2 can be carried out exactly as in~\cite{P99} since the method of enlargement of obstacles used in that paper has been translated to the discrete setting in~\cite{BAR00}. This method allows us to define the clearing set $\mathscr{U}_{\rm cl}$ as a union of large  boxes (lattice animal) that are almost free of obstacles (cf.~\cite[(51)--(52)]{P99}). Since $\mathscr{U}_{\rm cl}^c$ has rather high density of obstacles, we can effectively discard it when we consider the eigenvalue  (cf.~\cite[(23), (55)]{P99}): 
\begin{equation}
\left|\lambda^{{\rm RW}, (1)}_{\mathscr{U}_{\rm cl}}-\lambda^{{\rm RW}, (1)}_{\ball{0}{2N}\setminus\Oi}\right|\le \varrho_N^{-\rho}
\end{equation}
for some $\rho>0$. Combining the above two properties,  in the same way as~\cite[Proposition~1]{P99}, we can prove that for some $\alpha_1>0$, the $\mu_N$-probability of 
\begin{equation}
|\mathscr{U}_{\rm cl}|\log \tfrac{1}{p}+N\lambda^{{\rm RW},(1)}_{\mathscr{U}_{\rm cl}}\le N^{\frac{d}{d+2}}\left(c(d,p)+N^{-\frac{\alpha_1}{d+2}}\right)\label{eq:MEO_var}
\end{equation}
tends to one as $N\to \infty$.

As for Step 3, if $\mathscr{U}_{\rm cl}$ were a subset of $\R^d$, then the quantitative Faber--Krahn inequality~\cite[Theorem~A]{P99} would imply that $\mathscr{U}_{\rm cl}$ satisfying~\eqref{eq:MEO_var} must be close to a ball with radius $\varrho_N$ in the symmetric difference. But we are in the discrete setting and hence we need to find a set in $\R^d$ whose volume and (continuum) eigenvalue are almost the same as $|\mathscr{U}_{\rm cl}|$ and $\lambda^{{\rm RW},(1)}_{\mathscr{U}_{\rm cl}}$, respectively. By the results in~\cite{Weinberger1958}, the set $\mathscr{U}_{\rm cl}^+=\left\{x\in\R^d\colon {\rm dist}_{\ell^\infty}(x,\mathscr{U}_{\rm cl})\le 2\right\}$
has the eigenvalue $\lambda^{(1)}_{\mathscr{U}_{\rm cl}^+}$ almost the same as $\lambda^{{\rm RW},(1)}_{\mathscr{U}_{\rm cl}}$. Also, since $\mathscr{U}_{\rm cl}$ is a union of large boxes, the volume of $\mathscr{U}_{\rm cl}^+$ is close to $|\mathscr{U}_{\rm cl}|$. In this way, we can conclude that there exists a ball $\ball{\mathcal{x}_N}{\varrho_N}$ that almost coincides with  $\mathscr{U}_{\rm cl}$. The outside of this ball has rather high density of obstacles and hence just as in~\cite[Lemma~1]{P99}, we have
\begin{equation}
\lambda^{{\rm RW}, (1)}_{\ball{0}{2N}\setminus\ball{\mathcal{x}_N}{\varrho_N}}
\ge \varrho_N^{-2+\alpha_3}
\end{equation}
for some $\alpha_3>0$. 

Step 4 relies on a simple functional analytic argument and there is no difficulty in adapting it to the discrete setting. It corresponds to~\cite[Lemma~2]{P99} and provides the following lower bound on the partition function:
\begin{equation}
 \P\otimes\bP\left(\tau_\Oi>N\right) \ge N^{-c}\E\left[\exp\left\{-N\lambda^{{\rm RW}, (1)}_{\ball{0}{2N}\setminus\Oi}\right\}\right].
\end{equation}
See also~\cite[Lemma~2]{F08} for a slightly simplified argument. 

Finally, Step 5 roughly goes as follows. For $0\le k< l\le N$, consider the event that $S_k$ is away from the confinement ball $\ball{\mathcal{x}_N}{\varrho_N}$ and returns to it at time $l$ for the first time after $k$. In this situation, the survival probability between $k$ and $l$ decays like $\exp\{-(l-k)\lambda^{{\rm RW}, (1)}_{\ball{0}{2N}\setminus\ball{\mathcal{x}_N}{\varrho_N}}\}$ and hence using Steps~3 and~4, we obtain 
\begin{equation}
\begin{split}
&\mu_N\left(S_k\not\in\ball{\mathcal{x}_N}{\varrho_N+\varrho_N^{\epsilon_1}}, k+\tau_{\ball{\mathcal{x}_N}{\varrho_N}}\circ\theta_k=l\right)\\
&\quad \lesssim N^c \frac{\E\left[\exp\left\{-(N-l+k)\lambda^{{\rm RW}, (1)}_{\ball{0}{2N}\setminus\Oi}+(l-k)\varrho_N^{-2+\alpha_3}\right\}\right]}{\E\left[\exp\left\{-N\lambda^{{\rm RW}, (1)}_{\ball{0}{2N}\setminus\Oi}\right\}\right]}.
\end{split}
\label{eq:step5}
\end{equation}
Note that we may impose $\lambda^{{\rm RW}, (1)}_{\ball{0}{2N}\setminus\Oi}\le 2c(d,p)\varrho_N^{-2}$ both in the numerator and denominator, in view of~\eqref{eq:surv}. 
Now if $l-k\ge \varrho_N^{-2+2\alpha_3}$, then the term $(l-k)\varrho_N^{-2+\alpha_3}$ in the numerator causes a large additional cost and hence the right-hand side decays stretched exponentially in $N$. If $l-k< \varrho_N^{-2+2\alpha_3}$, then we choose $\epsilon_1>1-\alpha_3$ (so that $|S_l-S_k|\gg (l-k)^{1/2}$) and use the Gaussian heat kernel bound for the random walk, instead of the eigenvalue bound, to derive a stretched exponential decay. Summing over $k$ and $l$, we conclude that the random walk does not make a crossing from $\ball{\mathcal{x}_N}{\varrho_N+\varrho_N^{\epsilon_1}}^c$ to $\ball{\mathcal{x}_N}{\varrho_N}$. The case that the random walk does not return to $\ball{\mathcal{x}_N}{\varrho_N}$ after visiting $\ball{\mathcal{x}_N}{\varrho_N+\varrho_N^{\epsilon_1}}^c$ but this can be dealt with in a similar way, by changing $l$ to the last visit to $\ball{\mathcal{x}_N}{\varrho_N}$ before $k$. 

\section{Index of notation}
\label{app:notation}
\begin{multicols}{2}
{\renewcommand\arraystretch{1}
\begin{tabular}[t]{l@{\hspace{60pt}}l}
 $\tau_A$ &\eqref{tauO}\\
 $\mu_N$&\eqref{eq:def_mu}\\
 $c(d,p)$ &\eqref{eq:surv} \\
 $\varrho_N$ &\eqref{eq:Rn}\\
 $\mathcal{x}_N$ &Theorem~\ref{th:confine}\\
 $\ball{x}{R}$ &Theorem~\ref{th:confine}\\
 $\partial A$ &\eqref{eq:ex_boundary}\\
 $|A|, {\rm vol}(A)$& End of Section~\ref{sec:intro} \\
 $E_l^\delta$& \eqref{eq:E_l^delta} \\
 $A(x;r,R)$& \eqref{eq:6TinA} \\
 $\Ti$& Definition~\ref{D:trulyopen}
\end{tabular}
}
\columnbreak
\begin{wraptable}{l}{0pt}
{\renewcommand\arraystretch{1.2}
\begin{tabular}[t]{l@{\hspace{35pt}}l}
 $\cal{X}_l$& \eqref{X-separation} and~\eqref{X-covering}\\
 $p^D_n(u,v)$& \eqref{eq:trans_prob} \\
 $\Gamma(k)$& \eqref{eq:Gamma} \\
 $E_l^{\delta,\rho}$& \eqref{eq:Edeltarho} \\
 $\cal{X}_l^{\circ}$& \eqref{eq:calXo} \\
 $G_{\Oi}$& \eqref{eq:Green} \\
 $\Li(l$)& \eqref{eq:defL} \\
 $\lambda^{{\rm RW},(k)}_{A}$ and $\phi^{{\rm RW},(k)}_{A}$& \eqref{eq:EFexpansion} \\
 $\lambda_{A}^{(k)}$ and $\phi_{A}^{(k)}$& Below~\eqref{eq:cont-hull} 
\end{tabular}
}
\end{wraptable}

\end{multicols}
\bigskip
\noindent
{\bf Acknowledgements}
This paper is dedicated to the memory of Kazumasa Kuwada, who passed away in December 2018. He and R.~Fukushima organized RIMS camp style seminar ``New Trends in Stochastic Analysis'' in 2015, where J.~Ding and R.~Fukushima first discussed the model studied in this paper. 
J.~Ding is supported by NSF grant DMS-1757479 and an Alfred Sloan fellowship.
R.~Fukushima is supported by JSPS KAKENHI Grant Number 16K05200 and ISHIZUE 2019 of Kyoto University Research Development Program. 
R.~Sun is supported by NUS Tier 1 grant R-146-000-253-114.

\end{document}